\documentclass[reqno]{amsart}
\usepackage{amssymb,mathrsfs, color}
\allowdisplaybreaks

\usepackage{cite}

\usepackage{todonotes}

\usepackage{enumerate}
\usepackage{graphicx}
\usepackage{xcolor} 
\usepackage[colorlinks,linkcolor=blue,anchorcolor=red,citecolor=blue]{hyperref}
\usepackage{geometry}
\usepackage{amsfonts,amssymb,amsmath,mathtools}
\usepackage{slashed}
\usepackage{mathrsfs}
\usepackage{dsfont}
\usepackage[titletoc,title]{appendix}
\usepackage{wrapfig}

\usepackage{cancel}

\usepackage{cite}
\usepackage{framed}

\definecolor{green}{rgb}{0,0.8,0} 
\makeatletter
\renewcommand{\maketag@@@}[1]{\hbox{\m@th\normalsize\normalfont#1}}
\makeatother
\definecolor{deepgreen}{cmyk}{1,0,1,0.5}

\newcommand{\R}{\mathbb{R}}

\makeatletter

\newcommand{\Rmnum}[1]{\expandafter\@slowromancap\romannumeral #1@}
\makeatother

\newcommand{\Del}[1]{}

\numberwithin{equation}{section}

\newtheorem{theorem}{Theorem}[section]
\newtheorem{corollary}[theorem]{Corollary}
\newtheorem{lemma}[theorem]{Lemma}
\newtheorem{remark}[theorem]{Remark}
\newtheorem{definition}[theorem]{Definition}

\renewcommand{\div}{\mathrm{div}\,}

\makeatletter
\@namedef{subjclassname@2020}{%
	\textup{2020} Mathematics Subject Classification}
\makeatother

\begin{document}

\title[Steady compressible Navier-Stokes-Fourier system]{Steady compressible Navier-Stokes-Fourier system with slip boundary conditions arising from kinetic theory}

\author[R.-J. Duan]{Renjun Duan}
\address[RJD]{Department of Mathematics, The Chinese University of Hong Kong,
Shatin, Hong Kong, P.R.~China}
\email{rjduan@math.cuhk.edu.hk}

\author[J.-H. Zhang]{Junhao Zhang}
\address[JHZ]{Department of Mathematics, The Chinese University of Hong Kong, Shatin, Hong Kong, P.R.~China}
\email{jhzhang@math.cuhk.edu.hk}


\begin{abstract}
This paper studies the boundary value problem on the steady compressible Navier-Stokes-Fourier system in a channel domain $(0,1)\times\mathbb{T}^2$ with a class of generalized slip boundary conditions that were systematically derived from the Boltzmann equation by Coron \cite{Coron-JSP-1989} and later by Aoki et al \cite{Aoki-Baranger-Hattori-Kosuge-Martalo-Mathiaud-Mieussens-JSP-2017}.  We establish the existence and uniqueness of strong solutions in $(L_{0}^{2}\cap H^{2}(\Omega))\times V^{3}(\Omega)\times H^{3}(\Omega)$ provided that the wall temperature is near a positive constant.  The proof relies on the construction of a new variational formulation for the corresponding linearized problem and employs a fixed point argument. The main difficulty arises from the interplay of velocity and temperature derivatives together with the effect of density dependence on the boundary. 
\end{abstract}



\maketitle 


\thispagestyle{empty}
\tableofcontents


\section{Introduction}
Consider a bounded domain $\Omega\subset\mathbb{R}^{3}$. The steady motion of monatomic gas confined in $\Omega$ can be described by using the steady compressible Navier-Stokes-Fourier system. Specifically, the gas density $\rho>0$, the gas velocity $\mathbf{u}=(u_{1},u_{2},u_{3})$ and the gas temperature $\theta>0$ satisfy the following equations:
\begin{equation}\label{SNS}
\left\{
\begin{aligned}
&\div(\rho\mathbf{u}) =0, \\
&\div(\rho\mathbf{u}\otimes\mathbf{u})+\nabla p=\div\mathbb{S}, \\
&\div\left[(\rho\mathscr{E}+p)\mathbf{u}\right] =\div\mathbf{q}+\div(\mathbb{S}\mathbf{u}), 
\end{aligned}
\right.
\end{equation}
for $\mathbf{x}=(x_{1},x_{2},x_{3})\in\Omega$. Here, $\mathscr{E}=\frac{1}{2}\rho|\mathbf{u}|^{2}+\rho e$ is the energy density, and $\mathbb{S}$, $\mathbf{q}$ are the stress tensor and the heat flux vector, respectively, given by
\begin{equation*}
  \mathbb{S}(\mathbf{u},\theta)=\mu(\theta)\left(\nabla\mathbf{u}+\nabla\mathbf{u}^{\mathsf{T}}-\frac{2}{3}\div\mathbf{u}\,\mathbb{I}_{3}\right),\quad \mathbf{q}(\theta)=\kappa(\theta)\nabla\theta,
\end{equation*}
where the coefficient of viscosity $\mu$ and the coefficient of heat conduction $\kappa$ are assumed to be positive smooth functions of $\theta>0$, and $\mathbb{I}_{3}$ is the $3\times3$ unit matrix. The physical properties of a gas are reflected through constitutive equations relating the state variables to the pressure $p$ and the internal energy density $e$. In this paper, we restrict considerations to the case of perfect polytropic gases with the pressure and the internal energy density defined by the formula
\begin{equation*}
p=R\rho\theta,\quad e=c_{v}\theta=\frac{R}{r-1}\theta,
\end{equation*}
where $R>0$ is a generic gas constant and $c_{v}=\frac{R}{r-1}$ is the specific heat with the adiabatic constant $r>1$.

We note that from the Chapman-Enskog expansion for the Boltzmann equation, $\mu$ and $\kappa$ can be represented as 
$$
\mu(\theta)=\mu_0\theta^{1-\frac{\gamma}{2}},\quad \kappa(\theta)=\kappa_0\theta^{1-\frac{\gamma}{2}},
$$
for constants $\mu_0>0$ and $\kappa_0>0$, where $-3<\gamma\leq 1$ denotes an inhomogeneity parameter in connection with the Boltzmann collision kernel. In addition, the experimental evidence, cf.~\cite{Becker-1966,Zeldovich-Raizer-1967}, has indicates that $\mu$ and $\kappa$ usually depend on both $\rho$ and $\theta$. In this paper, we do not pursue this general situation, but the method and result can be extended without difficulties.

To solve the steady compressible Navier-Stokes-Fourier system \eqref{SNS}, some ``correct'' boundary conditions to the bounded domain should be posed. In this aspect, various kinds of boundary conditions have been studied. Typically, there are two cases: the no-slip conditions, and Navier-slip conditions, which are characterized by the physical meaning of whether it allows the gases or fluids to slide along the boundary. Historically, Navier-slip boundary conditions were first proposed by NAVIER \cite{Naiver-MARSP-1823} that claims that at the boundary surface, the tangent component of the fluid velocity should be proportional to the rate of strain at the surface, while the normal component of the velocity is zero, since mass cannot penetrate an impermeable solid surface. Later, the Navier-slip boundary conditions have been theoretically derived from the Maxwell reflection boundary condition for Boltzmann equation, see MASMOUDI and SAINT-RAYMOND \cite{Masmoudi-Saint-CPAM-2003} for incompressible Stokes-Fourier system, as well as JIANG and MASMOUDI \cite{Jiang-Masmoudi-CPAM-2017} for incompressible Naiver-Stokes-Fourier system. 

Although it seems natural to consider the boundary value problems of compressible Navier-Stokes-Fourier system with Navier-slip boundary conditions, there has been a lack of literature providing a rigorous theoretical justification for a long time. CORON \cite{Coron-JSP-1989} firstly proposed the idea of obtaining a new type of slip boundary conditions for the Naiver-Stokes-Fourier system from the kinetic boundary condition and gave a precise analysis for a gas between two plates. Very recently, for general domains AOKI and his collaborators \cite{Aoki-Baranger-Hattori-Kosuge-Martalo-Mathiaud-Mieussens-JSP-2017, Hattori-Kosuge-Aoki-PRF-2018} re-visited the derivation of those slip boundary conditions that are given by
\begin{equation}\label{BC}
\left\{
\begin{aligned}
&\mathbf{u}\cdot \mathbf{n}=0, \\
&\rho \mathbf{u}\cdot \mathbf{t}+\sqrt{\frac{2}{R}}a_{\mathbf{u}}^{I}\frac{\mu(\theta_{w})}{\sqrt{\theta_{w}}}(\nabla\mathbf{u}+\nabla\mathbf{u}^{\mathsf{T}}):\mathbf{n}\otimes\mathbf{t}-\frac{4}{5R}a_{\theta}^{I}\frac{\kappa(\theta_{w})}{\theta_{w}}\nabla\theta\cdot\mathbf{t}=0, \\
&\rho(\theta-\theta_{w})-\frac{1}{R}a_{\mathbf{u}}^{II}\mu(\theta_{w})\nabla\mathbf{u}:\mathbf{n}\otimes\mathbf{n}+\frac{2}{5R}\sqrt{\frac{2}{R}}a_{\theta}^{II}\frac{\kappa(\theta_{w})}{\sqrt{\theta_{w}}}\nabla\theta\cdot\mathbf{n}=0,
\end{aligned}
\right.
\end{equation}
for $\mathbf{x}\in\partial\Omega$, where $\mathbf{n}$ is the unit outer normal vector, $\mathbf{t}$ is the unit tangent vector, $a_{\mathbf{u}}^{I}, a_{\mathbf{u}}^{II}, a_{\theta}^{I}, a_{\theta}^{II}>0$ are the slip coefficients, and $\theta_{w}$ is the wall temperature which is a positive function defined on the boundary $\partial\Omega$. Indeed, \eqref{BC} is derived from the kinetic view with the aim to reproduce the correct overall compressible viscid fluid approximation solutions to the corresponding boundary value problem on the Boltzmann equation with Maxwell reflection boundary condition in the hydrodynamic sense. It is clear that \eqref{BC} has a more solid theoretical basis when applied to the study of boundary value problems for compressible Navier-Stokes-Fourier system compared to the case of Navier-slip boundary conditions. 

In this paper, we aim to study existence of strong solutions to the boundary value problem \eqref{SNS} and \eqref{BC} for the specific finite channel domain $\Omega:=(0,1)\times\mathbb{T}^2$
provided that $\theta_{w}$ is near a positive constant. For the future, we expect to further study the viscous compressible fluid approximation to the Boltzmann equation in the current setting with such generalized slip boundary conditions \eqref{BC}; see a recent work \cite{Duan-Liu-Yang-Zhang-CAM-2022} by the first author and collaborators for the one-dimensional heat transfer problem.

The well-posedness of the boundary value problem for the steady compressible Navier-Stokes system has been studied extensively. In what follows we give a brief review of the related known results. There are two main types of results regarding the existence of the stationary solution to the compressible Navier-Stokes equations. One way is to construct a (large) solution in some weak sense. The first result of this kind is established by LIONS \cite{Lions-CRASP-1993-1, Lions-CRASP-1993-2} for the isentropic fluid, proving that there exists a weak solution $(\rho,\mathbf{u})\in L^{\gamma}(\Omega)\times H^{1}(\Omega)$ for large external data with restricted range of the adiabatic coefficient $\gamma$. As for the non-isentropic case, MUCHA and PORKORN\'{Y} \cite{Mucha-Pokorny-CMP-2009} firstly constructed a ``physical reasonable'' weak solution $(\rho,\mathbf{u},\theta)\in L^{\infty}(\Omega)\times W^{1,q}(\Omega)\times W^{1,q}(\Omega)$ for all $q\geq1$ where constitutive equation is assumed to contain elastic part and heat conductivity is assumed to depend on temperature. For more results concerning different boundary conditions, we refer to \cite{Pokorny-Mucha-DCDS-2008, Mucha-Pokorny-MMMAS-2010, Muzereau-Neustupa-Penel-AA-2011, Feireisl-Novotny-AIHP-2018, Consiglieri-SPJMS-2021, Pokorny-JMFM-2022}. The other way is to construct a (small) solution in some strong sense. The first result of this kind is obtained by MATSUMURA and NISHIDA \cite{Matsumura-Nishida-1981, Matsumura-Nishida-1982}. They found a solution in the classic sense where the external force field is supposed to be the gradient of a time-independent function. For the general situation, VALLI \cite{Valli-ASNSP-1983} proved the existence of a strong solution for isentropic fluid by proposing an argument called the stability method. Later, the same author \cite{Valli-AIHP-1987} gave another proof using the fixed point argument. For the non-isentropic case, the stability method has been applied by VALLI and ZAJACKOWSKI \cite{Valli-Zajaczkowski-CMP-1986}, and the fixed point method has been worked by BEIR\~{A}O DA VEIGA \cite{Beirao-CMP-1987}. For more results concerning different boundary conditions, we refer to \cite{Padula-ARMA-1987, Farwig-CPDE-1989, Plotnikov-Ruban-Sokolowski-JMPA-2009, Piasecki-JDE-2010, Nishida-Padula-Teramoto-JMFM-2013, Piasecki-Pokorny-ZAMM-2014, Guo-Jiang-Zhou-SIAM-2015}.

However, it should be pointed out that all results above were carried out for non-slip boundary conditions or Navier-slip boundary conditions. Comparing to these common boundary conditions, there are three significant features in slip boundary conditions \eqref{BC}: (1) It reveals the interplay between velocity derivatives $\nabla\mathbf{u}$ and temperature derivatives $\nabla\theta$ along the boundary; (2) It contains nonlinear terms $\rho\mathbf{u}$ and $\rho(\theta-\theta_{w})$ that depend on the density; (3) The slip coefficients $a_{\mathbf{u}}^{I}$, $a_{\mathbf{u}}^{II}$, $a_{\theta}^{I}$ and $a_{\theta}^{II}$ are specific values determined only by the physical law obeyed by the molecules of the fluid or the gas. The main motivation of this paper is to generalize the previous studies to the problem with generalized slip boundary conditions \eqref{BC}.

Because of technical issues with our method, we need to make the following assumption:
\begin{itemize}
  \item [(A)] Let $\overline{\rho}$ and $\overline{\theta}$ be positive constants such that the following holds true: 
      \begin{equation}\label{assumption1}
      \frac{16}{15R}a_{\theta}^{I}\frac{\kappa(\overline{\theta})}{\overline{\theta}}\overline{\rho}-\frac{2}{5R}a_{\mathbf{u}}^{II}a_{\theta}^{I}\frac{\kappa(\overline{\theta})}{\overline{\theta}}\overline{\rho}>0,
      \end{equation}
and moreover, there are two positive constants $\lambda_{0},\lambda_{1}>0$ such that the following matrix 
      \begin{equation}\label{assumption2}
        \setlength{\arraycolsep}{0.01\hsize}
        \begin{pmatrix}
         m_{11} & 0 & m_{13} & m_{14} & m_{15} & m_{16} \\
        0 & m_{22} & m_{23} & 0 & 0 & 0 \\
        m_{13} & m_{23} & m_{33} & 0 & 0 & 0 \\
        m_{14} & 0 & 0 & m_{44} & 0 & 0 \\
         m_{15} & 0 & 0 & 0 & m_{55} & m_{56} \\
        m_{16} & 0 & 0 & 0 & m_{56} & m_{66}
        \end{pmatrix}
      \end{equation}
      is also positive-definite,  where those entries are given by
\begin{align*}
&m_{11}=R\overline{\theta},\  
m_{13}=\frac{1}{2}R\overline{\rho}+\frac{\lambda_{0}}{2}R\frac{4}{5R^2}a_{\mathbf{u}}^{II}a_\theta^I\kappa(\overline{\theta}),\  
m_{14}=\frac{\lambda_{0}}{2}R\frac{4}{5R^2}a_{\mathbf{u}}^{II}a_\theta^I\kappa(\overline{\theta}),\\
&
m_{15}=\frac{1}{2}\mu(\overline{\theta}), m_{16}=\frac{\lambda_{1}}{2}R\frac{4}{5R^2}a_{\mathbf{u}}^{II}a_\theta^I\kappa(\overline{\theta}),\\
&
m_{22}=\lambda_{0}\frac{2}{3R}\sqrt{\frac{2}{R}}a_{\mathbf{u}}^{I}a_{\mathbf{u}}^{II}\frac{\mu(\overline{\theta})^{2}}{\sqrt{\overline{\theta}}},\  
m_{23}=\lambda_{0}\frac{2}{5R^2}a_{\mathbf{u}}^{II}a_{\theta}^I\frac{\kappa(\overline{\theta})}{\overline{\theta}},\  
m_{33}=\lambda_{0}\frac{32}{75R^2}\sqrt{\frac{2}{R}}a_{\theta}^{I}a_{\theta}^{II}\frac{\kappa(\overline{\theta})^{2}}{\overline{\theta}\sqrt{\overline{\theta}}},\\
&
m_{44}=\lambda_{0}\frac{16}{15R}a_{\theta}^{I}\frac{\kappa(\overline{\theta})}{\overline{\theta}}\overline{\rho},\\ 
&
m_{55}=\lambda_{1}\frac{2}{3R}\sqrt{\frac{2}{R}}a_{\mathbf{u}}^{I}a_{\mathbf{u}}^{II}\frac{\mu(\overline{\theta})^{2}}{\sqrt{\overline{\theta}}},\  
m_{56}=\lambda_{1}\frac{2}{5R^2}a_{\mathbf{u}}^{II}a_{\theta}^I\frac{\kappa(\overline{\theta})}{\overline{\theta}},\  
m_{66}=\lambda_{1}\frac{32}{75R^2}\sqrt{\frac{2}{R}}a_{\theta}^{I}a_{\theta}^{II}\frac{\kappa(\overline{\theta})^{2}}{\overline{\theta}\sqrt{\overline{\theta}}}.
\end{align*}
\end{itemize}

Our main result is stated as follows. Notations on norms will be introduced at the end of this section.

\begin{theorem}\label{mainTheorem}
Let $\Omega=(0,1)\times\mathbb{T}^2$ and the assumption (A) hold for the slip coefficients $a_{\mathbf{u}}^{I}, a_{\mathbf{u}}^{II}, a_{\theta}^{I}, a_{\theta}^{II}$ and positive constants $\overline{\rho}$ and $\overline{\theta}$. There exists a constant $\varepsilon_{0}>0$ such that for $\varepsilon\in[0,\varepsilon_{0})$, if $\theta_{w}\in H^{\vartheta}(\partial\Omega)$ with $\|\theta_{w}-\overline{\theta}\|_{H^{\vartheta}(\partial\Omega)}\leq \varepsilon$ where $\vartheta>\frac{7}{2}$, then the boundary value problem \eqref{SNS} with \eqref{BC} admits a unique solution $(\rho,\mathbf{u},\theta)\in H^{2}(\Omega)\times V^{3}(\Omega)\times H^{3}(\Omega)$ such that
$\int_{\Omega}\rho\,d\mathbf{x}=\overline{\rho}|\Omega|$ and
  \begin{equation}\label{mainEstimate}
    \|\rho-\overline{\rho}\|_{H^2(\Omega)}+\|\mathbf{u}\|_{H^3(\Omega)}+\|\theta-\overline{\theta}\|_{H^3(\Omega)}\leq C\varepsilon,
  \end{equation}
for a generic constant $C>0$ independent of $\varepsilon$. 
\end{theorem}

\begin{remark}
  The assumption (A) being made in Theorem \ref{mainTheorem} may seem overwhelming, but it actually reveals the connection between Navier-slip boundary conditions and generalized slip boundary conditions \eqref{BC}. To see this, let us consider the case that
  $$
  a_{u}^{I}\sim O(1),\  a_{\theta}^{I}\sim O(\epsilon),\  a_{\mathbf{u}}^{II}\sim O(\epsilon),\  a_{\theta}^{II}\sim O(1)
  $$ 
  for $\epsilon>0$ sufficiently small, where $O(1), O(\epsilon)$ suggest the relative size of the slip coefficients, then the assumption (A) can be verified directly. Moreover, for this particular case, the slip boundary conditions \eqref{BC} tend to Navier-slip boundary conditions formally as $\epsilon\rightarrow0$. In fact, CORON \cite{Coron-JSP-1989} has pointed out that $\nabla\theta\cdot\mathbf{t}$ in $\eqref{BC}_{2}$ and $\nabla\mathbf{u}:\mathbf{n}\otimes\mathbf{n}$ in $\eqref{BC}_{3}$ are small compared to other quantities in \eqref{BC}. Thus, if these small terms are neglected, then slip boundary conditions \eqref{BC} become the usual Navier-slip boundary conditions.
\end{remark}

Now let us list the main procedures and the key ideas in the proof of Theorem \ref{mainTheorem}. 

To construct the solution, we first study an approximated linear problem obtained by linearization of the nonlinear problem, see \eqref{LSNS} and \eqref{LBC}. The linear problem can be viewed as the combination of a transport equation for approximated fluid density and an elliptic system for approximated fluid velocity and temperature. The transport equation $\eqref{LSNS}_{1}$ is solved by using elliptic approximation, see \eqref{ellipticApproximationTransport} and \eqref{neumannBC}. Some uniform estimates for the approximation are essential, see \eqref{H1UniformEstimateEllipticApproximationTransport} in Lemma \ref{H1UniformEstimateEllipticApproximationTransportLemma}. The elliptic system $\eqref{LSNS}_{2,3}$ is solved by using the Lax-Milgram theorem. The important step is to design a variational form for the problem, see \eqref{weakFormEllipticSystem} in Definition \ref{definitionWeakSolutionEllipticSystem}. This special construction of variation form reduces the strong coupling boundary effect between $\nabla\mathbf{u}$ and $\nabla\theta$. Although it is then possible to proceed with a basic energy estimate, there is an extra bad term: 
\begin{equation}\label{badTerm1}
    \int_{\Omega}-\frac{4}{5R^2}a_{\mathbf{u}}^{II}a_{\theta}^I\frac{\mu(\theta_{w})\kappa(\theta_{w})}{\mu(\overline{\theta})\theta_{w}}\nabla\zeta\cdot\mathfrak{n}\left(\nabla u_{1}+\partial_{1}\mathbf{u}-\frac{2}{3}\mathbf{e}_{1}\div\mathbf{u}\right)d\mathbf{x},
\end{equation}
which is a quadratic form of $\nabla\mathbf{u}$ and $\nabla\zeta$, where $\mathfrak{n}(\mathbf{x}):=2x_{1}-1$. It appears because of our approach to deal with the slip boundary conditions \eqref{BC}, which is not needed if only common boundary conditions are considered, see \cite{Matsumura-Nishida-1981, Matsumura-Nishida-1982, Valli-ASNSP-1983, Valli-AIHP-1987, Valli-Zajaczkowski-CMP-1986, Beirao-CMP-1987, Farwig-CPDE-1989, Plotnikov-Ruban-Sokolowski-JMPA-2009, Piasecki-JDE-2010, Nishida-Padula-Teramoto-JMFM-2013, Piasecki-Pokorny-ZAMM-2014}. Fortunately, $\|\mathbf{u}\|_{H^1(\Omega)}$ and $\|\theta\|_{H^{1}(\Omega)}$ are also obtained from the basic estimate, which implies there is a quadratic form in the estimate. Therefore our basic idea is to control this bad term by making full use of the quadratic form. This is the partial reason for the assumptions in our result, see the assumption (A) and the main Theorem \ref{mainTheorem}. We will explain further after the proof of Theorem \ref{basicEstimateEllipticSystem}, see Remark \ref{remarkBoundaryEffect}. 

As a second step, we establish a series of closed estimates to apply the Schauder fixed point theorem. The situation becomes more complicated than what we have done to the linear cases. Taking the basic estimates as an example, we have to deal with the following bad term:
\begin{equation}\label{badTerm2}
    \int_{\Omega}\frac{4}{5R}\overline{\theta}a_{\mathbf{u}}^{II}a_{\theta}^I\frac{\mu(\theta_{w})\kappa(\theta_{w})}{\mu(\overline{\theta})\theta_{w}}\zeta \mathfrak{n}\partial_{1}\varphi' d\mathbf{x},
\end{equation}
which is a quadratic term involving $\zeta$ and $\partial_{1}\varphi'$. Note that \eqref{badTerm2} can be easily controlled for the results concerning only the linear system since $(\varphi',\mathbf{u}',\zeta')$ is viewed as the source term in the partial differential system. However, it is necessary to close the estimate for this term here. It appears for the same reason by the slip boundary conditions \eqref{BC}. Inspired by the control of \eqref{badTerm1}, since it is possible to obtain the estimates for $\|\zeta\|_{L^{2}(\Omega)}$ and $\|\partial_{1}\varphi\|_{L^{2}(\Omega)}$, we are able to control it by using another quadratic form, which is the further reason for us to make the assumption (A). To obtain the estimate $\|\partial_{1}\varphi\|_{L^{2}(\Omega)}$, the key idea is to consider 
$$
\int_{\Omega}\partial_{1}\varphi\cdot\left(\frac{4}{3}\frac{\mu(\overline{\theta})}{\overline{\rho}}\eqref{LSNS}_{1}+\mathbf{e}_{1}\cdot\eqref{LSNS}_{2}\right)d\mathbf{x}.
$$
Then we can control $\|\partial_{1}\varphi\|_{L^{2}(\Omega)}$ by $\|(\partial_{2}\mathbf{u},\partial_{3}\mathbf{u})\|_{H^{1}(\Omega)}$and $\|\zeta\|_{H^{1}(\Omega)}$, which can be obtained directly by the basic (co-normal) energy estimate. Meanwhile, due to the nonlinearity of $\rho\mathbf{u}$ and $\rho(\theta-\theta_{w})$ in the boundary conditions, we need a more careful analysis in basic (co-normal) energy estimates. Precisely speaking, by the Sobolev embedding, these terms indicate $H^{\frac{1}{2}}(\Omega)$ regularity of the solution. Fortunately, we can control them by performing detailed analysis with the help of Fourier transform. For higher regularity estimates of $\mathbf{u}$ and $\theta$, as indicated by the estimates of Stokes system and Poisson equation, it is essential to estimate $\nabla\div\mathbf{u}$, which can be done by considering $$\int_{\Omega}\partial_{1}\div\mathbf{u}\cdot(\mathbf{e}_{1}\cdot\eqref{LSNS}_{2})d\mathbf{x}.$$

As a byproduct of our approach, we also obtain (1) Korn's type inequality for a bounded domain with tangency boundary condition; and (2) Poincar\'{e}'s inequality with traces for bounded domains. Although they may be simple cases of the results of others, we concisely establish them in the current work. It turns out that the specific computation we have performed makes our results suitable for practical use.

This paper is organized as follows. In section $2$, we establish the existence and the regularity results of the linearized problem. In sections $3$ and $4$, we study some properties of the solution operator of the linearized problem. In section $5$, we prove the existence of the nonlinear problem by applying the Schauder fixed point theorem. In the appendices, we give the proof of a Korn's type inequality and Poincar\'{e} type inequality with traces, and in addition, some well-known results which are used in the paper will be introduced for the reader's convenience.

\medskip
\noindent \textit{Notations.} Throughout the paper, we write $\mathbb{T}^{2}_{p}:=\{p\}\times\mathbb{T}^2$ for $p=0,1$, and $\partial_{i}:=\partial_{x_{i}}$ for $i=1,2,3$. In addition, the positive generic constants are denoted by $C$ or $C_{i}$ with index $i\in\mathbb{N}$. If necessary, we will point out the dependency between $C, C_{i}$ and $a_{\mathbf{u}}^{I}, a_{\mathbf{u}}^{II}, a_{\theta}^{I}, a_{\theta}^{II}, \mu, \kappa, \overline{\rho}, \overline{\theta}$, etc. For function spaces, $H^{k}(\Omega)$ denotes the Sobolev spaces with the standard norm denoted by $\|\cdot\|_{H^{k}(\Omega)}$. Although it may be an abuse of notations, we will use $H^{k}(\Omega)$ (or $L^{2}(\Omega)$) for both real-valued functions and vector-valued functions. $L_{0}^{2}(\Omega)$ denotes the space of functions $f\in L^{2}(\Omega)$ such that $\int_{\Omega}f(\mathbf{x})d\mathbf{x}=0$. $V^{k}(\Omega)$ denotes the space of 3-dimensional vector valued functions $f\in H^{k}(\Omega)$ whose normal component is zero on the boundary $\partial\Omega$.  

\section{Well-posedness of linearized problem and solution operator $\mathds{T}$}
Let $\overline{\rho}>0$ and $\overline{\theta}>0$ be given. We always look for solutions such that the total mass of the fluid $\int_{\Omega}\rho(x)d\mathbf{x}=\overline{\rho}|\Omega|$. Moreover, for the wall temperature $\theta_{w}\in H^{\vartheta}(\partial\Omega)$ where $\vartheta>\frac{7}{2}$, we assume its variation near $\overline{\theta}$ is $\delta>0$, which means that $\|\theta_{w}-\overline{\theta}\|_{H^{\vartheta}(\partial\Omega)}\leq \delta$. By the Sobolev trace theorem, there exists a function $\tilde{\theta}\in H^{\vartheta+\frac{1}{2}}(\Omega)$ such that $\tilde{\theta}\equiv\theta_{w}$ on the boundary and
$\|\tilde{\theta}-\overline{\theta}\|_{H^{\vartheta+\frac{1}{2}}(\Omega)}\leq  C\delta$ where $C>0$ is a generic constant.

Denote 
\begin{equation}\label{perturbation}
\varphi=\rho-\overline{\rho},\quad\zeta=\theta-\tilde{\theta}.
\end{equation}
After plugging \eqref{perturbation} into \eqref{SNS}, we have the boundary value problem for $(\varphi,\mathbf{u},\zeta)$:
\begin{equation}\label{PSNS}
\left\{
\begin{aligned}
    &\div(\varphi\mathbf{u})+\overline{\rho}\div\mathbf{u}=0, \\
    &-\div\mathbb{S}(\mathbf{u},\overline{\theta})+R\overline{\theta}\nabla\varphi+R\overline{\rho}\nabla\zeta =\mathbf{F}(\varphi,\mathbf{u},\zeta), \\
    &-\div(\kappa(\overline{\theta})\nabla\zeta)+R\overline{\rho}\overline{\theta}\div\mathbf{u} =G(\varphi,\mathbf{u},\zeta), 
\end{aligned}
\right.
\end{equation}
for $x\in\Omega$, with the boundary conditions
\begin{equation}\label{PBC}
\left\{
\begin{aligned}
    &\mathbf{u}\cdot \mathbf{n}=0, \\
    &(\varphi+\overline{\rho}) \mathbf{u}\cdot \mathbf{t}+\sqrt{\frac{2}{R}}a_{\mathbf{u}}^{I}\frac{\mu(\theta_{w})}{\sqrt{\theta_{w}}}(\nabla\mathbf{u}+\nabla\mathbf{u}^{\mathsf{T}}):\mathbf{n}\otimes\mathbf{t}-\frac{4}{5R}a_{\theta}^{I}\frac{\kappa(\theta_{w})}{\theta_{w}}\nabla(\zeta+\tilde{\theta})\cdot\mathbf{t}=0, \\
    &(\varphi+\overline{\rho})\zeta-\frac{1}{R}a_{\mathbf{u}}^{II}\mu(\theta_{w})\nabla\mathbf{u}:\mathbf{n}\otimes\mathbf{n}+\frac{2}{5R}\sqrt{\frac{2}{R}}a_{\theta}^{II}\frac{\kappa(\theta_{w})}{\sqrt{\theta_{w}}}\nabla(\zeta+\tilde{\theta})\cdot\mathbf{n}=0,
\end{aligned}
\right.
\end{equation}
for $x\in\partial\Omega$, where
\begin{align*}
    \mathbf{F}(\varphi,\mathbf{u},\zeta):=& -\varphi\mathbf{u}\cdot\nabla\mathbf{u}-\overline{\rho}\mathbf{u}\cdot\nabla\mathbf{u}-R\varphi\nabla\zeta-R\varphi\nabla\tilde{\theta}-R\overline{\rho}\nabla\tilde{\theta}-R\zeta\nabla\varphi-R(\tilde{\theta}-\overline{\theta})\nabla\varphi\\
    &+\div(\mathbb{S}(\mathbf{u},\tilde{\theta}+\zeta)-\mathbb{S}(\mathbf{u},\overline{\theta})),\\
    G(\varphi,\mathbf{u},\zeta):=& -c_{v}\varphi\mathbf{u}\cdot\nabla\zeta-c_{v}\overline{\rho}\mathbf{u}\cdot\nabla\zeta-c_{v}\varphi \mathbf{u}\cdot\nabla\tilde{\theta}-c_{v}\overline{\rho}\mathbf{u}\cdot\nabla\tilde{\theta}\\
    & -R\varphi\zeta\div\mathbf{u}-R\overline{\rho}\zeta\div\mathbf{u}-R\varphi\tilde{\theta}\div\mathbf{u}-R\overline{\rho}(\tilde{\theta}-\overline{\theta})\div\mathbf{u}\\
    & +\div([\kappa(\theta)-\kappa(\overline{\theta})]\nabla\zeta)+\div(\kappa(\theta)\nabla\tilde{\theta})+\nabla\mathbf{u}:\mathbb{S}(\mathbf{u},\tilde{\theta}+\zeta).
\end{align*}

Therefore, instead of studying the boundary value problem \eqref{SNS} and \eqref{BC}, it suffices to consider the boundary value problem \eqref{PSNS} and \eqref{PBC}. Before studying $(\varphi,\mathbf{u},\zeta)$, let us quantify $\mathbf{F}$ and $G$ in terms of $\varphi, \mathbf{u}, \zeta$ by the following result. We denote 
\begin{equation*}
    N_{k}(\varphi,\mathbf{u},\zeta)=\|\varphi\|_{H^{k}(\Omega)}+\|\mathbf{u}\|_{H^{k+1}(\Omega)}+\|\zeta\|_{H^{k+1}(\Omega)}.
\end{equation*}
Then we have the following result.

\begin{lemma}\label{EstimateFGLemma}
There exists a constant $C>0$ only depending on $c_{v}, R, \mu, \kappa, \overline{\rho}, \overline{\theta}$ such that 
\begin{align}
  &\|(\mathbf{F},G)\|_{L^{2}(\Omega)}\leq C(1+N_{1}(\varphi,\mathbf{u},\zeta))(\sqrt{\delta}+N_{2}(\varphi,\mathbf{u},\zeta))^{2}, \label{L2EstimateFG}\\
  &\|(\nabla\mathbf{F},\nabla G)\|_{L^{2}(\Omega)}\leq C(1+N_{2}(\varphi,\mathbf{u},\zeta))(\sqrt{\delta}+N_{2}(\varphi,\mathbf{u},\zeta))^{2}. \label{H1EstimateFG}
\end{align}
\end{lemma}
\begin{proof}
For $\|\mathbf{F}\|_{L^{2}(\Omega)}$, by H\"{o}lder's inequality and Sobolev's inequality,
  \begin{align}\label{L2EstimateF}
    \|\mathbf{F}\|_{L^{2}(\Omega)}\leq&C_{1}(\|\varphi\mathbf{u}\cdot\nabla\mathbf{u}\|_{L^{2}(\Omega)}+\|\overline{\rho}\mathbf{u}\cdot\nabla\mathbf{u}\|_{L^{2}(\Omega)}+\|\varphi\nabla\zeta\|_{L^{2}(\Omega)}+\|\varphi\nabla\tilde{\theta}\|_{L^{2}(\Omega)}+\|\overline{\rho}\nabla\tilde{\theta}\|_{L^{2}(\Omega)} \nonumber\\
    &+\|\zeta\nabla\varphi\|_{L^{2}(\Omega)}+\|(\tilde{\theta}-\overline{\theta})\nabla\varphi\|_{L^{2}(\Omega)}+\|\div(\mathbb{S}(\mathbf{u},\tilde{\theta}+\zeta)-\mathbb{S}(\mathbf{u},\overline{\theta}))\|_{L^{2}(\Omega)}) \nonumber\\
    \leq&C_{1}(\|\mathbf{u}\|_{L^{\infty}(\Omega)}\|\varphi\|_{L^{4}(\Omega)}\|\nabla\mathbf{u}\|_{L^{4}(\Omega)}+\|\mathbf{u}\|_{L^{\infty}(\Omega)}\|\nabla\mathbf{u}\|_{L^{2}(\Omega)}+\|\varphi\|_{L^{4}(\Omega)}\|\nabla\zeta\|_{L^{4}(\Omega)} \nonumber\\
    &+\varepsilon\|\varphi\|_{L^{2}(\Omega)}+\varepsilon+\|\zeta\|_{L^{\infty}(\Omega)}\|\nabla\varphi\|_{L^{2}(\Omega)}+\varepsilon\|\nabla\varphi\|_{L^{2}(\Omega)}+\|\nabla\zeta\|_{L^{4}(\Omega)}\|\nabla\mathbf{u}\|_{L^{4}(\Omega)} \nonumber\\
    &+\|\zeta\|_{L^{\infty}(\Omega)}\|\nabla\div\mathbf{u}\|_{L^{2}(\Omega)}) \nonumber\\
    \leq&C_{1}(1+N_{1}(\varphi,\mathbf{u},\zeta))(\sqrt{\delta}+N_{2}(\varphi,\mathbf{u},\zeta))^{2},
  \end{align}
  where $C_{1}>0$ is a constant only depending on $R, \overline{\rho}, \mu, \overline{\theta}$.
 
 For $\|G\|_{L^{2}(\Omega)}$, by H\"{o}lder's inequality and Sobolev's inequality,
  \begin{align}\label{L2EstimateG}
    \|G\|_{L^{2}(\Omega)}\leq&C_{2}(\|\varphi\mathbf{u}\cdot\nabla\zeta\|_{L^{2}(\Omega)}+\|\overline{\rho}\mathbf{u}\cdot\nabla\zeta\|_{L^{2}(\Omega)}+\|\varphi\mathbf{u}\cdot\nabla\tilde{\theta}\|_{L^{2}(\Omega)}+\|\overline{\rho}\mathbf{u}\cdot\nabla\tilde{\theta}\|_{L^{2}(\Omega)}+\|\varphi\zeta\div\mathbf{u}\|_{L^{2}(\Omega)} \nonumber\\
    &+\|\overline{\rho}\zeta\div\mathbf{u}\|_{L^{2}(\Omega)}+\|\varphi\tilde{\theta}\div\mathbf{u}\|_{L^{2}(\Omega)}+\|(\tilde{\theta}-\overline{\theta})\nabla\varphi\|_{L^{2}(\Omega)}+\|\div(\kappa(\theta)-\kappa(\overline{\theta}))\nabla\zeta)\|_{L^{2}(\Omega)} \nonumber\\
    &+\|\div(\kappa(\theta)\nabla\tilde{\theta})\|_{L^{2}(\Omega)}+\|\nabla\mathbf{u}:\mu(\theta)\mathbb{S}(\mathbf{u})\|_{L^{2}(\Omega)}) \nonumber\\
    \leq&C_{2}(\|\mathbf{u}\|_{L^{\infty}(\Omega)}\|\varphi\|_{L^{4}(\Omega)}\|\nabla\zeta\|_{L^{4}(\Omega)}+\|\mathbf{u}\|_{H^{2}(\Omega)}\|\nabla\zeta\|_{L^{2}(\Omega)}+\varepsilon\|\varphi\|_{L^{4}(\Omega)}\|\mathbf{u}\|_{L^{4}(\Omega)} \nonumber\\
    &+\varepsilon\|\mathbf{u}\|_{L^{2}(\Omega)}+\|\varphi\|_{L^{4}(\Omega)}\|\zeta\|_{L^{\infty}(\Omega)}\|\div\mathbf{u}\|_{L^{4}(\Omega)}+\|\zeta\|_{H^{2}(\Omega)}\|\div\mathbf{u}\|_{L^{2}(\Omega)} \nonumber\\
    &+\|\varphi\|_{L^{4}(\Omega)}\|\div\mathbf{u}\|_{L^{4}(\Omega)}+\varepsilon\|\nabla\varphi\|_{L^{2}(\Omega)}+\|\nabla\zeta\|_{L^{4}(\Omega)}\|\nabla\zeta\|_{L^{4}(\Omega)}+\|\zeta\|_{H^{2}(\Omega)}\|\div\nabla\zeta\|_{L^{2}(\Omega)} \nonumber\\
    &+\varepsilon\|\nabla\zeta\|_{L^{2}(\Omega)}+\|\zeta\|_{H^{2}(\Omega)}\|\mathbf{u}\|_{H^{2}(\Omega)}\|\nabla\mathbf{u}\|_{L^{2}(\Omega)}+\varepsilon\|\mathbf{u}\|_{H^{2}(\Omega)}\|\nabla\mathbf{u}\|_{L^{2}(\Omega)}) \nonumber\\
    \leq&C_{2}(1+N_{1}(\varphi,\mathbf{u},\zeta))(\sqrt{\delta}+N_{2}(\varphi,\mathbf{u},\zeta))^{2},
  \end{align}
  where $C_{2}>0$ is a constant only depending on $c_{v}, R, \mu, \kappa, \overline{\rho}, \overline{\theta}$. 
  
  Combining \eqref{L2EstimateF} and \eqref{L2EstimateG}, we obtain \eqref{L2EstimateFG}.

  For $\|\nabla\mathbf{F}\|_{L^{2}(\Omega)}$, by H\"{o}lder's inequality and Sobolev's inequality,
  \begin{align}\label{H1EstimateF}
    \|\nabla\mathbf{F}\|_{L^{2}(\Omega)}\leq&C_{3}(\|\nabla\varphi\cdot\mathbf{u}\cdot\nabla\mathbf{u}\|_{L^{2}(\Omega)}+\|\varphi\nabla\mathbf{u}\cdot\nabla\mathbf{u}\|_{L^{2}(\Omega)}+\|\varphi\mathbf{u}\cdot\nabla^{2}\mathbf{u}\|_{L^{2}(\Omega)}+\|\overline{\rho}\nabla\mathbf{u}\cdot\nabla\mathbf{u}\|_{L^{2}(\Omega)} \nonumber\\
    &+\|\overline{\rho}\mathbf{u}\cdot\nabla^{2}\mathbf{u}\|_{L^{2}(\Omega)}+\|\nabla\varphi\cdot\nabla\zeta\|_{L^{2}(\Omega)}+\|\varphi\nabla^{2}\zeta\|_{L^{2}(\Omega)}+\|\nabla\varphi\cdot\nabla\tilde{\theta}\|_{L^{2}(\Omega)}+\|\overline{\rho}\nabla^{2}\tilde{\theta}\|_{L^{2}(\Omega)} \nonumber\\
    &+\|\nabla\zeta\cdot\nabla\varphi\|_{L^{2}(\Omega)}+\|\zeta\nabla^{2}\varphi\|_{L^{2}(\Omega)}+\|\nabla(\tilde{\theta}-\overline{\theta})\cdot\nabla\varphi\|_{L^{2}(\Omega)}+\|(\tilde{\theta}-\overline{\theta})\nabla^{2}\varphi\|_{L^{2}(\Omega)} \nonumber\\
    &+\|\nabla\div(\mathbb{S}(\mathbf{u},\tilde{\theta}+\zeta)-\mathbb{S}(\mathbf{u},\overline{\theta}))\|_{L^{2}(\Omega)}) \nonumber\\
    \leq&C_{3}(\|\mathbf{u}\|_{L^{\infty}(\Omega)}\|\nabla\mathbf{u}\|_{L^{\infty}(\Omega)}\|\nabla\varphi\|_{L^{2}(\Omega)}+\|\varphi\|_{L^{6}(\Omega)}\|\nabla\mathbf{u}\|_{L^{6}(\Omega)}\|\nabla\mathbf{u}\|_{L^{6}(\Omega)} \nonumber\\
    &+\|\varphi\|_{L^{\infty}(\Omega)}\|\mathbf{u}\|_{L^{\infty}(\Omega)}\|\nabla^{2}\mathbf{u}\|_{L^{2}(\Omega)}+\|\nabla\mathbf{u}\|_{L^{4}(\Omega)}\|\nabla\mathbf{u}\|_{L^{4}(\Omega)}+\|\mathbf{u}\|_{L^{\infty}(\Omega)}\|\nabla^{2}\mathbf{u}\|_{L^{2}(\Omega)} \nonumber\\
    &+\|\nabla\zeta\|_{L^{4}(\Omega)}\|\nabla\varphi\|_{L^{4}(\Omega)}+\|\varphi\|_{L^{\infty}(\Omega)}\|\nabla^{2}\zeta\|_{L^{2}(\Omega)}+\varepsilon\|\nabla\varphi\|_{L^{2}(\Omega)}+\varepsilon+\|\nabla\zeta\|_{L^{4}(\Omega)}\|\nabla\varphi\|_{L^{4}(\Omega)} \nonumber\\
    &+\|\zeta\|_{L^{\infty}}\|\nabla^{2}\varphi\|_{L^{2}(\Omega)}+\varepsilon\|\nabla\varphi\|_{L^{2}(\Omega)}+\varepsilon\|\nabla^{2}\varphi\|_{L^{2}(\Omega)}+\|\nabla\mathbf{u}\|_{L^{\infty}(\Omega)}\|\nabla^{2}\zeta\|_{L^{2}(\Omega)} \nonumber\\
    &+\|\nabla\zeta\|_{H^{2}(\Omega)}\|\nabla^{2}\mathbf{u}\|_{L^{2}(\Omega)}+\|\nabla\zeta\|_{H^{2}(\Omega)}\|\nabla\div\mathbf{u}\|_{L^{2}(\Omega)}+\|\zeta\|_{H^{2}(\Omega)}\|\nabla^{2}\div\mathbf{u}\|_{L^{2}(\Omega)}) \nonumber\\
    \leq&C_{3}(1+N_{2}(\varphi,\mathbf{u},\zeta))(\sqrt{\delta}+N_{2}(\varphi,\mathbf{u},\zeta))^{2},
  \end{align}
  where $C_{3}>0$ is a constant only depending on $R, \mu, \overline{\rho}, \overline{\theta}$.
  
  For $\|\nabla G\|_{L^{2}(\Omega)}$, by H\"{o}lder's inequality and Sobolev's inequality,
  \begin{align}\label{H1EstimateG}
    \|\nabla G\|_{L^{2}(\Omega)}\leq&C_{4}(\|\nabla\varphi\cdot\mathbf{u}\cdot\nabla\zeta\|_{L^{2}(\Omega)}+\|\varphi\nabla\mathbf{u}\cdot\nabla\zeta\|_{L^{2}(\Omega)}+\|\varphi\mathbf{u}\cdot\nabla^{2}\zeta\|_{L^{2}(\Omega)}+\|\overline{\rho}\nabla\mathbf{u}\cdot\nabla\zeta\|_{L^{2}(\Omega)} \nonumber\\
    &+\|\overline{\rho}\mathbf{u}\cdot\nabla^{2}\zeta\|_{L^{2}(\Omega)}+\|\nabla\varphi\cdot\mathbf{u}\cdot\nabla\tilde{\theta}\|_{L^{2}(\Omega)}+\|\varphi\nabla\mathbf{u}\cdot\nabla\tilde{\theta}\|_{L^{2}(\Omega)}+\|\varphi\mathbf{u}\cdot\nabla^{2}\tilde{\theta}\|_{L^{2}(\Omega)} \nonumber\\
    &+\|\overline{\rho}\nabla\mathbf{u}\cdot\nabla\tilde{\theta}\|_{L^{2}(\Omega)}+\|\overline{\rho}\mathbf{u}\cdot\nabla^{2}\tilde{\theta}\|_{L^{2}(\Omega)}+\|\nabla\varphi\cdot\zeta\div\mathbf{u}\|_{L^{2}(\Omega)}+\|\varphi\nabla\zeta\div\mathbf{u}\|_{L^{2}(\Omega)} \nonumber\\
    &+\|\varphi\zeta\nabla\div\mathbf{u}\|_{L^{2}(\Omega)}+\|\overline{\rho}\nabla\zeta\div\mathbf{u}\|_{L^{2}(\Omega)}+\|\overline{\rho}\zeta\nabla\div\mathbf{u}\|_{L^{2}(\Omega)}+\|\nabla\varphi\cdot\tilde{\theta}\div\mathbf{u}\|_{L^{2}(\Omega)} \nonumber\\
    &+\|\varphi\tilde{\theta}\nabla\div\mathbf{u}\|_{L^{2}(\Omega)}+\|\overline{\rho}(\tilde{\theta}-\overline{\theta})\nabla\div\mathbf{u}\|_{L^{2}(\Omega)}+\|\nabla\div([\kappa(\theta)-\kappa(\tilde{\theta})]\nabla\zeta)\|_{L^{2}(\Omega)} \nonumber\\
    &+\|\nabla\div(\kappa(\theta)\nabla\tilde{\theta})\|_{L^{2}(\Omega)}+\|\nabla^{2}\mathbf{u}:\nabla\mathbb{S}(\mathbf{u},\tilde{\theta}+\zeta)\|_{L^{2}(\Omega)}+\|\nabla\mathbf{u}:\nabla\mathbb{S}(\mathbf{u},\tilde{\theta}+\zeta)\|_{L^{2}(\Omega)}) \nonumber\\
    \leq&C_{4}(\|\mathbf{u}\|_{L^{\infty}}\|\nabla\zeta\|_{L^{\infty}(\Omega)}\|\nabla\varphi\|_{L^{2}(\Omega)}+\|\varphi\|_{L^{6}(\Omega)}\|\nabla\mathbf{u}\|_{L^{6}(\Omega)}\|\nabla\zeta\|_{L^{6}(\Omega)} \nonumber\\
    &+\|\varphi\|_{L^{\infty}(\Omega)}\|\mathbf{u}\|_{L^{\infty}(\Omega)}\|\nabla^{2}\zeta\|_{L^{2}(\Omega)}+\|\nabla\mathbf{u}\|_{L^{4}(\Omega)}\|\nabla\zeta\|_{L^{4}(\Omega)}+\|\mathbf{u}\|_{H^{2}(\Omega)}\|\nabla^{2}\zeta\|_{L^{2}(\Omega)} \nonumber\\
    &+\varepsilon\|\mathbf{u}\|_{H^{2}(\Omega)}\|\nabla\varphi\|_{L^{2}(\Omega)}+\varepsilon\|\varphi\|_{L^{4}(\Omega)}\|\nabla\mathbf{u}\|_{L^{4}(\Omega)}+\varepsilon\|\mathbf{u}\|_{L^{\infty}(\Omega)}\|\varphi\|_{L^{2}(\Omega)}+\varepsilon\|\nabla\mathbf{u}\|_{L^{2}(\Omega)} \nonumber\\
    &+\varepsilon\|\mathbf{u}\|_{L^{2}(\Omega)}+\|\zeta\|_{L^{\infty}(\Omega)}\|\div\mathbf{u}\|_{L^{4}(\Omega)}\|\nabla\varphi\|_{L^{4}(\Omega)}+\|\varphi\|_{L^{6}(\Omega)}\|\nabla\zeta\|_{L^{6}(\Omega)}\|\div\mathbf{u}\|_{L^{6}(\Omega)} \nonumber\\
    &+\|\zeta\|_{L^{\infty}(\Omega)}\|\varphi\|_{L^{\infty}(\Omega)}\|\nabla\div\mathbf{u}\|_{L^{2}(\Omega)}+\|\nabla\zeta\|_{L^{4}(\Omega)}\|\div\mathbf{u}\|_{L^{4}(\Omega)}+\|\zeta\|_{L^{\infty}(\Omega)}\|\nabla\div\mathbf{u}\|_{L^{2}(\Omega)} \nonumber\\
    &+\|\div\mathbf{u}\|_{L^{\infty}(\Omega)}\|\nabla\varphi\|_{L^{2}(\Omega)}+\|\varphi\|_{L^{\infty}(\Omega)}\|\nabla\div\mathbf{u}\|_{L^{2}(\Omega)}+\varepsilon\|\nabla\div\mathbf{u}\|_{L^{2}(\Omega)} \nonumber\\
    &+\|\nabla\zeta\|_{L^{\infty}(\Omega)}\|\nabla^{2}\zeta\|_{L^{2}(\Omega)}+\|\nabla\zeta\|_{L^{\infty}(\Omega)}\|\Delta\zeta\|_{L^{2}(\Omega)}+\|\zeta\|_{L^{\infty}}\|\nabla\Delta\zeta\|_{L^{2}(\Omega)}+\varepsilon\|\nabla^{2}\zeta\|_{L^{2}(\Omega)} \nonumber\\
    &+\varepsilon+\|\nabla\mathbf{u}\|_{L^{\infty}(\Omega)}\|\nabla^{2}\mathbf{u}\|_{L^{2}(\Omega)}+\|\nabla\zeta\|_{L^{6}(\Omega)}\|\nabla\mathbf{u}\|_{L^{6}(\Omega)}\|\nabla\mathbf{u}\|_{L^{6}(\Omega)} \nonumber\\
    \leq&C_{4}(1+N_{2}(\varphi,\mathbf{u},\zeta))(\sqrt{\delta}+N_{2}(\varphi,\mathbf{u},\zeta))^{2},
  \end{align}
  where $C_{4}>0$ is a constant only depending on $c_{v}, R, \mu, \kappa, \overline{\rho}, \overline{\theta}$.

  Combining \eqref{H1EstimateF} and \eqref{H1EstimateG}, we obtain \eqref{H1EstimateFG} and thus complete the proof of Lemma \ref{EstimateFGLemma}.
\end{proof}

\subsection{Variational formulation of linearized problem}
To solve the above boundary value problem \eqref{PSNS} and \eqref{PBC}, we consider the following linear problem. Let $h>0$ be an arbitrary fixed constant and $(\varphi',\mathbf{u}',\zeta')\in (H^{2}(\Omega)\cap L_{0}^{2}(\Omega))\times V^{3}(\Omega)\times H^{3}(\Omega)$ with $N_{2}(\varphi',\mathbf{u}',\zeta')\leq\varepsilon$ where $\varepsilon>0$ is chosen to be small such that $\|\varphi'\|_{L^{\infty}(\Omega)}+\overline{\rho}>0$ and $\|\zeta'\|_{L^{\infty}(\Omega)}+\overline{\theta}>0$, then we are going to prove that there exists a unique
solution $(\varphi,\mathbf{u},\zeta)\in (L_{0}^{2}(\Omega)\cap H^{2}(\Omega))\times V^{3}(\Omega)\times H^{3}(\Omega)$ to the following boundary value problem:
\begin{equation}\label{LSNS}
    \left\{
    \begin{aligned}
        &\frac{\varphi-\varphi'}{h}+\div(\varphi\mathbf{u})+\overline{\rho}\div\mathbf{u}=0, \\
        &-\div\mathbb{S}(\mathbf{u},\overline{\theta})+R\overline{\rho}\nabla\zeta =\mathbf{F}(\varphi',\mathbf{u}',\zeta')-R\overline{\theta}\nabla\varphi', \\
        &-\div(\kappa(\overline{\theta})\nabla\zeta)+(R\overline{\rho}\overline{\theta}+\mathcal{R}\overline{\rho})\div\mathbf{u} =G(\varphi',\mathbf{u}',\zeta')-\mathcal{R}\div(\varphi'\mathbf{u}'),
    \end{aligned}
    \right.
\end{equation}
for $x\in\Omega$, with the boundary conditions
\begin{equation}\label{LBC}
    \left\{
    \begin{aligned}
        &\mathbf{u}\cdot\mathbf{n}=0, \\
        &(\varphi'+\overline{\rho}) \mathbf{u}\cdot \mathbf{t}+\sqrt{\frac{2}{R}}a_{\mathbf{u}}^{I}\frac{\mu(\theta_{w})}{\sqrt{\theta_{w}}}(\nabla\mathbf{u}+\nabla\mathbf{u}^{\mathsf{T}}):\mathbf{n}\otimes\mathbf{t}-\frac{4}{5R}a_{\theta}^{I}\frac{\kappa(\theta_{w})}{\theta_{w}}\nabla(\zeta+\tilde{\theta})\cdot\mathbf{t}=0, \\
        &(\varphi'+\overline{\rho})\zeta-\frac{1}{R}a_{\mathbf{u}}^{II}\mu(\theta_{w})\nabla\mathbf{u}:\mathbf{n}\otimes\mathbf{n}+\frac{2}{5R}\sqrt{\frac{2}{R}}a_{\theta}^{II}\frac{\kappa(\theta_{w})}{\sqrt{\theta_{w}}}\nabla(\zeta+\tilde{\theta})\cdot\mathbf{n}=0,
    \end{aligned}
    \right.
\end{equation}
for $x\in\partial\Omega$, where $\mathcal{R}\in\mathbb{R}$ is also a constant satisfying
\begin{equation}\label{def.crad}
    R\overline{\rho}\sqrt{\frac{2}{R}}\frac{2}{3R}a_{\mathbf{u}}^{I}a_{\mathbf{u}}^{II}\frac{\mu(\overline{\theta})^{2}}{\sqrt{\overline{\theta}}}=R(\overline{\rho}\overline{\theta}+\mathcal{R}\overline{\rho})\frac{32}{75R^2}\sqrt{\frac{2}{R}}a_{\theta}^{I}a_{\theta}^{II}\frac{\mu(\overline{\theta})\kappa(\overline{\theta})}{\overline{\theta}\sqrt{\overline{\theta}}}.
\end{equation}
The choice of $\mathcal{R}$ will be explained after the proof of Theorem \ref{basicEstimateEllipticSystem}, see Remark \ref{remarkHyperbolicEffect}.

Now, we define the solution operator of the linear problem \eqref{LSNS} and \eqref{LBC}.
\begin{definition}
    Given $\overline{\rho}>0$, $\overline{\theta}>0$, $h>0$ be fixed constants. Let the wall temperature $\theta_{w}\in H^{\vartheta}(\partial\Omega)$ where $\vartheta>\frac{7}{2}$, and $(\varphi',\mathbf{u}',\zeta')\in (H^{2}(\Omega)\cap L_{0}^{2}(\Omega))\times V^{3}(\Omega)\times H^{3}(\Omega)$. If the boundary value problem \eqref{LSNS} and \eqref{LBC} admits a unique solution $(\varphi,\mathbf{u},\zeta)\in (L_{0}^{2}(\Omega)\cap H^{2}(\Omega))\times V^{3}(\Omega)\times H^{3}(\Omega)$, then we define the solution operator $\mathds{T}$ by $\mathds{T}(\varphi',\mathbf{u}',\zeta')=(\varphi,\mathbf{u},\zeta)$. 
\end{definition}

In the following, we will prove that $\mathds{T}$ is well-defined.

In order to consider the existence result for the boundary value problem $\eqref{LSNS}_{2}$ and $\eqref{LSNS}_{3}$ with slip boundary conditions \eqref{LBC}, we first give the definition of solutions in the variational form.

\begin{definition}\label{definitionWeakSolutionEllipticSystem}
We say that $(\mathbf{u},\xi)\in V(\Omega)\times H^{1}(\Omega)$ is a weak solution to the equations $\eqref{LSNS}_{2}$ and $\eqref{LSNS}_{3}$ with the slip boundary conditions \eqref{LBC} if 
\begin{align}\label{weakFormEllipticSystem}
    &\int_{\Omega}\sqrt{\frac{2}{R}}\frac{2}{3R}a_{\mathbf{u}}^{I}a_{\mathbf{u}}^{II}\frac{\mu(\tilde{\theta})^{2}}{\sqrt{\tilde{\theta}}}\nabla\mathbf{v}:\left(\nabla\mathbf{u}+\nabla\mathbf{u}^{\mathsf{T}}-\frac{2}{3}\div\mathbf{u}\mathbb{I}_{3}\right)d\mathbf{x} +\int_{\Omega}\frac{32}{75R^2}\sqrt{\frac{2}{R}}a_{\theta}^{I}a_{\theta}^{II}\frac{\kappa(\tilde{\theta})^{2}}{\tilde{\theta}\sqrt{\tilde{\theta}}}\nabla\xi\cdot\nabla\zeta d\mathbf{x} \nonumber\\
    &+\int_{\Omega}-\frac{4}{5R^2}a_{\mathbf{u}}^{II}a_{\theta}^I\frac{\mu(\tilde{\theta})\kappa(\tilde{\theta})}{\mu(\overline{\theta})\tilde{\theta}}\nabla\xi\cdot\mathfrak{n}\left(\nabla u_{1}+\partial_{1}\mathbf{u}^{\mathsf{T}}-\frac{2}{3}\mathbf{e}_{1}\div\mathbf{u}\right)d\mathbf{x} \nonumber \\
    &+\int_{\Omega}\frac{2}{3R}\sqrt{\frac{2}{R}}a_{\mathbf{u}}^{I}a_{\mathbf{u}}^{II}\nabla\left(\frac{\mu(\tilde{\theta})^{2}}{\sqrt{\tilde{\theta}}}\right)\otimes\mathbf{v}:\left(\nabla\mathbf{u}+\nabla\mathbf{u}^{\mathsf{T}}-\frac{2}{3}\div\mathbf{u}\mathbb{I}_{3}\right)d\mathbf{x} \nonumber\\
    &+\int_{\Omega}\frac{32}{75R^2}\sqrt{\frac{2}{R}}a_{\theta}^{I}a_{\theta}^{II}\nabla\left(\frac{\kappa(\tilde{\theta})\kappa(\tilde{\theta})}{\tilde{\theta}\sqrt{\tilde{\theta}}}\right)\cdot\xi\nabla\zeta d\mathbf{x} \nonumber\\
    &+\int_{\Omega}-\frac{4}{5R^2}a_{\mathbf{u}}^{II}a_{\theta}^I\nabla\left(\frac{\mu(\tilde{\theta})\kappa(\tilde{\theta})}{\mu(\overline{\theta})\tilde{\theta}}\right)\cdot\xi\mathfrak{n}\left(\nabla u_{1}+\partial_{1}\mathbf{u}^{\mathsf{T}}-\frac{2}{3}\mathbf{e}_{1}\div\mathbf{u}\right)d\mathbf{x} \nonumber\\
    &+\int_{\Omega}R\overline{\rho}\sqrt{\frac{2}{R}}\frac{2}{3R}a_{\mathbf{u}}^{I}a_{\mathbf{u}}^{II}\frac{\mu(\tilde{\theta})^{2}}{\mu(\overline{\theta})\sqrt{\tilde{\theta}}}\mathbf{v}\cdot\nabla\zeta d\mathbf{x} +\int_{\Omega}\frac{32}{75R^2}\sqrt{\frac{2}{R}}a_{\theta}^{I}a_{\theta}^{II}\frac{\kappa(\tilde{\theta})^{2}}{\kappa(\overline{\theta})\tilde{\theta}\sqrt{\tilde{\theta}}}(R\overline{\rho}\overline{\theta}+\mathcal{R}\overline{\rho})\xi\div\mathbf{u}d\mathbf{x} \nonumber\\
    &+\int_{\Omega}-\frac{4}{5R}\overline{\rho}a_{\mathbf{u}}^{II}a_{\theta}^I\frac{\mu(\tilde{\theta})\kappa(\tilde{\theta})}{\mu(\overline{\theta})\tilde{\theta}}\mathfrak{n}\xi\partial_{1}\zeta d\mathbf{x} \nonumber\\
    &+\sum_{p=0,1}\int_{\mathbb{T}_{p}^{2}}\frac{2}{3R}a_{\mathbf{u}}^{II}\mu(\theta_{w})(\varphi'+\overline{\rho})\mathbf{v}\cdot\mathbf{u}dx_{2}dx_{3} +\sum_{p=0,1}\int_{\mathbb{T}_{p}^{2}}\frac{16}{15R}a_{\theta}^{I}\frac{\kappa(\theta_{w})}{\theta_{w}}(\varphi'+\overline{\rho})\xi\zeta dx_{2}dx_{3} \nonumber \\
    &+\sum_{p=0,1}\int_{\mathbb{T}_{p}^{2}}\frac{8}{15R^2}a_{\mathbf{u}}^{II}a_{\theta}^{I}\frac{\mu(\theta_{w})\kappa(\theta_{w})}{\theta_{w}}\mathbf{v}\cdot\nabla\tilde{\theta}dx_{2}dx_{3} +\sum_{p=0,1}\int_{\mathbb{T}_{p}^{2}}\frac{32}{75R^2}a_{\theta}^{I}a_{\theta}^{II}\frac{\kappa(\theta_{w})^2}{\theta_{w}\sqrt{\theta_{w}}}\mathfrak{n}\xi\partial_{1}\tilde{\theta}dx_{2}dx_{3} \nonumber \\
    &+\sum_{p=0,1}\sum_{j=2,3}\int_{\mathbb{T}_{p}^{2}}\frac{8}{15R^{2}}a_{\mathbf{u}}^{II}a_{\theta}^{I}\left(\partial_{j}^{\frac{1}{2}}\left[\frac{\mu(\theta_{w})\kappa(\theta_{w})}{\theta_{w}}u_{j}\right]\cdot\partial_{j}^{\frac{1}{2}}\xi-\partial_{j}^{\frac{1}{2}}\left[\frac{\mu(\theta_{w})\kappa(\theta_{w})}{\theta_{w}}v_{j}\right]\cdot\partial_{j}^{\frac{1}{2}}\zeta\right) dx_{2}dx_{3} \nonumber\\
    = 
    &\int_{\Omega}\sqrt{\frac{2}{R}}\frac{2}{3R}a_{\mathbf{u}}^{I}a_{\mathbf{u}}^{II}\frac{\mu(\tilde{\theta})^{2}}{\mu(\overline{\theta})\sqrt{\tilde{\theta}}}\mathbf{v}\cdot\mathbf{F}(\varphi',\mathbf{u}',\zeta')d\mathbf{x}+\int_{\Omega}-R\tilde{\theta}\sqrt{\frac{2}{R}}\frac{2}{3R}a_{\mathbf{u}}^{I}a_{\mathbf{u}}^{II}\frac{\mu(\tilde{\theta})^{2}}{\mu(\overline{\theta})\sqrt{\tilde{\theta}}}\mathbf{v}\cdot\nabla\varphi'd\mathbf{x} \nonumber \\
    &+\int_{\Omega}\frac{32}{75R^2}\sqrt{\frac{2}{R}}a_{\theta}^{I}a_{\theta}^{II}\frac{\kappa(\tilde{\theta})^{2}}{\kappa(\overline{\theta})\tilde{\theta}\sqrt{\tilde{\theta}}}\xi G(\varphi',\mathbf{u}',\zeta')d\mathbf{x}+\int_{\Omega}-\frac{32}{75R^2}\sqrt{\frac{2}{R}}a_{\theta}^{I}a_{\theta}^{II}\frac{\kappa(\tilde{\theta})^{2}}{\kappa(\overline{\theta})\tilde{\theta}\sqrt{\tilde{\theta}}}\mathcal{R}\xi\div(\varphi'\mathbf{u}')d\mathbf{x} \nonumber \\
    &+\int_{\Omega}-\frac{4}{5R^2}a_{\mathbf{u}}^{II}a_{\theta}^I\frac{\mu(\tilde{\theta})\kappa(\tilde{\theta})}{\mu(\overline{\theta})\tilde{\theta}}\mathfrak{n}\xi F_{1}(\varphi',\mathbf{u}',\zeta')d\mathbf{x}+\int_{\Omega}\frac{4}{5R}\overline{\theta}a_{\mathbf{u}}^{II}a_{\theta}^I\frac{\mu(\tilde{\theta})\kappa(\tilde{\theta})}{\mu(\overline{\theta})\tilde{\theta}}\mathfrak{n}\xi\partial_{1}\varphi' d\mathbf{x},
\end{align}
for all $(\mathbf{v},\xi)\in V^{1}(\Omega)\times H^{1}(\Omega)$. Here, $\mathfrak{n}(\mathbf{x})=2x_{1}-1$, $\partial_{j}^{\frac{1}{2}}$ in the tenth term is in the sense of Fourier transform, 
\begin{align*}
    &\sum_{p=0,1}\sum_{j=2,3}\int_{\mathbb{T}_{p}^{2}}\frac{8}{15R^{2}}a_{\mathbf{u}}^{II}a_{\theta}^{I}\left(\partial_{j}^{\frac{1}{2}}\left[\frac{\mu(\theta_{w})\kappa(\theta_{w})}{\theta_{w}}u_{j}\right]\cdot\partial_{j}^{\frac{1}{2}}\xi-\partial_{j}^{\frac{1}{2}}\left(\frac{\mu(\theta_{w})\kappa(\theta_{w})}{\theta_{w}}v_{j}\right)\cdot\partial_{j}^{\frac{1}{2}}\zeta\right) dx_{2}dx_{3} \nonumber\\
    =&\sum_{p=0,1}\sum_{j=2,3}\sum_{m_{j}\in\mathbb{Z}}\frac{16\pi i}{15R^{2}}a_{\mathbf{u}}^{II}a_{\theta}^{I}m_{j}\left(\left[\frac{\mu(\theta_{w})\kappa(\theta_{w})}{\theta_{w}}u_{j}\right]^{\wedge}(m_{j})\cdot[\xi]^{\wedge}(m_{j})-\left[\frac{\mu(\theta_{w})\kappa(\theta_{w})}{\theta_{w}}v_{j}\right]^{\wedge}(m_{j})\cdot[\zeta]^{\wedge}(m_{j})\right), 
\end{align*}
where we denote the Fourier transform,
\begin{equation*}
    [f]^{\wedge}(m_{j}):=\int_{\mathbb{T}}f(x_{j})e^{-2\pi im_{j}x_{j}}dx_{j},
\end{equation*}
for $j=2,3$.
\end{definition}

\begin{remark}
  Here, the reason for introducing fractional derivatives is to ensure that the weak solution is well-defined in the space $V^{1}(\Omega)\times H^{1}(\Omega)$. Indeed, from our construction of the variational form, it is expected that the following boundary term would appear:
  \begin{equation}\label{badTerm3}
    \sum_{p=0,1}\int_{\mathbb{T}_{p}^{2}}\frac{8}{15R^{2}}a_{\mathbf{u}}^{II}a_{\theta}^{I}\frac{\mu(\theta_{w})\kappa(\theta_{w})}{\theta_{w}}(\mathbf{u}\cdot\nabla\xi-\mathbf{v}\cdot\nabla\zeta)dx_{2}dx_{3},
  \end{equation}
  in the weak form of the boundary value problem $\eqref{LSNS}_{2,3}$ and \eqref{LBC}. However, by the Sobolev embedding theorem, functions in the space $H^{1}(\Omega)$ only have traces in $H^{\frac{1}{2}}(\Omega)$, so that \eqref{badTerm3} is not well-defined in this framework. Fortunately, the normal component of $\mathbf{u}$ vanishes on the boundary, which makes it possible to control \eqref{badTerm3} in $H^{\frac{1}{2}}(\partial\Omega)$. This kind of difficulty does not appear in the previous studies of boundary value problems in the field of compressible fluids, and it does show that the slip boundary conditions \eqref{BC} behave differently from those ``classic'' boundary conditions.
\end{remark}

\subsection{Existence results for $\mathbf{u}$ and $\zeta$}

To prove the existence of the weak solution, it suffices to prove the following result.

\begin{lemma}\label{coervivityEllipticSystem}
    Under the assumption of Theorem \ref{mainTheorem}, suppose $(u,\zeta)\in V^{1}(\Omega)\times H^{1}(\Omega)$ is a weak solution in the sense of \eqref{weakFormEllipticSystem}, then there exist two constants $\delta_{0}, \varepsilon_{0}>0$ such that for $\delta\in[0,\delta_{0}), \varepsilon\in[0,\varepsilon_{0})$, if the wall temperature $\theta_{w}\in H^{\vartheta}(\partial\Omega)$ with $\|\theta_{w}-\overline{\theta}\|_{H^{\vartheta}(\partial\Omega)}\leq \delta$ where $\vartheta>\frac{7}{2}$, and $(\varphi',\mathbf{u}',\zeta')\in (H^{2}(\Omega)\cap L_{0}^{2}(\Omega))\times V^{3}(\Omega)\times H^{3}(\Omega)$ with $N_{2}(\varphi',\mathbf{u}',\zeta')\leq\varepsilon$ such that $\|\varphi'\|_{L^{\infty}(\Omega)}+\overline{\rho}>0$ and $\|\zeta'\|_{L^{\infty}(\Omega)}+\overline{\theta}>0$, then there exists a constant $C>0$ depends only on $a_{\mathbf{u}}^{I}, a_{\mathbf{u}}^{II}, a_{\theta}^{I}, a_{\theta}^{II}, c_{v}, R, \mu, \kappa, \overline{\rho},\overline{\theta}$ such that
    \begin{equation}\label{basicEstimateEllipticSystem}
        \|\mathbf{u}\|_{H^1(\Omega)}^2 +\|\zeta\|_{H^1(\Omega)}^2\leq CN_{1}(\varphi',\mathbf{u}',\zeta')^{2}(1+N_{2}(\varphi',\mathbf{u}',\zeta'))^{4}.
    \end{equation}
\end{lemma}

\begin{proof}
    Since $(u,\zeta)\in(H^{1}(\Omega)\cap V)\times H^{1}(\Omega)$, taking $\mathbf{v}=\mathbf{u}$, $\xi=\zeta$ gives
    \begin{align*}
        &\underbrace{\int_{\Omega}\frac{2}{3R}\sqrt{\frac{2}{R}}a_{\mathbf{u}}^{I}a_{\mathbf{u}}^{II}\frac{\mu(\tilde{\theta})^{2}}{\sqrt{\tilde{\theta}}}\nabla\mathbf{u}:\left(\nabla\mathbf{u}+\nabla\mathbf{u}^{\mathsf{T}}-\frac{2}{3}\div\mathbf{u}\mathbb{I}_{3}\right)d\mathbf{x}}_{I_{1}} +\underbrace{\int_{\Omega}\frac{32}{75R^2}\sqrt{\frac{2}{R}}a_{\theta}^{I}a_{\theta}^{II}\frac{\kappa(\tilde{\theta})^{2}}{\tilde{\theta}\sqrt{\tilde{\theta}}}|\nabla\zeta|^{2}d\mathbf{x}}_{I_{2}} \\
        &+\underbrace{\int_{\Omega}-\frac{4}{5R^2}a_{\mathbf{u}}^{II}a_{\theta}^I\frac{\mu(\tilde{\theta})\kappa(\tilde{\theta})}{\mu(\overline{\theta})\tilde{\theta}}\nabla\zeta\cdot\mathfrak{n}\left(\nabla u_{1}+\partial_{1}\mathbf{u}^{\mathsf{T}}-\frac{2}{3}\mathbf{e}_{1}\div\mathbf{u}\right)d\mathbf{x} }_{I_{3}} \\
        &+\underbrace{\int_{\Omega}\frac{2}{3R}\sqrt{\frac{2}{R}}a_{\mathbf{u}}^{I}a_{\mathbf{u}}^{II}\nabla\left(\frac{\mu(\tilde{\theta})^{2}}{\sqrt{\tilde{\theta}}}\right)\otimes\mathbf{u}:\left(\nabla\mathbf{u}+\nabla\mathbf{u}^{\mathsf{T}}-\frac{2}{3}\div\mathbf{u}\mathbb{I}_{3}\right)d\mathbf{x}}_{I_{4}} \nonumber\\
        &+\underbrace{\int_{\Omega}\frac{32}{75R^2}\sqrt{\frac{2}{R}}a_{\theta}^{I}a_{\theta}^{II}\nabla\left(\frac{\kappa(\tilde{\theta})\kappa(\tilde{\theta})}{\tilde{\theta}\sqrt{\tilde{\theta}}}\right)\zeta\cdot\nabla\zeta d\mathbf{x}}_{I_{5}} \\
        &+\underbrace{\int_{\Omega}-\frac{4}{5R^2}a_{\mathbf{u}}^{II}a_{\theta}^I\nabla\left(\frac{\mu(\tilde{\theta})\kappa(\tilde{\theta})}{\mu(\overline{\theta})\tilde{\theta}}\right)\cdot\zeta\mathfrak{n}\left(\nabla u_{1}+\partial_{1}\mathbf{u}^{\mathsf{T}}-\frac{2}{3}\mathbf{e}_{1}\div\mathbf{u}\right)d\mathbf{x}}_{I_{6}} \\
        &+\underbrace{\int_{\Omega}R\overline{\rho}\sqrt{\frac{2}{R}}\frac{2}{3R}a_{\mathbf{u}}^{I}a_{\mathbf{u}}^{II}\frac{\mu(\tilde{\theta})^{2}}{\mu(\overline{\theta})\sqrt{\tilde{\theta}}}\mathbf{u}\cdot\nabla\zeta d\mathbf{x}}_{I_{7}} +\underbrace{\int_{\Omega}\frac{32}{75R^2}\sqrt{\frac{2}{R}}a_{\theta}^{I}a_{\theta}^{II}\frac{\kappa(\tilde{\theta})^{2}}{\kappa(\overline{\theta})\tilde{\theta}\sqrt{\tilde{\theta}}}(R\overline{\rho}\overline{\theta}+\mathcal{R}\overline{\rho})\zeta\div\mathbf{u}d\mathbf{x}}_{I_{8}} \\
        &+\underbrace{\int_{\Omega}-\frac{4}{5R}\overline{\rho}a_{\mathbf{u}}^{II}a_{\theta}^I\frac{\mu(\tilde{\theta})\kappa(\tilde{\theta})}{\mu(\overline{\theta})\tilde{\theta}}\mathfrak{n}\zeta\partial_{1}\zeta d\mathbf{x}}_{I_{9}} \\
        &+\underbrace{\sum_{p=0,1}\int_{\mathbb{T}_{p}^{2}}\frac{2}{3R}a_{\mathbf{u}}^{II}\mu(\theta_{w})(\varphi'+\overline{\rho})|\mathbf{u}|^{2}dx_{2}dx_{3}}_{I_{10}} +\underbrace{\sum_{p=0,1}\int_{\mathbb{T}_{p}^{2}}\frac{16}{15R}a_{\theta}^{I}\frac{\kappa(\theta_{w})}{\theta_{w}}(\varphi'+\overline{\rho})|\zeta|^{2} dx_{2}dx_{3}}_{I_{11}} \\
        &+\underbrace{\sum_{p=0,1}\int_{\mathbb{T}_{p}^{2}}\frac{8}{15R^2}a_{\mathbf{u}}^{II}a_{\theta}^{I}\frac{\mu(\theta_{w})\kappa(\theta_{w})}{\theta_{w}}\mathbf{u}\cdot\nabla\tilde{\theta}dx_{2}dx_{3}}_{I_{12}} +\underbrace{\sum_{p=0,1}\int_{\mathbb{T}_{p}^{2}}\frac{32}{75R^2}a_{\theta}^{I}a_{\theta}^{II}\frac{\kappa(\theta_{w})^2}{\theta_{w}\sqrt{\theta_{w}}}\mathfrak{n}\zeta\partial_{1}\tilde{\theta}dx_{2}dx_{3}}_{I_{13}} \\
        = 
        &\underbrace{\int_{\Omega}\frac{2}{3R}\sqrt{\frac{2}{R}}a_{\mathbf{u}}^{I}a_{\mathbf{u}}^{II}\frac{\mu(\tilde{\theta})^{2}}{\mu(\overline{\theta})\sqrt{\tilde{\theta}}}\mathbf{u}\cdot\mathbf{F}(\varphi',\mathbf{u}',\zeta')d\mathbf{x}}_{I_{14}}+\underbrace{\int_{\Omega}-R\tilde{\theta}\sqrt{\frac{2}{R}}\frac{2}{3R}a_{\mathbf{u}}^{I}a_{\mathbf{u}}^{II}\frac{\mu(\tilde{\theta})^{2}}{\mu(\overline{\theta})\sqrt{\tilde{\theta}}}\mathbf{u}\cdot\nabla\varphi'd\mathbf{x}}_{I_{15}} \\
        &+\underbrace{\int_{\Omega}\frac{32}{75R^2}\sqrt{\frac{2}{R}}a_{\theta}^{I}a_{\theta}^{II}\frac{\kappa(\tilde{\theta})^{2}}{\kappa(\overline{\theta})\tilde{\theta}\sqrt{\tilde{\theta}}}\zeta G(\varphi',\mathbf{u}',\zeta')d\mathbf{x}}_{I_{16}}+\underbrace{\int_{\Omega}-\frac{32}{75R^2}\sqrt{\frac{2}{R}}a_{\theta}^{I}a_{\theta}^{II}\frac{\kappa(\tilde{\theta})^{2}}{\kappa(\overline{\theta})\tilde{\theta}\sqrt{\tilde{\theta}}}\mathcal{R}\zeta\div(\varphi'\mathbf{u}')d\mathbf{x}}_{I_{17}} \\
        &+\underbrace{\int_{\Omega}-\frac{4}{5R^2}a_{\mathbf{u}}^{II}a_{\theta}^I\frac{\mu(\tilde{\theta})\kappa(\tilde{\theta})}{\mu(\overline{\theta})\tilde{\theta}}\zeta\mathfrak{n}F_{1}(\varphi',\mathbf{u}',\zeta')d\mathbf{x}}_{I_{18}} +\underbrace{\int_{\Omega}\frac{4}{5R}\overline{\theta}a_{\mathbf{u}}^{II}a_{\theta}^I\frac{\mu(\tilde{\theta})\kappa(\tilde{\theta})}{\mu(\overline{\theta})\tilde{\theta}}\zeta\mathfrak{n}\partial_{1}\varphi' d\mathbf{x}}_{I_{19}}.
    \end{align*}

    For $I_{1}$, $I_{2}$ and $I_{3}$, by H\"{o}lder's inequality and \eqref{assumption2} in assumption (A), 
    \begin{align*}
        I_{1}+I_{2}+I_{3}\geq&\left(\frac{1}{3R}\sqrt{\frac{2}{R}}a_{\mathbf{u}}^{I}a_{\mathbf{u}}^{II}\frac{\mu(\overline{\theta})^{2}}{\sqrt{\overline{\theta}}}-C\delta\right)\int_{\Omega}\left|\nabla\mathbf{u}+\nabla\mathbf{u}^{\mathsf{T}}-\frac{2}{3}\div\mathbf{u}\mathbb{I}_{3}\right|^2d\mathbf{x} \nonumber\\
        &+\left(\frac{32}{75R^2}\sqrt{\frac{2}{R}}a_{\theta}^{I}a_{\theta}^{II}\frac{\kappa(\overline{\theta})^{2}}{\overline{\theta}\sqrt{\overline{\theta}}}-C\delta\right)\int_{\Omega}|\nabla\zeta|^2d\mathbf{x} \nonumber\\
        &-\left(\frac{4}{5R^2}a_{\mathbf{u}}^{II}a_{\theta}^I\frac{\mu(\overline{\theta})\kappa(\overline{\theta})}{\mu(\overline{\theta})\overline{\theta}}+C\delta\right)\int_{\Omega}|\nabla\zeta|\cdot\left|\nabla\mathbf{u}+\nabla\mathbf{u}^{\mathsf{T}}-\frac{2}{3}\div\mathbf{u}\mathbb{I}_{3}\right|d\mathbf{x} \nonumber\\
        \geq&C_{0}(\|\nabla\mathbf{u}+\nabla\mathbf{u}'-\frac{2}{3}\div\mathbf{u}\mathbb{I}_{3}\|_{L^{2}(\Omega)}^2+\|\nabla\zeta\|_{L^{2}(\Omega)}^2),  
    \end{align*}
    where $C_{0}>0$ is a constant only depending on $a_{\mathbf{u}}^{I}, a_{\mathbf{u}}^{II}, a_{\theta}^{I}, a_{\theta}^{II}, R, \mu, \kappa, \overline{\theta}$. Then applying Korns' type inequality \eqref{KornInequality}, we have
    \begin{equation}\label{basicEstimateEllipticSystemI1I2I3}
        I_{1}+I_{2}+I_{3}\geq C_{0}(\|\nabla\mathbf{u}\|_{L^{2}(\Omega)}^{2}+\|\nabla\zeta\|_{L^{2}}^{2}).
    \end{equation}

    For $I_{4}$, $I_{5}$ and $I_{6}$, by H\"{o}lder's inequality
    \begin{equation}\label{basicEstimateEllipticSystemI4I5I6}
        |I_{4}|+|I_{5}|+|I_{6}|\leq C_{1}\delta(\|\mathbf{u}\|_{H^{1}(\Omega)}^2+\|\zeta\|_{H^{1}(\Omega)}^2),
    \end{equation}
    where $C_{1}>0$ is a constant only depending on $a_{\mathbf{u}}^{I}, a_{\mathbf{u}}^{II}, a_{\theta}^{I}, a_{\theta}^{II}, R, \mu, \kappa, \overline{\theta}$.

    For $I_{7}$ and $I_{8}$, recall from \eqref{def.crad} that there is $\mathcal{R}\in \R$ such that
    \begin{equation}\label{basicEstimateEllipticSystemI7I8}
        R\overline{\rho}\sqrt{\frac{2}{R}}\frac{2}{3R}a_{\mathbf{u}}^{I}a_{\mathbf{u}}^{II}\mu(\overline{\theta})^{2}\overline{\theta}^{-\frac{1}{2}}=R(\overline{\rho}\overline{\theta}+\mathcal{R}\overline{\rho})\frac{32}{75R^2}\sqrt{\frac{2}{R}}a_{\theta}^{I}a_{\theta}^{II}\mu(\overline{\theta})\kappa(\overline{\theta})\overline{\theta}^{-\frac{3}{2}}.
    \end{equation}
    Therefore, it follows by H\"{o}lder's inequality,
    \begin{align*}
        |I_{7}+I_{8}|\leq&C_{2}\delta\left(\int_{\Omega}|\mathbf{u}\cdot\nabla\zeta|d\mathbf{x}+\int_{\Omega}|\zeta\div\mathbf{u}|d\mathbf{x}\right)\\
        \leq&C_{2}\delta\left(\|\mathbf{u}\|_{H^{1}(\Omega)}^{2}+\|\zeta\|_{H^{1}(\Omega)}^{2}\right),
    \end{align*}
    where $C_{2}>0$ is a constant only depending on $a_{\mathbf{u}}^{I},a_{\mathbf{u}}^{II},a_{\theta}^{I},a_{\theta}^{II}, R, \mu, \kappa, \overline{\rho},\overline{\theta}$. 

    For $I_{9}$, by divergence theorem, 
    \begin{align}\label{basicEstimateEllipticSystemI9}
        I_{9}=&\int_{\Omega}\frac{4}{5R}\overline{\rho}a_{\mathbf{u}}^{II}a_{\theta}^I\frac{\mu(\tilde{\theta})\kappa(\tilde{\theta})}{\mu(\overline{\theta})\tilde{\theta}}\mathfrak{n}|\zeta|^2 d\mathbf{x}+\int_{\Omega}\frac{2}{5R}\overline{\rho}a_{\mathbf{u}}^{II}a_{\theta}^I\partial_{1}\left(\frac{\mu(\tilde{\theta})\kappa(\tilde{\theta})}{\mu(\overline{\theta})\tilde{\theta}}\right)\cdot\mathfrak{n}|\zeta|^2 d\mathbf{x} \nonumber\\
        &+\sum_{p=0,1}\int_{\mathbb{T}_{p}^{2}}-\frac{2}{5R}\overline{\rho}a_{\mathbf{u}}^{II}a_{\theta}^I\frac{\mu(\theta_{w})\kappa(\theta_{w})}{\mu(\overline{\theta})\theta_{w}}|\zeta|^2 dx_{2}dx_{3} \nonumber\\
        \geq&\left(\frac{4}{5R}a_{\mathbf{u}}^{II}a_{\theta}^I\overline{\rho}\frac{\kappa(\overline{\theta})}{\overline{\theta}}-C_{3}\delta\right)\|\zeta\|_{L^{2}(\Omega)}^2-\left(\frac{2}{5R}a_{\mathbf{u}}^{II}a_{\theta}^I\frac{\kappa(\overline{\theta})}{\overline{\theta}}\overline{\rho}+C_{3}\delta\right)\|\zeta\|_{L^{2}(\partial\Omega)}^2,
    \end{align}
    where $C_{3}>0$ is a constant only depending on $a_{\mathbf{u}}^{I}, a_{\mathbf{u}}^{II}, a_{\theta}^{I}, a_{\theta}^{II}, R, \mu, \kappa, \overline{\rho},\overline{\theta}$. 

    For $I_{10}$, $I_{11}$, $I_{12}$ and $I_{13}$,
    \begin{align}\label{basicEstimateEllipticSystemI10I11I12I13}
        \sum_{i=10}^{13}I_{i}\geq&\frac{2}{3R}a_{\mathbf{u}}^{II}\mu(\overline{\theta})(\overline{\rho}-C_{4}\delta-C_{4}\varepsilon)\|\mathbf{u}\|_{L^{2}(\partial\Omega)}^2 \nonumber\\
        &+\frac{16}{15R}a_{\theta}^{I}\frac{\kappa(\overline{\theta})}{\overline{\theta}}(\overline{\rho}-C_{4}\delta-C_{4}\varepsilon)\|\zeta\|_{L^{2}(\partial\Omega)}^2,
    \end{align}
    where $C_{4}>0$ is a constant only depending on $a_{\mathbf{u}}^{I}, a_{\mathbf{u}}^{II}, a_{\theta}^{I},a_{\theta}^{II}, R, \mu, \kappa, \overline{\rho},\overline{\theta}$.

    For $I_{14}$, $I_{15}$, $I_{16}$, $I_{17}$, $I_{18}$ and $I_{19}$, by H\"{o}lder's inequality, 
    \begin{equation}\label{basicEstimateEllipticSystemI14I15I16I17I18I19}
        \sum_{i=14}^{19}|I_{i}|\leq\frac{c}{2}(\|\mathbf{u}\|_{L^{2}(\Omega)}^2+\|\zeta\|_{L^{2}(\Omega)}^2)+C_{5}(\|\partial_{1}\varphi'\|_{L^{2}(\Omega)}^2+\|\div(\varphi'\mathbf{u}')\|_{L^{2}(\Omega)}^2+\|\mathbf{F'}\|_{L^{2}(\Omega)}^2+\|G'\|_{L^{2}(\Omega)}^2),
    \end{equation}
    where $C_{5}>0$ is a constant only depending on $a_{\mathbf{u}}^{I}, a_{\mathbf{u}}^{II}, a_{\theta}^{I}, a_{\theta}^{II}, R, \mu, \kappa, \overline{\theta}$.

    Combining \eqref{basicEstimateEllipticSystemI1I2I3}, \eqref{basicEstimateEllipticSystemI4I5I6}, \eqref{basicEstimateEllipticSystemI7I8}, \eqref{basicEstimateEllipticSystemI9}, \eqref{basicEstimateEllipticSystemI10I11I12I13}, \eqref{basicEstimateEllipticSystemI14I15I16I17I18I19} and Lemma \ref{EstimateFGLemma}, by \eqref{assumption1} in assumption (A), then choosing $\delta, \varepsilon>0$ sufficiently small, we obtain \eqref{basicEstimateEllipticSystem} and thus complete the proof of Lemma \ref{coervivityEllipticSystem}.
\end{proof}

\begin{remark}\label{remarkBoundaryEffect}
    As the aforementioned part in the introduction, the slip boundary conditions \eqref{BC} involves $\nabla\mathbf{u}$ and $\nabla\zeta$. Therefore if one applies the usual energy estimate, that is, $\int_{\Omega}\mathbf{u}\cdot\eqref{LSNS}_{2}d\mathbf{x}+\int_{\Omega}\zeta\cdot\eqref{LSNS}_{3}d\mathbf{x}$, then some boundary terms with $\nabla\mathbf{u}$ and $\nabla\zeta$ would be generated. Hence by the Sobolev's embedding, these terms indicate $(\mathbf{u},\theta)\in H^{\frac{3}{2}}(\Omega)\times H^{\frac{3}{2}}(\Omega)$. However, by the dissipation induced from the viscosity and heat conductivity, the solution can only be in $H^{1}(\Omega)$, which makes the usual energy estimate method break down. Our idea is to design new test functions by full use of the structure of the system and the boundary conditions. Then some cancellation effects are observed making the boundary terms involving the derivatives of velocity and temperature vanish. This is one of the key ideas in the construction of the variational form.
\end{remark}

\begin{remark}\label{remarkHyperbolicEffect}
    It is well-known that the system \eqref{PSNS} is hyperbolic-elliptic. For the basic energy estimate, it is important to deal with the linear convection terms $R\overline{\theta}\nabla\varphi$, $R\overline{\rho}\nabla\zeta$, and this is usually done by making a suitable combination of the estimates for each equation of \eqref{PSNS} in the past literature. However, because of the design of new test functions, we have to make a seemingly more complicated combination, which is achieved by introducing the constant $\mathcal{R}$ as in \eqref{def.crad}.
\end{remark}

Applying Theorem \ref{coervivityEllipticSystem}, it is straightforward to obtain the following weak existence result by Lax-Milgram theorem.
\begin{corollary}\label{weakExistenceEllipticSystem}
    Under the assumption of Theorem \ref{mainTheorem}, there exists a constant $\delta_{0}>0$ such that for $\delta\in[0,\delta_{0})$, if the wall temperature $\theta_{w}\in H^{\vartheta}(\partial\Omega)$ with $\|\theta_{w}-\overline{\theta}\|_{H^{\vartheta}(\partial\Omega)}\leq \delta$ where $\vartheta>\frac{7}{2}$, and $(\varphi',\mathbf{u}',\zeta')\in (H^{2}(\Omega)\cap L_{0}^{2}(\Omega))\times V^{3}(\Omega)\times H^{3}(\Omega)$ with $N_{2}(\varphi',\mathbf{u}',\zeta')\leq\varepsilon$ such that $\|\varphi'\|_{L^{\infty}(\Omega)}+\overline{\rho}>0$ and $\|\zeta'\|_{L^{\infty}(\Omega)}+\overline{\theta}>0$, then the boundary value problem $\eqref{LSNS}_{2}$ and $\eqref{LSNS}_{3}$ with \eqref{LBC} has a unique weak solution $(u,\zeta)\in V^{1}(\Omega)\times H^{1}(\Omega)$ in the sense of \eqref{weakFormEllipticSystem}. 
\end{corollary}

\subsection{Regularity results for $\mathbf{u}$ and $\zeta$}
Now we try to prove the following regularity result of the weak solution. 

\begin{lemma}\label{regularityEllipticSystem}
    Under the assumption of Theorem \ref{mainTheorem}, suppose $(u,\zeta)\in V^{1}(\Omega)\times H^{1}(\Omega)$ is a weak solution in the sense of \eqref{weakFormEllipticSystem}. Then there exist constants $\delta, \varepsilon_{0}>0$ such that for $\delta\in[0,\varepsilon_{0}), \varepsilon\in[0,\varepsilon_{0})$, if the wall temperature $\theta_{w}\in H^{\vartheta}(\partial\Omega)$ with $\|\theta_{w}-\overline{\theta}\|_{H^{\vartheta}(\partial\Omega)}\leq \delta$ where $\vartheta>\frac{7}{2}$, and $(\varphi',\mathbf{u}',\zeta')\in (H^{2}(\Omega)\cap L_{0}^{2}(\Omega))\times V^{3}(\Omega)\times H^{3}(\Omega)$ with $N_{2}(\varphi',\mathbf{u}',\zeta')\leq\varepsilon$ such that $\|\varphi'\|_{L^{\infty}(\Omega)}+\overline{\rho}>0$ and $\|\zeta'\|_{L^{\infty}(\Omega)}+\overline{\theta}>0$, then $(u,\zeta)\in V^{3}(\Omega)\times H^{3}(\Omega)$ and there exists a constant $C>0$ depends only on $a_{\mathbf{u}}^{I}, a_{\mathbf{u}}^{II}, a_{\theta}^{I}, a_{\theta}^{II}, c_{v}, R, \mu, \kappa, \overline{\rho}, \overline{\theta}$ such that
    \begin{equation}\label{regularityEstimateEllipticSystem}
        \|\mathbf{u}\|_{H^{3}(\Omega)}^2 +\|\zeta\|_{H^{3}(\Omega)}^2\leq CN_{2}(\varphi',\mathbf{u}',\zeta')^{2}(1+N_{2}(\varphi',\mathbf{u}',\zeta'))^{4}.
    \end{equation} 
\end{lemma}

\begin{proof}
    Let us denote the $i$-difference quotient of sufficiently small size $d>0$ for a function $f$ by
    \begin{equation*}
        D_{i,d}f(\mathbf{x}):=\frac{f(\mathbf{x}+d\mathbf{e}_{i})-f(\mathbf{x})}{d},
    \end{equation*}
    for $i=2,3$. 

    As a first step, we prove that $(\mathbf{u},\zeta)\in H^{2}(\Omega)\cap H^{2}(\Omega)$. Taking $\mathbf{v}=-D_{i,-d}(D_{i,d}\mathbf{u})$ and $\xi=-D_{i,-d}(D_{i,d}\zeta)$ for $i=2,3$, then 
    \begin{align*}
        &\underbrace{\int_{\Omega}\sqrt{\frac{2}{R}}\frac{2}{3R}a_{\mathbf{u}}^{I}a_{\mathbf{u}}^{II}\frac{\mu(\tilde{\theta})^{2}}{\sqrt{\tilde{\theta}}}D_{i,d}\nabla\mathbf{u}:D_{i,d}\left(\nabla\mathbf{u}+\nabla\mathbf{u}^{\mathsf{T}}-\frac{2}{3}\div\mathbf{u}\mathbb{I}_{3}\right)d\mathbf{x}}_{I_{1}} +\underbrace{\int_{\Omega}\frac{32}{75R^2}\sqrt{\frac{2}{R}}a_{\theta}^{I}a_{\theta}^{II}\frac{\kappa(\tilde{\theta})^{2}}{\tilde{\theta}\sqrt{\tilde{\theta}}}|D_{i,d}\nabla\zeta|^{2}d\mathbf{x}}_{I_{2}} \nonumber\\
        &+\underbrace{\int_{\Omega}-\frac{4}{5R^2}a_{\mathbf{u}}^{II}a_{\theta}^I\frac{\mu(\tilde{\theta})\kappa(\tilde{\theta})}{\mu(\overline{\theta})\tilde{\theta}}D_{i,d}\nabla\zeta\cdot\mathfrak{n}D_{i,d}\left(\nabla u_{1}+\partial_{1}\mathbf{u}^{\mathsf{T}}-\frac{2}{3}\mathbf{e}_{1}\div\mathbf{u}\right)d\mathbf{x}}_{I_{3}} \nonumber\\
        &\underbrace{\int_{\Omega}\sqrt{\frac{2}{R}}\frac{2}{3R}a_{\mathbf{u}}^{I}a_{\mathbf{u}}^{II}D_{i,d}\left(\frac{\mu(\tilde{\theta})^{2}}{\sqrt{\tilde{\theta}}}\right)\cdot D_{i,d}\nabla\mathbf{u}:\left(\nabla\mathbf{u}+\nabla\mathbf{u}^{\mathsf{T}}-\frac{2}{3}\div\mathbf{u}\mathbb{I}_{3}\right)d\mathbf{x}}_{I_{4}} \nonumber\\
        &+\underbrace{\int_{\Omega}\frac{32}{75R^2}\sqrt{\frac{2}{R}}a_{\theta}^{I}a_{\theta}^{II}D_{i,d}\left(\frac{\kappa(\tilde{\theta})^{2}}{\tilde{\theta}\sqrt{\tilde{\theta}}}\right)\cdot D_{i,d}\nabla\zeta\cdot\nabla\zeta\mathbf{x}}_{I_{5}} \nonumber\\
        &+\underbrace{\int_{\Omega}-\frac{4}{5R^2}a_{\mathbf{u}}^{II}a_{\theta}^ID_{i,d}\left(\frac{\mu(\tilde{\theta})\kappa(\tilde{\theta})}{\mu(\overline{\theta})\tilde{\theta}}\right)\cdot D_{i,d}\nabla\zeta\cdot\mathfrak{n}\left(\nabla u_{1}+\partial_{1}\mathbf{u}^{\mathsf{T}}-\frac{2}{3}\mathbf{e}_{1}\div\mathbf{u}\right)d\mathbf{x}}_{I_{6}} \nonumber\\
        &+\underbrace{\int_{\Omega}\sqrt{\frac{2}{R}}\frac{2}{3R}a_{\mathbf{u}}^{I}a_{\mathbf{u}}^{II}\sum_{0\leq j\leq1}D_{i,d}^{1-j}\nabla\left(\frac{\mu(\tilde{\theta})^{2}}{\sqrt{\tilde{\theta}}}\right)\otimes D_{i,d}\mathbf{u}:D_{i,d}^{j}\left(\nabla\mathbf{u}+\nabla\mathbf{u}^{\mathsf{T}}-\frac{2}{3}\div\mathbf{u}\mathbb{I}_{3}\right)d\mathbf{x}}_{I_{7}} \nonumber\\ &+\underbrace{\int_{\Omega}R\overline{\rho}\sqrt{\frac{2}{R}}\frac{2}{3R}a_{\mathbf{u}}^{I}a_{\mathbf{u}}^{II}\sum_{0\leq j\leq1}D_{i,d}^{1-j}\nabla\left(\frac{\mu(\tilde{\theta})^{2}}{\mu(\overline{\theta})\sqrt{\tilde{\theta}}}\right)\cdot D_{i,d}\zeta\cdot D_{i,d}^{j}\nabla\zeta d\mathbf{x}}_{I_{8}} \nonumber\\
        &+\underbrace{\int_{\Omega}-\frac{4}{5R^2}a_{\mathbf{u}}^{II}a_{\theta}^I\sum_{0\leq j\leq1}D_{i,d}^{1-j}\nabla\left(\frac{\mu(\tilde{\theta})\kappa(\tilde{\theta})}{\mu(\overline{\theta})\tilde{\theta}}\right)\cdot D_{i,d}\zeta\cdot\mathfrak{n}D_{i,d}^{j}\left(\nabla u_{1}+\partial_{1}\mathbf{u}^{\mathsf{T}}-\frac{2}{3}\mathbf{e}_{1}\div\mathbf{u}\right)d\mathbf{x}}_{I_{9}} \nonumber\\
        &+\underbrace{\int_{\Omega}R\overline{\rho}\sqrt{\frac{2}{R}}\frac{2}{3R}a_{\mathbf{u}}^{I}a_{\mathbf{u}}^{II}\frac{\mu(\tilde{\theta})^{2}}{\mu(\overline{\theta})\sqrt{\tilde{\theta}}}D_{i,d}\mathbf{u}\cdot D_{i,d}\nabla\zeta d\mathbf{x}}_{I_{10}} \nonumber\\
        &+\underbrace{\int_{\Omega}\frac{32}{75R^2}\sqrt{\frac{2}{R}}a_{\theta}^{I}a_{\theta}^{II}\frac{\kappa(\tilde{\theta})^{2}}{\kappa(\overline{\theta})\tilde{\theta}\sqrt{\tilde{\theta}}}(R\overline{\rho}\overline{\theta}+\mathcal{R}\overline{\rho})D_{i,d}\xi \cdot D_{i,d}\div\mathbf{u}d\mathbf{x}}_{I_{11}} \nonumber\\
        &+\underbrace{\int_{\Omega}-\frac{4}{5R}\overline{\rho}a_{\mathbf{u}}^{II}a_{\theta}^I\frac{\mu(\tilde{\theta})\kappa(\tilde{\theta})}{\mu(\overline{\theta})\tilde{\theta}}\mathfrak{n}D_{i,d}\zeta\cdot D_{i,d}\partial_{1}\zeta d\mathbf{x}}_{I_{12}} \nonumber\\
        &+\underbrace{\int_{\Omega}R\overline{\rho}\sqrt{\frac{2}{R}}\frac{2}{3R}a_{\mathbf{u}}^{I}a_{\mathbf{u}}^{II}D_{i,d}\left(\frac{\mu(\tilde{\theta})^{2}}{\mu(\overline{\theta})\sqrt{\tilde{\theta}}}\right)D_{i,d}\mathbf{u}\cdot\nabla\zeta d\mathbf{x}}_{I_{13}} \nonumber\\ &+\underbrace{\int_{\Omega}\frac{32}{75R^2}\sqrt{\frac{2}{R}}a_{\theta}^{I}a_{\theta}^{II}D_{i,d}\left(\frac{\kappa(\tilde{\theta})^{2}}{\kappa(\overline{\theta})\tilde{\theta}\sqrt{\tilde{\theta}}}\right)(R\overline{\rho}\overline{\theta}+\mathcal{R}\overline{\rho})D_{i,d}\zeta \cdot\div\mathbf{u}d\mathbf{x}}_{I_{14}} \nonumber\\
        &+\underbrace{\int_{\Omega}-\frac{4}{5R}\overline{\rho}a_{\mathbf{u}}^{II}a_{\theta}^ID_{i,d}\left(\frac{\mu(\tilde{\theta})\kappa(\tilde{\theta})}{\mu(\overline{\theta})\tilde{\theta}}\right)\cdot\mathfrak{n}D_{i,d}\zeta\cdot\partial_{1}\zeta d\mathbf{x}}_{I_{15}} \nonumber\\
        &+\underbrace{\sum_{p=0,1}\int_{\mathbb{T}_{p}^{2}}\frac{2}{3R}a_{\mathbf{u}}^{II}\mu(\theta_{w})(\varphi'+\overline{\rho})|D_{i,d}\mathbf{u}|^2dx_{2}dx_{3}}_{I_{16}} +\underbrace{\sum_{p=0,1}\int_{\mathbb{T}_{p}^{2}}\frac{16}{15R}a_{\theta}^{I}\frac{\kappa(\theta_{w})}{\theta_{w}}(\varphi'+\overline{\rho})|D_{i,d}\zeta|^2dx_{2}dx_{3}}_{I_{17}} \nonumber \\
        &+\underbrace{\sum_{p=0,1}\int_{\mathbb{T}_{p}^{2}}\frac{2}{3R}a_{\mathbf{u}}^{II}(\varphi'+\overline{\rho})D_{i,d}\mu(\theta_{w})\cdot D_{i,d}\mathbf{u}\cdot\mathbf{u}dx_{2}dx_{3}}_{I_{18}} \nonumber\\
        &+\underbrace{\sum_{p=0,1}\int_{\mathbb{T}_{p}^{2}}\frac{16}{15R}a_{\theta}^{I}(\varphi'+\overline{\rho})D_{i,d}\left(\frac{\kappa(\theta_{w})}{\theta_{w}}\right)D_{i,d}\zeta\cdot \zeta dx_{2}dx_{3}}_{I_{19}} \nonumber \\
        &+\underbrace{\sum_{p=0,1}\int_{\mathbb{T}_{p}^{2}}\frac{8}{15R^2}a_{\mathbf{u}}^{II}a_{\theta}^{I}\sum_{0\leq j\leq1}D_{i,d}^{1-j}\left(\frac{\mu(\theta_{w})\kappa(\theta_{w})}{\theta_{w}}\right)\cdot D_{i,d}\mathbf{u}\cdot D_{i,d}^{j}\nabla\tilde{\theta}dx_{2}dx_{3}}_{I_{20}} \nonumber\\ &+\underbrace{\sum_{p=0,1}\int_{\mathbb{T}_{p}^{2}}\frac{32}{75R^2}a_{\theta}^{I}a_{\theta}^{II}\sum_{0\leq j\leq1}D_{i,d}^{1-j}\left(\frac{\kappa(\theta_{w})^2}{\theta_{w}\sqrt{\theta_{w}}}\right)\cdot\mathfrak{n}D_{i,d}\zeta\cdot D_{i,d}^{j}\partial_{1}\tilde{\theta}dx_{2}dx_{3}}_{I_{21}} \nonumber \\
        = 
        &\underbrace{\int_{\Omega}\sqrt{\frac{2}{R}}\frac{2}{3R}a_{\mathbf{u}}^{I}a_{\mathbf{u}}^{II}\frac{\mu(\tilde{\theta})^{2}}{\mu(\overline{\theta})\sqrt{\tilde{\theta}}}D_{i,-d}(D_{i,d}\mathbf{u})\cdot\mathbf{F}(\varphi',\mathbf{u}',\zeta')d\mathbf{x}}_{I_{22}} \nonumber\\
        &+\underbrace{\int_{\Omega}-R\tilde{\theta}\sqrt{\frac{2}{R}}\frac{2}{3R}a_{\mathbf{u}}^{I}a_{\mathbf{u}}^{II}\frac{\mu(\tilde{\theta})^{2}}{\mu(\overline{\theta})\sqrt{\tilde{\theta}}}D_{i,-d}(D_{i,d}\mathbf{u})\cdot\nabla\varphi'd\mathbf{x}}_{I_{23}} \nonumber \\
        &+\underbrace{\int_{\Omega}\frac{32}{75R^2}\sqrt{\frac{2}{R}}a_{\theta}^{I}a_{\theta}^{II}\frac{\kappa(\tilde{\theta})^{2}}{\kappa(\overline{\theta})\tilde{\theta}\sqrt{\tilde{\theta}}}D_{i,-d}(D_{i,d}\zeta)\cdot G(\varphi',\mathbf{u}',\zeta')d\mathbf{x}}_{I_{24}} \nonumber\\
        &+\underbrace{\int_{\Omega}-\frac{32}{75R^2}\sqrt{\frac{2}{R}}a_{\theta}^{I}a_{\theta}^{II}\frac{\kappa(\tilde{\theta})^{2}}{\kappa(\overline{\theta})\tilde{\theta}\sqrt{\tilde{\theta}}}\mathcal{R}D_{i,-d}(D_{i,d}\zeta)\cdot\div(\varphi'\mathbf{u}')d\mathbf{x}}_{I_{25}} \nonumber \\
        &+\underbrace{\int_{\Omega}-\frac{4}{5R^2}a_{\mathbf{u}}^{II}a_{\theta}^I\frac{\mu(\tilde{\theta})\kappa(\tilde{\theta})}{\mu(\overline{\theta})\tilde{\theta}}\mathfrak{n}D_{i,-d}(D_{i,d}\xi)\cdot F_{1}(\varphi',\mathbf{u}',\zeta')d\mathbf{x}}_{I_{26}} \nonumber\\
        &+\underbrace{\int_{\Omega}\frac{4}{5R}\overline{\theta}a_{\mathbf{u}}^{II}a_{\theta}^I\frac{\mu(\tilde{\theta})\kappa(\tilde{\theta})}{\mu(\overline{\theta})\tilde{\theta}}\mathfrak{n}D_{i,-d}(D_{i,d}\zeta)\cdot\partial_{1}\varphi'd\mathbf{x}}_{I_{27}}.
    \end{align*}

    For $I_{1}$, $I_{2}$ and $I_{3}$, by H\"{o}lder's inequality and \eqref{assumption2} in assumption (A),
    \begin{align*}
        I_{1}+I_{2}+I_{3}\geq&\left(\frac{1}{3R}\sqrt{\frac{2}{R}}a_{\mathbf{u}}^{I}a_{\mathbf{u}}^{II}\frac{\mu(\overline{\theta})^{2}}{\sqrt{\overline{\theta}}}-C\delta\right)\int_{\Omega}\left|D_{i,d}\left(\nabla\mathbf{u}+\nabla\mathbf{u}^{\mathsf{T}}-\frac{2}{3}\div\mathbf{u}\mathbb{I}_{3}\right)\right|^2d\mathbf{x} \nonumber\\
        &+\left(\frac{32}{75R^2}\sqrt{\frac{2}{R}}a_{\theta}^{I}a_{\theta}^{II}\frac{\kappa(\overline{\theta})^{2}}{\overline{\theta}\sqrt{\overline{\theta}}}-C\delta\right)\int_{\Omega}|D_{i,d}\nabla\zeta|^2d\mathbf{x} \nonumber\\
        &-\left(\frac{4}{5R^2}a_{\mathbf{u}}^{II}a_{\theta}^I\frac{\mu(\overline{\theta})\kappa(\overline{\theta})}{\mu(\overline{\theta})\overline{\theta}}+C\delta\right)\int_{\Omega}|D_{i,d}\nabla\zeta|\cdot\left|D_{i,d}\left(\nabla\mathbf{u}+\nabla\mathbf{u}^{\mathsf{T}}-\frac{2}{3}\div\mathbf{u}\mathbb{I}_{3}\right)\right|d\mathbf{x} \nonumber\\
        \geq&C_{0}'\left(\left\|D_{i,d}\left(\nabla\mathbf{u}+\nabla\mathbf{u}'-\frac{2}{3}\div\mathbf{u}\mathbb{I}_{3}\right)\right\|_{L^{2}(\Omega)}^2+\|D_{i,d}\nabla\zeta\|_{L^{2}(\Omega)}^2\right), 
    \end{align*}
    where $C_{0}'>0$ is a constant only depending on $a_{\mathbf{u}}^{I}, a_{\mathbf{u}}^{II}, a_{\theta}^{I}, a_{\theta}^{II}, R, \mu, \kappa, \overline{\theta}$. Then applying Korns' type inequality \eqref{KornInequality}, we have
    \begin{equation}\label{regularityEstimateEllipticSystemI1I2I3}
        I_{1}+I_{2}+I_{3}\geq C_{0}'(\|D_{i,d}\nabla\mathbf{u}\|_{L^{2}(\Omega)}^2+\|D_{i,d}\nabla\zeta\|_{L^{2}(\Omega)}^2).
    \end{equation}

    For $I_{4},...,I_{9}$, by H\"{o}lder's inequality, 
    \begin{equation}\label{regularityEstimateEllipticSystemI4I5I6}
        |I_{4}|+\cdots+|I_{9}|\leq C_{1}'\delta\sum_{0\leq j\leq1}(\|D_{i,d}^{j}\mathbf{u}\|_{H^{1}(\Omega)}^2+\|D_{i,d}^{j}\zeta\|_{H^{1}(\Omega)}^2),
    \end{equation}
    where $C_{1}'>0$ is a constant only depending on $a_{\mathbf{u}}^{I}, a_{\mathbf{u}}^{II}, a_{\theta}^{I}, a_{\theta}^{II}, R, \mu, \kappa, \overline{\theta}$.

    For $I_{10}$ and $I_{11}$, since
        \begin{equation*}
        R\overline{\rho}\sqrt{\frac{2}{R}}\frac{2}{3R}a_{\mathbf{u}}^{I}a_{\mathbf{u}}^{II}\mu(\overline{\theta})^{2}\overline{\theta}^{-\frac{1}{2}}=R(\overline{\rho}\overline{\theta}+\mathcal{R}\overline{\rho})\frac{32}{75R^2}\sqrt{\frac{2}{R}}a_{\theta}^{I}a_{\theta}^{II}\mu(\overline{\theta})\kappa(\overline{\theta})\overline{\theta}^{-\frac{3}{2}},
        \end{equation*}
   it follows by H\"{o}lder's inequality that
    \begin{align}\label{regularityEstimateEllipticSystemI7I8}
        |I_{10}+I_{11}|\leq&C_{2}'\delta\left(\int_{\Omega}|D_{i,d}\mathbf{u}\cdot D_{i,d}\nabla\zeta|d\mathbf{x}+\int_{\Omega}|D_{i,d}\zeta\cdot D_{i,d}\div\mathbf{u}|d\mathbf{x}\right) \nonumber\\
        \leq&C_{2}'\delta\left(\|D_{i,d}\mathbf{u}\|_{H^{1}(\Omega)}^{2}+\|D_{i,d}\zeta\|_{H^{1}(\Omega)}^{2}\right),
    \end{align}
    where $C_{2}'>0$ is a constant only depending on $a_{\mathbf{u}}^{I}, a_{\mathbf{u}}^{II}, a_{\theta}^{I}, a_{\theta}^{II}, R, \mu, \kappa, \overline{\rho}, \overline{\theta}$.

    For $I_{12}$, by divergence theorem, 
        \begin{align}\label{regularityEstimateEllipticSystemI9}
        I_{12}=&\int_{\Omega}\frac{4}{5R}\overline{\rho}a_{\mathbf{u}}^{II}a_{\theta}^I\frac{\mu(\tilde{\theta})\kappa(\tilde{\theta})}{\mu(\overline{\theta})\tilde{\theta}}|D_{i,d}\zeta|^2 d\mathbf{x}+\int_{\Omega}\frac{2}{5R}\overline{\rho}a_{\mathbf{u}}^{II}a_{\theta}^I\partial_{1}\left(\frac{\mu(\tilde{\theta})\kappa(\tilde{\theta})}{\mu(\overline{\theta})\tilde{\theta}}\right)\cdot\mathfrak{n}|D_{i,d}\zeta|^2 d\mathbf{x} \nonumber\\
        &+\sum_{p=0,1}\int_{\mathbb{T}_{p}^{2}}-\frac{2}{5R}\overline{\rho}a_{\mathbf{u}}^{II}a_{\theta}^I\frac{\mu(\theta_{w})\kappa(\theta_{w})}{\mu(\overline{\theta})\theta_{w}}|D_{i,d}\zeta|^2 dx_{2}dx_{3} \nonumber\\
        \geq&\left(\frac{4}{5R}a_{\mathbf{u}}^{II}a_{\theta}^I\overline{\rho}\frac{\kappa(\overline{\theta})}{\overline{\theta}}-C_{3}'\delta\right)\|D_{i,d}\zeta\|_{L^{2}(\Omega)}^2 -\left(\frac{2}{5R}a_{\mathbf{u}}^{II}a_{\theta}^I\frac{\kappa(\overline{\theta})}{\overline{\theta}}\overline{\rho}+C_{3}'\delta\right)\|D_{i,d}\zeta\|_{L^{2}(\partial\Omega)}^2,
    \end{align}
    where $C_{3}'>0$ is a constant only depending on $a_{\mathbf{u}}^{I}, a_{\mathbf{u}}^{II}, a_{\theta}^{I}, a_{\theta}^{II}, R, \mu, \kappa, \overline{\rho}, \overline{\theta}$.

    For $I_{13}$, $I_{14}$ and $I_{15}$, by H\"{o}lder's inequality, 
    \begin{equation}\label{regularityEstimateEllipticSystemI10I11I12}
        |I_{13}|+|I_{14}|+|I_{15}|\leq C_{4}'\varepsilon(\|D_{i,d}\mathbf{u}\|_{H^{1}(\Omega)}^{2}+\|D_{i,d}\zeta\|_{H^{1}(\Omega)}^{2}),
    \end{equation}
    where $C_{4}'>0$ is a constant only depending on $a_{\mathbf{u}}^{I}, a_{\mathbf{u}}^{II}, a_{\theta}^{I}, a_{\theta}^{II}, R, \mu, \kappa, \overline{\rho}, \overline{\theta}$.

    For $I_{16},...,I_{21}$,
    \begin{align}\label{regularityEstimateEllipticSystemI13I14I15I16I17I18}
        |I_{16}|+\cdots+|I_{21}|\geq&\frac{2}{3R}a_{\mathbf{u}}^{II}\mu(\overline{\theta})(\overline{\rho}-C_{5}'\delta-C_{5}'\varepsilon)\|D_{i,d}\mathbf{u}\|_{L^{2}(\partial\Omega)}^2 \nonumber\\
        &+\frac{16}{15R}a_{\theta}^{I}\frac{\kappa(\overline{\theta})}{\overline{\theta}}(\overline{\rho}-C_{5}'\delta-C_{5}'\varepsilon)\|D_{i,d}\zeta\|_{L^{2}(\partial\Omega)}^2, 
    \end{align}
    where $C_{5}'>0$ is a constant only depending on $a_{\mathbf{u}}^{I}, a_{\mathbf{u}}^{II}, a_{\theta}^{I}, a_{\theta}^{II}, R, \mu, \kappa, \overline{\rho}, \overline{\theta}$.

    For $I_{22},...,I_{27}$, by H\"{o}lder's inequality,
    \begin{align}\label{regularityEstimateEllipticSystemI19I20I21I22I23I24}
        |I_{2}|+\cdots+|I_{27}|\leq&\frac{C_{0}'}{2}(\|D_{i,-d}(D_{i,d}\mathbf{u})\|_{L^{2}(\Omega)}^{2}+\|D_{i,-d}(D_{i,d}\zeta)\|_{L^{2}(\Omega)}^{2}) \nonumber\\
        &+C_{6}'(\|\nabla\varphi'\|_{L^{2}(\Omega)}^{2}+\|\mathbf{F'}\|_{L^{2}(\Omega)}^{2}+\|G'\|_{L^{2}(\Omega)}^{2}+\|\div(\varphi'\mathbf{u}')\|_{L^{2}(\Omega)}^{2}),
    \end{align}
    where $C_{6}'>0$ is a constant only depending on $a_{\mathbf{u}}^{I}, a_{\mathbf{u}}^{II}, a_{\theta}^{I}, a_{\theta}^{II}, R, \mu, \kappa, \overline{\theta}$.

    Combining \eqref{regularityEstimateEllipticSystemI1I2I3}, \eqref{regularityEstimateEllipticSystemI4I5I6}, \eqref{regularityEstimateEllipticSystemI7I8}, \eqref{regularityEstimateEllipticSystemI9}, \eqref{regularityEstimateEllipticSystemI10I11I12}, \eqref{regularityEstimateEllipticSystemI13I14I15I16I17I18}, \eqref{regularityEstimateEllipticSystemI19I20I21I22I23I24} and Lemma \ref{EstimateFGLemma}, by \eqref{assumption1} in assumption (A), then choosing $\delta, \varepsilon, d>0$ sufficiently small, we have
    \begin{equation}\label{tangentialH2RegularityEstimateEllipticSystem}
        \sum_{i=2,3}\|\partial_{i}\mathbf{u}\|_{H^{1}(\Omega)}^2 +\sum_{i=2,3}\|\partial_{i}\zeta\|_{H^{1}(\Omega)}^2\leq  CN_{1}(\varphi',\mathbf{u}',\zeta')^{2}(1+N_{2}(\varphi',\mathbf{u}',\zeta'))^{4}.
    \end{equation}
    Moreover, from $\eqref{LSNS}_{2}$, we have
    \begin{equation*}
        -\frac{4}{3}\mu(\overline{\theta})\partial_{1}^{2}u_{1}=\mu(\overline{\theta})\sum_{j=2,3}\partial_{j}^{2}u_{1}+\mu(\overline{\theta})\sum_{j=2,3}\partial_{j}\partial_{1}u_{j}-R\overline{\rho}\partial_{1}\zeta+F_{1}(\varphi',\mathbf{u}',\zeta')-R\overline{\theta}\partial_{1}\varphi',
    \end{equation*}
    and  
    \begin{equation*}
        -\mu(\overline{\theta})\partial_{1}^{2}u_{i}=\mu(\overline{\theta})\sum_{j=2,3}\partial_{j}^{2}u_{i}+\mu(\overline{\theta})\sum_{1\leq j\leq3}\partial_{j}\partial_{i}u_{j}-\frac{2}{3}\partial_{i}\div\mathbf{u}-R\overline{\rho}\partial_{i}\zeta+F_{i}(\varphi',\mathbf{u}',\zeta')-R\overline{\theta}\partial_{i}\varphi',
    \end{equation*}
    for $i=2,3$. Similarly, from $\eqref{LSNS}_{3}$,
    \begin{equation*}
        -\kappa(\overline{\theta})\partial_{1}^{2}\zeta=\kappa(\overline{\theta})\sum_{i=2,3}\partial_{i}^{2}\zeta-(R\overline{\rho}\overline{\theta}+\mathcal{R}\overline{\rho})\div\mathbf{u}+G(\varphi',\mathbf{u}',\zeta')-\mathcal{R}\div(\varphi'\mathbf{u}').
    \end{equation*}
    Together with \eqref{basicEstimateEllipticSystem}, \eqref{tangentialH2RegularityEstimateEllipticSystem}, we have 
    \begin{equation}\label{H2RegularityEstimateEllipticSystem}
        \|\mathbf{u}\|_{H^{2}(\Omega)}+\|\zeta\|_{H^{2}(\Omega)}\leq  CN_{1}(\varphi',\mathbf{u}',\zeta')^{2}(1+N_{2}(\varphi',\mathbf{u}',\zeta'))^{4},
    \end{equation}
    which implies $(\mathbf{u},\zeta)\in H^{2}(\Omega)\times H^{2}(\Omega)$.

    Let us prove that $(\mathbf{u},\zeta)\in H^{3}(\Omega)\cap H^{3}(\Omega)$. Taking $\mathbf{v}=-D_{i,-d}^{2}(D_{i,d}^{2}\mathbf{u})$ and $\xi=-D_{i,-d}^{2}(D_{i,d}^{2}\zeta)$ for $i=2,3$, then
    \begin{align*}
        &\underbrace{\int_{\Omega}\sqrt{\frac{2}{R}}\frac{2}{3R}a_{\mathbf{u}}^{I}a_{\mathbf{u}}^{II}\frac{\mu(\tilde{\theta})^{2}}{\sqrt{\tilde{\theta}}}D_{i,d}^{2}\nabla\mathbf{u}:D_{i,d}^{2}\left(\nabla\mathbf{u}+\nabla\mathbf{u}^{\mathsf{T}}-\frac{2}{3}\div\mathbf{u}\mathbb{I}_{3}\right)d\mathbf{x}}_{J_{1}} \nonumber\\ 
        &+\underbrace{\int_{\Omega}\frac{32}{75R^2}\sqrt{\frac{2}{R}}a_{\theta}^{I}a_{\theta}^{II}\frac{\kappa(\tilde{\theta})^{2}}{\tilde{\theta}\sqrt{\tilde{\theta}}}|D_{i,d}^{2}\nabla\zeta|^{2}d\mathbf{x}}_{J_{2}} \nonumber\\
        &+\underbrace{\int_{\Omega}-\frac{4}{5R^2}a_{\mathbf{u}}^{II}a_{\theta}^I\frac{\mu(\tilde{\theta})\kappa(\tilde{\theta})}{\mu(\overline{\theta})\tilde{\theta}}D_{i,d}^{2}\nabla\zeta\cdot\mathfrak{n}D_{i,d}^{2}\left(\nabla u_{1}+\partial_{1}\mathbf{u}^{\mathsf{T}}-\frac{2}{3}\mathbf{e}_{1}\div\mathbf{u}\right)d\mathbf{x}}_{J_{3}} \nonumber\\
        &\underbrace{\int_{\Omega}\sqrt{\frac{2}{R}}\frac{2}{3R}a_{\mathbf{u}}^{I}a_{\mathbf{u}}^{II}\sum_{0\leq j\leq1}D_{i,d}^{2-j}\left(\frac{\mu(\tilde{\theta})^{2}}{\sqrt{\tilde{\theta}}}\right)\cdot D_{i,d}^{2}\nabla\mathbf{u}:D_{i,d}^{j}\left(\nabla\mathbf{u}+\nabla\mathbf{u}^{\mathsf{T}}-\frac{2}{3}\div\mathbf{u}\mathbb{I}_{3}\right)d\mathbf{x}}_{J_{4}} \nonumber\\ 
        &+\underbrace{\int_{\Omega}\frac{32}{75R^2}\sqrt{\frac{2}{R}}a_{\theta}^{I}a_{\theta}^{II}\sum_{0\leq j\leq1}D_{i,d}^{2-j}\left(\frac{\kappa(\tilde{\theta})^{2}}{\tilde{\theta}\sqrt{\tilde{\theta}}}\right)\cdot D_{i,d}^{2}\nabla\zeta\cdot D_{i,d}^{j}\nabla\zeta d\mathbf{x}}_{J_{5}} \nonumber\\
        &+\underbrace{\int_{\Omega}-\frac{4}{5R^2}a_{\mathbf{u}}^{II}a_{\theta}^I\frac{\mu(\tilde{\theta})\kappa(\tilde{\theta})}{\mu(\overline{\theta})\tilde{\theta}}D_{i,d}^{2}\nabla\zeta\cdot\mathfrak{n}D_{i,d}^{2}\left(\nabla u_{1}+\partial_{1}\mathbf{u}^{\mathsf{T}}-\frac{2}{3}\mathbf{e}_{1}\div\mathbf{u}\right)d\mathbf{x}}_{J_{6}} \nonumber\\
        &+\underbrace{\int_{\Omega}\sqrt{\frac{2}{R}}\frac{2}{3R}a_{\mathbf{u}}^{I}a_{\mathbf{u}}^{II}\sum_{0\leq j\leq2}D_{i,d}^{2-j}\nabla\left(\frac{\mu(\tilde{\theta})^{2}}{\sqrt{\tilde{\theta}}}\right)\otimes D_{i,d}^{2}\mathbf{u}:D_{i,d}^{j}\left(\nabla\mathbf{u}+\nabla\mathbf{u}^{\mathsf{T}}-\frac{2}{3}\div\mathbf{u}\mathbb{I}_{3}\right)d\mathbf{x}}_{J_{7}} \nonumber\\ &+\underbrace{\int_{\Omega}R\overline{\rho}\sqrt{\frac{2}{R}}\frac{2}{3R}a_{\mathbf{u}}^{I}a_{\mathbf{u}}^{II}\sum_{0\leq j\leq2}D_{i,d}^{2-j}\nabla\left(\frac{\mu(\tilde{\theta})^{2}}{\mu(\overline{\theta})\sqrt{\tilde{\theta}}}\right)\cdot D_{i,d}\mathbf{u}\cdot D_{i,d}^{j}\nabla\zeta d\mathbf{x}}_{J_{8}} \nonumber\\
        &+\underbrace{\int_{\Omega}-\frac{4}{5R^2}a_{\mathbf{u}}^{II}a_{\theta}^I\sum_{0\leq j\leq2}D_{i,d}^{2-j}\nabla\left(\frac{\mu(\tilde{\theta})\kappa(\tilde{\theta})}{\mu(\overline{\theta})\tilde{\theta}}\right)\cdot D_{i,d}\zeta\cdot\mathfrak{n}D_{i,d}^{j}\left(\nabla u_{1}+\partial_{1}\mathbf{u}^{\mathsf{T}}-\frac{2}{3}\mathbf{e}_{1}\div\mathbf{u}\right)d\mathbf{x}}_{J_{9}} \nonumber\\
        &+\underbrace{\int_{\Omega}R\overline{\rho}\sqrt{\frac{2}{R}}\frac{2}{3R}a_{\mathbf{u}}^{I}a_{\mathbf{u}}^{II}\frac{\mu(\tilde{\theta})^{2}}{\mu(\overline{\theta})\sqrt{\tilde{\theta}}}D_{i,d}^{2}\mathbf{u}\cdot D_{i,d}^{2}\nabla\zeta d\mathbf{x}}_{J_{10}} \nonumber\\
        &+\underbrace{\int_{\Omega}\frac{32}{75R^2}\sqrt{\frac{2}{R}}a_{\theta}^{I}a_{\theta}^{II}\frac{\kappa(\tilde{\theta})^{2}}{\kappa(\overline{\theta})\tilde{\theta}\sqrt{\tilde{\theta}}}(R\overline{\rho}\overline{\theta}+\mathcal{R}\overline{\rho})D_{i,d}^{2}\zeta \cdot D_{i,d}^{2}\div\mathbf{u}d\mathbf{x}}_{J_{11}} \nonumber\\
        &+\underbrace{\int_{\Omega}-\frac{4}{5R}\overline{\rho}a_{\mathbf{u}}^{II}a_{\theta}^I\frac{\mu(\tilde{\theta})\kappa(\tilde{\theta})}{\mu(\overline{\theta})\tilde{\theta}}\mathfrak{n}D_{i,d}^{2}\zeta\cdot D_{i,d}^{2}\partial_{1}\zeta d\mathbf{x}}_{J_{12}} \nonumber\\
        &+\underbrace{\int_{\Omega}R\overline{\rho}\sqrt{\frac{2}{R}}\frac{2}{3R}a_{\mathbf{u}}^{I}a_{\mathbf{u}}^{II}\sum_{0\leq j\leq2}D_{i,d}^{2-j}\left(\frac{\mu(\tilde{\theta})^{2}}{\mu(\overline{\theta})\sqrt{\tilde{\theta}}}\right)\cdot D_{i,d}^{2}\mathbf{u}\cdot D_{i,d}^{j}\nabla\zeta d\mathbf{x}}_{J_{13}} \nonumber\\
        &+\underbrace{\int_{\Omega}\frac{32}{75R^2}\sqrt{\frac{2}{R}}a_{\theta}^{I}a_{\theta}^{II}\sum_{0\leq j\leq2}D_{i,d}^{2-j}\left(\frac{\kappa(\tilde{\theta})^{2}}{\kappa(\overline{\theta})\tilde{\theta}\sqrt{\tilde{\theta}}}\right)\cdot(R\overline{\rho}\overline{\theta}+\mathcal{R}\overline{\rho})D_{i,d}^{2}\zeta \cdot D_{i,d}^{j}\div\mathbf{u}d\mathbf{x}}_{J_{11}} \nonumber\\
        &+\underbrace{\int_{\Omega}-\frac{4}{5R}\overline{\rho}a_{\mathbf{u}}^{II}a_{\theta}^I\sum_{0\leq j\leq2}D_{i,d}^{2-j}\left(\frac{\mu(\tilde{\theta})\kappa(\tilde{\theta})}{\mu(\overline{\theta})\tilde{\theta}}\right)\cdot\mathfrak{n}D_{i,d}^{2}\zeta\cdot D_{i,d}^{j}\partial_{1}\zeta d\mathbf{x}}_{J_{15}} \nonumber\\
        &+\underbrace{\sum_{p=0,1}\int_{\mathbb{T}_{p}^{2}}\frac{2}{3R}a_{\mathbf{u}}^{II}\mu(\theta_{w})(\varphi'+\overline{\rho})|D_{i,d}^{2}\mathbf{u}|^2dx_{2}dx_{3}}_{J_{16}} +\underbrace{\sum_{p=0,1}\int_{\mathbb{T}_{p}^{2}}\frac{16}{15R}a_{\theta}^{I}\frac{\kappa(\theta_{w})}{\theta_{w}}(\varphi'+\overline{\rho})|D_{i,d}^{2}\zeta|^2dx_{2}dx_{3}}_{J_{17}} \nonumber \\
        &+\underbrace{\sum_{p=0,1}\int_{\mathbb{T}_{p}^{2}}\frac{2}{3R}a_{\mathbf{u}}^{II}(\varphi'+\overline{\rho})\sum_{0\leq j\leq2}D_{i,d}^{2-j}\mu(\theta_{w})\cdot D_{i,d}^{2}\mathbf{u}\cdot D_{i,d}^{j}\mathbf{u}dx_{2}dx_{3}}_{J_{18}} \nonumber\\ &+\underbrace{\sum_{p=0,1}\int_{\mathbb{T}_{p}^{2}}\frac{16}{15R}a_{\theta}^{I}(\varphi'+\overline{\rho})\sum_{0\leq j\leq2}D_{i,d}^{2-j}\left(\frac{\kappa(\theta_{w})}{\theta_{w}}\right)\cdot D_{i,d}^{2}\zeta\cdot D_{i,d}^{j}\zeta dx_{2}dx_{3}}_{J_{19}} \nonumber \\
        &+\underbrace{\sum_{p=0,1}\int_{\mathbb{T}_{p}^{2}}\frac{8}{15R^2}a_{\mathbf{u}}^{II}a_{\theta}^{I}\sum_{0\leq j\leq2}D_{i,d}^{2-j}\left(\frac{\mu(\theta_{w})\kappa(\theta_{w})}{\theta_{w}}\right)\cdot D_{i,d}^{2}\mathbf{u}\cdot D_{i,d}^{j}\nabla\tilde{\theta}dx_{2}dx_{3}}_{J_{20}} \nonumber\\ &+\underbrace{\sum_{p=0,1}\int_{\mathbb{T}_{p}^{2}}\frac{32}{75R^2}a_{\theta}^{I}a_{\theta}^{II}\sum_{0\leq j\leq2}D_{i,d}^{2-j}\left(\frac{\kappa(\theta_{w})^2}{\theta_{w}\sqrt{\theta_{w}}}\right)\cdot\mathfrak{n}D_{i,d}^{2}\zeta\cdot D_{i,d}^{j}\partial_{1}\tilde{\theta}dx_{2}dx_{3}}_{J_{21}} \nonumber \\
        = 
        &\underbrace{\int_{\Omega}\sqrt{\frac{2}{R}}\frac{2}{3R}a_{\mathbf{u}}^{I}a_{\mathbf{u}}^{II}\sum_{0\leq j\leq1}D_{i,d}^{1-j}\left(\frac{\mu(\tilde{\theta})^{2}}{\mu(\overline{\theta})\sqrt{\tilde{\theta}}}\right)\cdot D_{i,-d}(D_{i,d}^{2}\mathbf{u})\cdot D_{i,d}^{j}\mathbf{F}(\varphi',\mathbf{u}',\zeta')d\mathbf{x}}_{J_{22}} \nonumber\\
        &+\underbrace{\int_{\Omega}-R\tilde{\theta}\sqrt{\frac{2}{R}}\frac{2}{3R}a_{\mathbf{u}}^{I}a_{\mathbf{u}}^{II}\sum_{0\leq j\leq1}D_{i,d}^{1-j}\left(\frac{\mu(\tilde{\theta})^{2}}{\mu(\overline{\theta})\sqrt{\tilde{\theta}}}\right)\cdot D_{i,-d}(D_{i,d}^{2}\mathbf{u})\cdot D_{i,d}^{j}\nabla\varphi'd\mathbf{x}}_{J_{23}} \nonumber \\
        &+\underbrace{\int_{\Omega}\frac{32}{75R^2}\sqrt{\frac{2}{R}}a_{\theta}^{I}a_{\theta}^{II}\sum_{0\leq j\leq1}D_{i,d}^{1-j}\left(\frac{\kappa(\tilde{\theta})^{2}}{\kappa(\overline{\theta})\tilde{\theta}\sqrt{\tilde{\theta}}}\right)\cdot D_{i,-d}(D_{i,d}^{2}\zeta)\cdot D_{i,d}^{j}G(\varphi',\mathbf{u}',\zeta')d\mathbf{x}}_{J_{24}} \nonumber\\
        &+\underbrace{\int_{\Omega}-\frac{32}{75R^2}\sqrt{\frac{2}{R}}a_{\theta}^{I}a_{\theta}^{II}\sum_{0\leq j\leq1}D_{i,d}^{1-j}\left(\frac{\kappa(\tilde{\theta})^{2}}{\kappa(\overline{\theta})\tilde{\theta}\sqrt{\tilde{\theta}}}\right)\cdot\mathcal{R}D_{i,-d}(D_{i,d}^{2}\zeta)\cdot D_{i,d}^{j}\div(\varphi'\mathbf{u}')d\mathbf{x}}_{J_{25}} \nonumber \\
        &+\underbrace{\int_{\Omega}-\frac{4}{5R^2}a_{\mathbf{u}}^{II}a_{\theta}^I\sum_{0\leq j\leq1}D_{i,d}^{1-j}\left(\frac{\mu(\tilde{\theta})\kappa(\tilde{\theta})}{\mu(\overline{\theta})\tilde{\theta}}\right)\cdot\mathfrak{n}D_{i,-d}(D_{i,d}^{2}\zeta)\cdot D_{i,d}^{j}F_{1}(\varphi',\mathbf{u}',\zeta')d\mathbf{x}}_{J_{26}} \nonumber\\
        &+\underbrace{\int_{\Omega}\frac{4}{5R}\overline{\theta}a_{\mathbf{u}}^{II}a_{\theta}^I\sum_{0\leq j\leq1}D_{i,d}^{1-j}\left(\frac{\mu(\tilde{\theta})\kappa(\tilde{\theta})}{\mu(\overline{\theta})\tilde{\theta}}\right)\cdot\mathfrak{n}D_{i,-d}(D_{i,d}^{2}\zeta)\cdot D_{i,d}^{j}\partial_{1}\varphi'd\mathbf{x}}_{J_{27}}.
    \end{align*}
    
    For $J_{1}$, $J_{2}$ and $J_{3}$, by H\"{o}lder's inequality and \eqref{assumption2} in assumption (A), 
    \begin{align*}
        J_{1}+J_{2}+J_{3}\geq&\left(\frac{2}{3R}\sqrt{\frac{2}{R}}a_{\mathbf{u}}^{I}a_{\mathbf{u}}^{II}\frac{\mu(\overline{\theta})^{2}}{\sqrt{\overline{\theta}}}-C\delta\right)\int_{\Omega}\left|D_{i,d}^{2}\left(\nabla\mathbf{u}+\nabla\mathbf{u}^{\mathsf{T}}-\frac{2}{3}\div\mathbf{u}\mathbb{I}_{3}\right)\right|^2d\mathbf{x} \nonumber\\
        &+\left(\frac{32}{75R^2}\sqrt{\frac{2}{R}}a_{\theta}^{I}a_{\theta}^{II}\frac{\kappa(\overline{\theta})^{2}}{\overline{\theta}\sqrt{\overline{\theta}}}-C\delta\right)\int_{\Omega}|D_{i,d}^{2}\nabla\zeta|^2d\mathbf{x} \nonumber\\
        &-\left(\frac{4}{5R^2}a_{\mathbf{u}}^{II}a_{\theta}^I\frac{\mu(\overline{\theta})\kappa(\overline{\theta})}{\mu(\overline{\theta})\overline{\theta}}+C\delta\right)\int_{\Omega}|D_{i,d}^{2}\nabla\zeta|\cdot\left|D_{i,d}^{2}\left(\nabla\mathbf{u}+\nabla\mathbf{u}^{\mathsf{T}}-\frac{2}{3}\div\mathbf{u}\mathbb{I}_{3}\right)\right|d\mathbf{x} \nonumber\\
        \geq&C_{0}''\left(\left\|D_{i,d}^{2}\left(\nabla\mathbf{u}+\nabla\mathbf{u}'-\frac{2}{3}\div\mathbf{u}\mathbb{I}_{3}\right)\right\|_{L^{2}(\Omega)}^2+\|D_{i,d}^{2}\nabla\zeta\|_{L^{2}(\Omega)}^2\right),
    \end{align*}
    where $C_{0}'>0$ is a constant only depending on $a_{\mathbf{u}}^{I}, a_{\mathbf{u}}^{II}, a_{\theta}^{I}, a_{\theta}^{II}, R, \mu, \kappa, \overline{\theta}$. Then applying Korn's type inequality \eqref{KornInequality}, we have
    \begin{equation}\label{regularityEstimateEllipticSystemJ1J2J3}
        J_{1}+J_{2}+J_{3}\geq C_{0}''(\|D_{i,d}^{2}\nabla\mathbf{u}\|_{L^{2}(\Omega)}^2+\|D_{i,d}^{2}\nabla\zeta\|_{L^{2}(\Omega)}^2).
    \end{equation}

    For $J_{4},...,I_{9}$, by H\"{o}lder's inequality, 
    \begin{equation}\label{regularityEstimateEllipticSystemJ4J5J6}
        |J_{4}|+\cdots+|J_{6}|\leq C_{1}''\delta\sum_{0\leq j\leq2}(\|D_{i,d}^{j}\mathbf{u}\|_{H^{1}(\Omega)}^{2}+\|D_{i,d}^{j}\zeta\|_{H^{1}(\Omega)}^{2}),
    \end{equation}
    where $C_{1}''>0$ is a constant only depending on $a_{\mathbf{u}}^{I}, a_{\mathbf{u}}^{II}, a_{\theta}^{I}, a_{\theta}^{II}, R, \mu, \kappa, \overline{\theta}$.

    For $J_{10}$ and $J_{11}$, since
    \begin{equation*}
        R\overline{\rho}\sqrt{\frac{2}{R}}\frac{2}{3R}a_{\mathbf{u}}^{I}a_{\mathbf{u}}^{II}\mu(\overline{\theta})^{2}\overline{\theta}^{-\frac{1}{2}}=R(\overline{\rho}\overline{\theta}+\mathcal{R}\overline{\rho})\frac{32}{75R^2}\sqrt{\frac{2}{R}}a_{\theta}^{I}a_{\theta}^{II}\mu(\overline{\theta})\kappa(\overline{\theta})\overline{\theta}^{-\frac{3}{2}},
    \end{equation*}
    therefore, by H\"{o}lder's inequality, 
    \begin{align}\label{regularityEstimateEllipticSystemJ7J8}
        |J_{10}+J_{11}|\leq&C_{2}''\delta\left(\int_{\Omega}|D_{i,d}\mathbf{u}\cdot D_{i,d}^{2}\nabla\zeta|d\mathbf{x}+\int_{\Omega}|D_{i,d}^{2}\zeta\cdot D_{i,d}^{2}\div\mathbf{u}|d\mathbf{x}\right) \nonumber\\
        \leq&C_{2}''\delta\left(\|D_{i,d}^{2}\mathbf{u}\|_{H^{1}(\Omega)}^2+\|D_{i,d}^{2}\zeta\|_{H^{1}(\Omega)}^2\right),
    \end{align}
    where $C_{2}''>0$ is a constant only depending on $a_{\mathbf{u}}^{I}, a_{\mathbf{u}}^{II}, a_{\theta}^{I}, a_{\theta}^{II}, R, \mu, \kappa, \overline{\rho}, \overline{\theta}$.

    For $J_{12}$, by divergence theorem, 
    \begin{align}\label{regularityEstimateEllipticSystemJ9}
        J_{12}=&\int_{\Omega}\frac{4}{5R}\overline{\rho}a_{\mathbf{u}}^{II}a_{\theta}^I\frac{\mu(\tilde{\theta})\kappa(\tilde{\theta})}{\mu(\overline{\theta})\tilde{\theta}}|D_{i,d}^{2}\zeta|^2 d\mathbf{x}+\int_{\Omega}\frac{2}{5R}\overline{\rho}a_{\mathbf{u}}^{II}a_{\theta}^I\partial_{1}\left(\frac{\mu(\tilde{\theta})\kappa(\tilde{\theta})}{\mu(\overline{\theta})\tilde{\theta}}\right)\cdot\mathfrak{n}|D_{i,d}^{2}\zeta|^2 d\mathbf{x} \nonumber\\
        &+\sum_{p=0,1}\int_{\mathbb{T}_{p}^{2}}-\frac{2}{5R}\overline{\rho}a_{\mathbf{u}}^{II}a_{\theta}^I\frac{\mu(\theta_{w})\kappa(\theta_{w})}{\mu(\overline{\theta})\theta_{w}}|D_{i,d}^{2}\zeta|^2 dx_{2}dx_{3} \nonumber\\
        \geq&(\frac{4}{5R}a_{\mathbf{u}}^{II}a_{\theta}^I\overline{\rho}\frac{\kappa(\overline{\theta})}{\overline{\theta}}-C_{3}''\delta)\|D_{i,d}^{2}\zeta\|_{L^{2}(\Omega)}^2-\left(\frac{2}{5R}a_{\mathbf{u}}^{II}a_{\theta}^I\frac{\kappa(\overline{\theta})}{\overline{\theta}}\overline{\rho}+C_{3}''\delta\right)\|D_{i,d}^{2}\zeta\|_{L^{2}(\partial\Omega)}^2,
    \end{align}
    where $C_{3}''>0$ is a constant only depending on $a_{\mathbf{u}}^{I}, a_{\mathbf{u}}^{II}, a_{\theta}^{I}, a_{\theta}^{II}, R, \mu, \kappa, \overline{\rho}, \overline{\theta}$.

    For $J_{13}$, $J_{14}$ and $J_{15}$, by H\"{o}lder's inequality, 
    \begin{equation}\label{regularityEstimateEllipticSystemJ10J11J12}
        |J_{13}|+|J_{14}|+|J_{15}|\leq C_{4}''\delta(\|D_{i,d}^{2}\mathbf{u}\|_{H^{1}(\Omega)}^{2}+\|D_{i,d}^{2}\zeta\|_{H^{1}(\Omega)}^{2}),
    \end{equation}
    where $C_{4}''>0$ is a constant only depending on $a_{\mathbf{u}}^{I}, a_{\mathbf{u}}^{II}, a_{\theta}^{I}, a_{\theta}^{II}, R, \mu, \kappa, \overline{\rho}, \overline{\theta}$.

    For $J_{16},...,J_{21}$,
    \begin{align}\label{regularityEstimateEllipticSystemJ13J14J15J16J17J18}
        |J_{16}|+\cdots+|J_{21}|\geq&\frac{2}{3R}a_{\mathbf{u}}^{II}\mu(\overline{\theta})(\overline{\rho}-C_{5}''\delta-C_{5}''\varepsilon)\|D_{i,d}^{2}\mathbf{u}\|_{L^{2}(\partial\Omega)}^2 \nonumber\\
        &+\frac{16}{15R}a_{\theta}^{I}\frac{\kappa(\overline{\theta})}{\overline{\theta}}(\overline{\rho}-C_{5}''\delta-C_{5}''\varepsilon)\|D_{i,d}^{2}\zeta\|_{L^{2}(\partial\Omega)}^2, 
    \end{align}
    where $C_{5}''>0$ is a constant only depending on $a_{\mathbf{u}}^{I}, a_{\mathbf{u}}^{II}, a_{\theta}^{I}, a_{\theta}^{II}, R, \mu, \kappa, \overline{\rho}, \overline{\theta}$.

    For $J_{22},...,J_{27}$, by H\"{o}lder's inequality, 
    \begin{align}\label{regularityEstimateEllipticSystemJ19J20J21J22J23J24}
        |J_{22}|+\cdots+|J_{27}|\leq&\frac{C_{0}''}{2}(\|D_{i,-d}(D_{i,d}^{2}\mathbf{u})\|_{L^{2}(\Omega)}^{2}+\|D_{i,-d}(D_{i,d}^{2}\zeta)\|_{L^{2}(\Omega)}^{2}) \nonumber\\
        &+C_{6}''(\|D_{i,d}\nabla\varphi'\|_{L^{2}(\Omega)}^{2}+\|D_{i,d}\mathbf{F'}\|_{L^{2}(\Omega)}^{2}+\|D_{i,d}G'\|_{L^{2}(\Omega)}^{2}+\|D_{i,d}\div(\varphi'\mathbf{u}')\|_{L^{2}(\Omega)}^{2}),
    \end{align}
    where $C_{6}''>0$ is a constant only depending on $a_{\mathbf{u}}^{I}, a_{\mathbf{u}}^{II}, a_{\theta}^{I}, a_{\theta}^{II}, R, \mu, \kappa, \overline{\theta}$.

    Combining \eqref{regularityEstimateEllipticSystemJ1J2J3}, \eqref{regularityEstimateEllipticSystemJ4J5J6}, \eqref{regularityEstimateEllipticSystemJ7J8}, \eqref{regularityEstimateEllipticSystemJ9}, \eqref{regularityEstimateEllipticSystemJ10J11J12}, \eqref{regularityEstimateEllipticSystemJ13J14J15J16J17J18}, \eqref{regularityEstimateEllipticSystemJ19J20J21J22J23J24} and Lemma \ref{EstimateFGLemma}, by \eqref{assumption1} in assumption (A), then choosing $\delta, \varepsilon, d>0$ sufficiently small, we have,
    \begin{equation}\label{tangentialH3RegularityEstimateEllipticSystem}
        \sum_{i=2,3}\|\partial_{i}\mathbf{u}\|_{H^{2}(\Omega)}^2 +\sum_{i=2,3}\|\partial_{i}\zeta\|_{H^{2}(\Omega)}^2\leq  CN_{2}(\varphi',\mathbf{u}',\zeta')^{2}(1+N_{2}(\varphi',\mathbf{u}',\zeta'))^{4}.
    \end{equation}
    Moreover, from $\eqref{LSNS}_{2}$, 
    \begin{equation*}
        -\frac{4}{3}\mu(\overline{\theta})\partial_{1}^{3}u_{1}=\mu(\overline{\theta})\sum_{j=2,3}\partial_{1}\partial_{j}^{2}u_{1}+\mu(\overline{\theta})\sum_{j=2,3}\partial_{1}\partial_{j}\partial_{1}u_{j}-R\overline{\rho}\partial_{1}^{2}\zeta+\partial_{1}F_{1}(\varphi',\mathbf{u}',\zeta')-R\overline{\theta}\partial_{1}^{2}\varphi',
    \end{equation*}
    and  
    \begin{equation*}
        -\mu(\overline{\theta})\partial_{1}^{3}u_{i}=\mu(\overline{\theta})\sum_{j=2,3}\partial_{1}\partial_{j}^{2}u_{i}+\mu(\overline{\theta})\sum_{1\leq j\leq3}\partial_{1}\partial_{j}\partial_{i}u_{j}-\frac{2}{3}\partial_{1}\partial_{i}\div\mathbf{u}-R\overline{\rho}\partial_{1}\partial_{i}\zeta+\partial_{1}F_{i}(\varphi',\mathbf{u}',\zeta')-R\overline{\theta}\partial_{1}\partial_{i}\varphi',
    \end{equation*}
    for $i=2,3$. Similarly, from $\eqref{LSNS}_{3}$,
    \begin{equation*}
        -\kappa(\overline{\theta})\partial_{1}^{3}\zeta=\kappa(\overline{\theta})\sum_{i=2,3}\partial_{1}\partial_{i}^{2}\zeta-(R\overline{\rho}\overline{\theta}+\mathcal{R}\overline{\rho})\partial_{1}\div\mathbf{u}+\partial_{1}G(\varphi',\mathbf{u}',\zeta')-\mathcal{R}\partial_{1}\div(\varphi'\mathbf{u}').
    \end{equation*}
    Together with $\eqref{H2RegularityEstimateEllipticSystem}, \eqref{tangentialH3RegularityEstimateEllipticSystem}$, we obtain $\eqref{regularityEstimateEllipticSystem}$,
    which also implies $(\mathbf{u},\zeta)\in H^{3}(\Omega)\times H^{3}(\Omega)$. The proof of Lemma \ref{regularityEllipticSystem} is complete.
\end{proof}

Applying Lemma \ref{regularityEllipticSystem} to Corollary \ref{weakExistenceEllipticSystem}, we obtain the existence of a solution in $V^{3}(\Omega)\times H^{3}(\Omega)$ to boundary value problem \eqref{LSNS}$_{2,3}$ and \eqref{LBC}.

\begin{corollary}\label{existenceEllipticSystem}
    Under the assumption of Theorem \ref{mainTheorem}, there exist constants $\delta_{0}, \varepsilon_{0}>0$ such that for $\delta\in[0,\delta_{0}), \varepsilon\in[0,\varepsilon_{0})$, if the wall temperature $\theta_{w}\in H^{\vartheta}(\partial\Omega)$ with $\|\theta_{w}-\overline{\theta}\|_{H^{\vartheta}(\partial\Omega)}\leq \delta$ where $\vartheta>\frac{7}{2}$, and $(\varphi',\mathbf{u}',\zeta')\in (H^{2}(\Omega)\cap L_{0}^{2}(\Omega))\times V^{3}(\Omega)\times H^{3}(\Omega)$ with $N_{2}(\varphi',\mathbf{u}',\zeta')\leq\varepsilon$ such that $\|\varphi'\|_{L^{\infty}(\Omega)}+\overline{\rho}>0$ and $\|\zeta'\|_{L^{\infty}(\Omega)}+\overline{\theta}>0$, then the boundary value problem $\eqref{LSNS}_{2}$ and $\eqref{LSNS}_{3}$ with \eqref{LBC} has a unique strong solution $(u,\zeta)\in V^{3}(\Omega)\times H^{3}(\Omega)$. 
\end{corollary}

\subsection{Existence and regularity results for $\varphi$}
Having obtained $(\mathbf{u},\zeta)\in(H^{3}(\Omega)\cap V)\times H^{3}(\Omega)$, let us consider $\eqref{LSNS}_{1}$ for $\varphi$. We rewrite $\eqref{LSNS}_{1}$ in the following way:
\begin{equation}\label{transportEquation}
  \varphi+h\div(\varphi\mathbf{u})=H(\varphi',\mathbf{u}),
\end{equation} 
where $H(\varphi',\mathbf{u})=\varphi'-h\overline{\rho}\div\mathbf{u}$. 

To construct a solution to \eqref{transportEquation}, let $k\in\mathbb{N}$, we consider:
\begin{equation}\label{ellipticApproximationTransport}
    -\frac{1}{k}\Delta\varphi^{k}+\varphi^{k}+h\div(\varphi^{k}\mathbf{u})=H(\varphi',\mathbf{u}),
\end{equation}
for $x\in\Omega$, with the Neumann boundary condition
\begin{equation}\label{neumannBC}
    \mathbf{n}\cdot\nabla\varphi^{k}=0,
\end{equation}
for $x\in\partial\Omega$. Then we have the following result.

\begin{lemma}\label{H1UniformEstimateEllipticApproximationTransportLemma}
    Under the assumption of Corollary \ref{existenceEllipticSystem}, for an arbitrary fixed constant $h>0$, the boundary value problem \eqref{ellipticApproximationTransport} and \eqref{neumannBC} has a unique strong solution $\varphi^{k}\in L_{0}^{2}(\Omega)\cap H^{4}(\Omega)$. Moreover, there exists a constant $C>0$ which is independent of $k$ such that 
    \begin{equation}\label{H1UniformEstimateEllipticApproximationTransport}
        \|\varphi^{k}\|_{H^{1}(\Omega)}^{2}\leq CN_{1}(\varphi',\mathbf{u}',\zeta')^{2}(1+N_{2}(\varphi',\mathbf{u}',\zeta'))^{4}.
    \end{equation}
\end{lemma}

\begin{proof}
     For the boundary value problem \eqref{ellipticApproximationTransport} and \eqref{neumannBC}, by the classic existence and regularity results of elliptic equations, the above problem admits a strong solution $\varphi^{k}\in H^4(\Omega)$. It suffices to prove \eqref{H1UniformEstimateEllipticApproximationTransport}.

    Multiplying $\eqref{ellipticApproximationTransport}$ by $\varphi^{k}$,
    \begin{equation*}
        \frac{1}{k}\|\nabla\varphi^{k}\|_{L^{2}(\Omega)}^{2}+\|\varphi^{k}\|_{L^{2}(\Omega)}^{2}+h\int_{\Omega}\varphi^{k}\div(\varphi^{k}\mathbf{u})d\mathbf{x}=\int_{\Omega}\varphi^{k}H(\varphi',\mathbf{u})d\mathbf{x}.
    \end{equation*}

    Note by H\"{o}lder's inequality and Sobolev's inequality, 
    \begin{equation*}
        \left|\int_{\Omega}\varphi^{k}\div(\varphi^{k}\mathbf{u})d\mathbf{x}\right|\leq \frac{1}{2}\int_{\Omega}|\varphi^{k}|^2|\div\mathbf{u}|d\mathbf{x}\leq C_{1}\|\varphi^{k}\|_{L^{2}(\Omega)}^2\|\mathbf{u}\|_{H^{3}(\Omega)},
    \end{equation*}
    and
    \begin{equation*}
        \left|\int_{\Omega}\varphi^{k}H(\varphi',\mathbf{u})d\mathbf{x}\right|\leq \frac{1}{2}\|\varphi^{k}\|_{L^{2}(\Omega)}^{2}+C_{2}(\|\varphi'\|_{L^{2}(\Omega)}^{2}+h^{2}\|\nabla\mathbf{u}\|_{L^{2}(\Omega)}^{2}).
    \end{equation*}
    Together with \eqref{basicEstimateEllipticSystem} in Lemma \ref{weakExistenceEllipticSystem} and \eqref{regularityEstimateEllipticSystem} in Lemma \ref{regularityEllipticSystem}, and then choosing $\varepsilon>0$ sufficiently small, we have
    \begin{equation}\label{L2EstimateEllipticApproximationTransport}
        \frac{1}{k}\|\nabla\varphi^{k}\|_{L^{2}(\Omega)}^{2}+\|\varphi^{k}\|_{L^{2}(\Omega)}^{2}\leq C_{3}N_{1}(\varphi',\mathbf{u}',\zeta'))^{2}(1+N_{2}(\varphi',\mathbf{u'},\zeta'))^{4}.
    \end{equation}

    Multiplying $\nabla\eqref{ellipticApproximationTransport}$ by $\nabla\varphi^{k}$,
    \begin{equation*}
        \frac{1}{k}\|\nabla^{2}\varphi^{k}\|_{L^{2}(\Omega)}^{2}+\|\nabla\varphi^{k}\|_{L^{2}(\Omega)}^{2}-h\int_{\Omega}\nabla\varphi^{k}\cdot\nabla\div(\varphi^{k}\mathbf{u})d\mathbf{x}=\int_{\Omega}\nabla\varphi^{k}\cdot\nabla H(\varphi',\mathbf{u})d\mathbf{x}.
    \end{equation*}

    We further note by H\"{o}lder's inequality and Sobolev's inequality, 
    \begin{align*}
        &\left|\int_{\Omega}\nabla\varphi^{k}\cdot\nabla\div(\varphi^{k}\mathbf{u})d\mathbf{x}\right| \\
        \leq&2\left(\int_{\Omega}|\nabla\varphi^{k}|^2|\div\mathbf{u}|d\mathbf{x}+\int_{\Omega}|\nabla\varphi^{k}\cdot\varphi^{k}\nabla\div\mathbf{u}|d\mathbf{x}\right) \\
        \leq&2(\|\nabla\varphi^{k}\|_{L^{2}(\Omega)}^{2}\|\div\mathbf{u}\|_{L^{\infty}(\Omega)}^2+\|\nabla\varphi^{k}\|_{L^{2}(\Omega)}\|\varphi^{k}\|_{L^{6}(\Omega)}\|\nabla\div\mathbf{u}\|_{L^{3}(\Omega)}) \\
        \leq&C_{4}\|\varphi^{k}\|_{H^{1}(\Omega)}^{2}\|\mathbf{u}\|_{H^{3}(\Omega)},
    \end{align*}
    and
    \begin{equation*}
        \left|\int_{\Omega}\nabla\varphi^{k}\cdot\nabla H(\varphi',\mathbf{u})d\mathbf{x}\right|\leq\frac{1}{2}\|\nabla\varphi^{k}\|_{L^{2}(\Omega)}^{2}+C_{5}(\|\nabla\varphi'\|_{L^{2}(\Omega)}^{2}+h^{2}\|\nabla\mathbf{u}\|_{L^{2}(\Omega)}^{2}).
    \end{equation*}
    Together with \eqref{regularityEstimateEllipticSystem} in Lemma \ref{regularityEllipticSystem}, and then choosing $\varepsilon>0$ sufficiently small, we have
    \begin{equation}\label{H1EstimateEllipticApproximationTransport}
        \frac{1}{k}\|\nabla^{2}\varphi^{k}\|_{L^{2}(\Omega)}^{2}+\|\nabla\varphi^{k}\|_{L^{2}(\Omega)}^{2}\leq C_{6}N_{1}(\varphi',\mathbf{u}',\zeta')^{2}(1+N_{2}(\varphi',\mathbf{u}',\zeta'))^{4}.
    \end{equation}

    Combining \eqref{L2EstimateEllipticApproximationTransport} and \eqref{H1EstimateEllipticApproximationTransport}, we obtain \eqref{H1UniformEstimateEllipticApproximationTransport} and thus complete the proof of Lemma \ref{H1UniformEstimateEllipticApproximationTransportLemma}.
\end{proof}

Since $\varphi^{k}$ is uniformly bounded in $H^{1}(\Omega)$ with respect to $k\in\mathbf{N}$. Therefore by subtracting a sequence in $H^{1}(\Omega)$ and taking the limit in $k\rightarrow+\infty$, we obtain the following existence results of \eqref{transportEquation}.

\begin{corollary}\label{weakExistenceTransportEquation}
    Under the assumption of Corollary \ref{existenceEllipticSystem}, for an arbitrary fixed constant $h>0$, \eqref{transportEquation} has a weak solution $\varphi\in L_{0}^{2}(\Omega)\cap H^{1}(\Omega)$.
\end{corollary}

Moreover, we have the following regularity result.

\begin{lemma}\label{regularityTransportEquation}
  Under the assumption of Corollary \ref{existenceEllipticSystem}, for an arbitrary fixed constant $h>0$, if $\varphi\in L^{2}_{0}(\Omega)\cap H^{1}(\Omega)$ is a weak solution to \eqref{transportEquation}, then $\varphi\in L_{0}^{2}(\Omega)\cap H^{2}(\Omega)$. Moreover, there exists a constant $C>0$ such that 
    \begin{equation}\label{H2EstimateDensity}
        \|\varphi\|_{H^{2}(\Omega)}^{2}\leq CN_{2}(\varphi',\mathbf{u}',\zeta')^{2}(1+N_{2}(\varphi',\mathbf{u}',\zeta'))^{4}.
    \end{equation}
\end{lemma}

\begin{proof}
Following the same procedure as \cite{Farwig-CPDE-1989}, we can prove that $\varphi$ is in $H^{2}(\Omega)$. It then suffices to prove the desired estimate \eqref{H2EstimateDensity}. Since $\varphi^{k}$ is uniformly bounded, from the construction of $\varphi$, by \eqref{H1UniformEstimateEllipticApproximationTransport}, we have
    \begin{equation}\label{H1Estimatedensity}
        \|\varphi\|_{H^{1}(\Omega)}^{2}\leq C_{1}N_{1}(\varphi',\mathbf{u}',\zeta')^{2}(1+N_{2}(\varphi',\mathbf{u}',\zeta'))^{4}.
    \end{equation}
To obtain $H^{2}(\Omega)$ estimates, multiplying $\partial_{i}\nabla\eqref{transportEquation}$ by $\partial_{i}\nabla\varphi$ for $i=1,2,3$, it holds that
    \begin{equation*}
        \|\partial_{i}\nabla\varphi\|_{L^{2}(\Omega)}^{2}-h\int_{\Omega}\partial_{i}\nabla\varphi\cdot\partial_{i}\nabla\div(\varphi\mathbf{u})dx=\int_{\Omega}\partial_{i}\nabla\varphi\cdot\partial_{i}\nabla H(\varphi',\mathbf{u})dx.
    \end{equation*}
For the second term, by H\"{o}lder's inequality and Sobolev's inequality, 
    \begin{align*}
        &\left|\int_{\Omega}\partial_{i}\nabla\varphi\cdot\partial_{i}\nabla\div(\varphi\mathbf{u})dx\right| \\
        =&\left|\int_{\Omega}\partial_{i}\nabla\varphi\cdot(\partial_{i}\nabla^{2}\varphi\cdot\mathbf{u}+\nabla^{2}\varphi\cdot\partial_{i}\mathbf{u}+2\partial_{i}\nabla\varphi\cdot\div\mathbf{u}+2\nabla\varphi\cdot\partial_{i}\div\mathbf{u}+\partial_{i}\varphi\nabla\div\mathbf{u}+\varphi\partial_{i}\nabla\div\mathbf{u})dx\right| \\
        \leq&C_{1}(\|\partial_{i}\nabla\varphi\|_{L^{2}(\Omega)}^{2}\|\div\mathbf{u}\|_{L^{\infty}}+\|\nabla^{2}\varphi\|_{L^{2}(\Omega)}^{2}\|\partial_{i}\mathbf{u}\|_{L^{\infty}(\Omega)}+\|\partial_{i}\nabla\varphi\|_{L^{2}(\Omega)}^{2}\|\div\mathbf{u}\|_{L^{\infty}(\Omega)} \\
        &+\|\partial_{i}\nabla\varphi\|_{L^{2}(\Omega)}\|\nabla\varphi\|_{L^{4}(\Omega)}\|\partial_{i}\div\mathbf{u}\|_{L^{4}(\Omega)}+\|\partial_{i}\nabla\varphi\|_{L^{2}(\Omega)}\|\partial_{i}\varphi\|_{L^{4}(\Omega)}\|\nabla\div\mathbf{u}\|_{L^{4}(\Omega)} \\
        &+\|\partial_{i}\nabla\varphi\|_{L^{2}(\Omega)}\|\varphi\|_{L^{\infty}(\Omega)}\|\partial_{i}\nabla\div\mathbf{u}\|_{L^{2}(\Omega)}) \\
        \leq&C_{1}N_{2}(\varphi',\mathbf{u}',\zeta')^{2}(1+N_{2}(\varphi',\mathbf{u}',\zeta'))^{4}.
    \end{align*}
For the third term, by H\"{o}lder's inequality, 
    \begin{equation*}
        \left|\int_{\Omega}\partial_{i}\nabla\varphi\cdot\partial_{i}\nabla H(\varphi',\mathbf{u})dx\right|\leq\frac{1}{2}(\|\nabla\varphi\|_{L^{2}(\Omega)}^{2}+\|\nabla H(\varphi',\mathbf{u})\|_{L^{2}(\Omega)}^{2}).
    \end{equation*}
Therefore
    \begin{equation*}
        \|\partial_{i}\nabla\varphi\|_{L^{2}(\Omega)}^{2}\leq 
        C_{2}N_{2}(\varphi',\mathbf{u}',\zeta')^{2}(1+N_{2}(\varphi',\mathbf{u}',\zeta'))^{4},
    \end{equation*}
    which together with \eqref{H1Estimatedensity}, gives \eqref{H2EstimateDensity}. Thus the proof of Lemma \ref{regularityTransportEquation} is complete.
\end{proof}

Combining Lemma \ref{regularityTransportEquation} with Corollary \ref{weakExistenceTransportEquation}, we obtain the existence of solutions to \eqref{transportEquation} in $L^{2}_{0}(\Omega)\cap H^{2}(\Omega)$.
\begin{corollary}\label{ExistenceTransportEquation}
  Under the assumption of Corollary \ref{existenceEllipticSystem}, for an arbitrary fixed constant $h>0$, \eqref{transportEquation} has a solution $\varphi\in L_{0}^{2}(\Omega)\cap H^{2}(\Omega)$.
\end{corollary}

\section{Continuity of solution operator $\mathds{T}$}
In this section, we prove that the operator $\mathds{T}$ is continuous. 

For $i=1,2$, suppose $(\varphi_{i},\mathbf{u}_{i},\zeta_{i})\in (L_{0}^{2}(\Omega)\cap H^{2}(\Omega))\times V^{3}(\Omega)\times H^{3}(\Omega)$ is constructed by using Lemma \ref{existenceEllipticSystem} and Lemma \ref{ExistenceTransportEquation} corresponding to the given data $(\varphi'_{i},\mathbf{u}'_{i},\zeta'_{i})\in (L_{0}^{2}(\Omega)\cap H^{2}(\Omega))\times V^{3}(\Omega)\times H^{3}(\Omega)$. Thus it solves
the boundary value problem:
\begin{equation}\label{LSNSs}
    \left\{
    \begin{aligned}
        &\frac{\varphi_{i}-\varphi'_{i}}{h}+\div(\varphi_{i}\mathbf{u}_{i})+\overline{\rho}\div\mathbf{u}_{i}=0, \\
        &-\div\mathbb{S}(\mathbf{u}_{i},\overline{\theta})+R\overline{\rho}\nabla\zeta_{i} =\mathbf{F}(\varphi'_{i},\mathbf{u}'_{i},\zeta'_{i})-R\overline{\theta}\nabla\varphi'_{i}, \\
        &-\div(\kappa(\overline{\theta})\nabla\zeta_{i})+(R\overline{\rho}\overline{\theta}+\mathcal{R}\overline{\rho})\div\mathbf{u}_{i} =G(\varphi'_{i},\mathbf{u}'_{i},\zeta'_{i})-\mathcal{R}\div(\varphi'_{i}\mathbf{u}'_{i}),
    \end{aligned}
    \right.
\end{equation}
for $x\in\Omega$, with the boundary conditions
\begin{equation}\label{LBCs}
    \left\{
    \begin{aligned}
        &\mathbf{u}_{i}\cdot\mathbf{n}=0, \\
        &(\varphi'_{i}+\overline{\rho}) \mathbf{u}_{i}\cdot \mathbf{t}+\sqrt{\frac{2}{R}}a_{\mathbf{u}}^{I}\frac{\mu(\theta_{w})}{\sqrt{\theta_{w}}}(\nabla\mathbf{u}_{i}+\nabla\mathbf{u}_{i}^{\mathsf{T}}):\mathbf{n}\otimes\mathbf{t}-\frac{4}{5R}a_{\theta}^{I}\frac{\kappa(\theta_{w})}{\theta_{w}}\nabla(\zeta_{i}+\tilde{\theta})\cdot\mathbf{t}=0, \\
        &(\varphi'_{i}+\overline{\rho})\zeta_{i}-\frac{1}{R}a_{\mathbf{u}}^{II}\mu(\theta_{w})\nabla\mathbf{u}_{i}:\mathbf{n}\otimes\mathbf{n}+\frac{2}{5R}\sqrt{\frac{2}{R}}a_{\theta}^{II}\frac{\kappa(\theta_{w})}{\sqrt{\theta_{w}}}\nabla(\zeta_{i}+\tilde{\theta})\cdot\mathbf{n}=0,
    \end{aligned}
    \right.
\end{equation}
for $x\in\partial\Omega$. To prove the continuity, let us denote
\begin{equation}\label{difference}
    (\mathcal{D},\boldsymbol{\mathcal{V}},\mathcal{T}):=(\varphi_1-\varphi_2,\mathbf{u}_1-\mathbf{u}_2,\zeta_1-\zeta_2).
\end{equation}
Then plugging \eqref{difference} into \eqref{LSNSs}, \eqref{LBCs}, we obtain that $(\mathcal{D},\mathcal{V},\mathcal{T})$ satisfies
\begin{equation}\label{DSNS}
    \left\{
    \begin{aligned}
    &\mathcal{D}+h\div(\mathcal{D}\mathbf{u}_1)+h\div(\varphi_1\boldsymbol{\mathcal{V}})=\mathcal{H}, \\
    &-\div\left[\mu(\overline{\theta})\left(\nabla\boldsymbol{\mathcal{V}}+\nabla\boldsymbol{\mathcal{V}}^{\mathsf{T}}-\frac{2}{3}\div\boldsymbol{\mathcal{V}}\mathbb{I}_{3}\right)\right]+R\overline{\rho}\nabla\mathcal{T} =\boldsymbol{\mathcal{F}}, \\
    &-\div(\kappa(\overline{\theta})\nabla\mathcal{T})+(R\overline{\rho}\overline{\theta}+\mathcal{R}\overline{\rho})\div\boldsymbol{\mathcal{V}} =\mathcal{G},
    \end{aligned}
    \right.
\end{equation}
for $x\in\Omega$, with 
\begin{equation}\label{DBC}
    \left\{
    \begin{aligned}
        &\boldsymbol{\mathcal{V}}\cdot\mathbf{n}=0, \\
        &(\varphi_{1}+\overline{\rho})\boldsymbol{\mathcal{V}}\cdot\mathbf{t}+(\varphi'_{1}-\varphi'_{2})\mathbf{u}_{2}\cdot\mathbf{t}+\sqrt{\frac{2}{R}}a_{\mathbf{u}}^{I}\frac{\mu(\theta_{w})}{\sqrt{\theta_{w}}}(\nabla\boldsymbol{\mathcal{V}}+\nabla\boldsymbol{\mathcal{V}}^{\mathsf{T}}):\mathbf{n}\otimes\mathbf{t}-\frac{4}{5R}a_{\theta}^{I}\frac{\kappa(\theta_{w})}{\theta_{w}}\nabla\mathcal{T}\cdot\mathbf{t}=0, \\
        &(\varphi'_{1}+\overline{\rho})\mathcal{T}+(\varphi'_{1}-\varphi'_{2})\zeta_{2}-\frac{1}{R}a_{\mathbf{u}}^{II}\mu(\theta_{w})\nabla\boldsymbol{\mathcal{V}}:\mathbf{n}\otimes\mathbf{n}-\frac{2}{5R}\sqrt{\frac{2}{R}}a_{\theta}^{II}\frac{\kappa(\theta_{w})}{\sqrt{\theta_{w}}}\nabla\mathcal{T}\cdot\mathbf{n}=0,
    \end{aligned}
    \right.
\end{equation}
for $x\in\partial\Omega$, where
\begin{align*}
    \mathcal{H}:=&H(\varphi_1',\mathbf{u}_1)-H(\varphi_2',\mathbf{u}_2),\\
    \boldsymbol{\mathcal{F}}:=&\mathbf{F}(\varphi_1',\mathbf{u}_1',\zeta_1')-\mathbf{F}(\varphi_2',\mathbf{u}_2',\zeta_2')-R\overline{\theta}\nabla(\varphi_1'-\varphi_2'),\\
    \mathcal{G}:=&G(\varphi_1',\mathbf{u}_1',\zeta_1')-G(\varphi_2',\mathbf{u}_2',\zeta_2')-\mathcal{R}\div(\varphi_1'\mathbf{u}_1'-\varphi_2'\mathbf{u}_2').
\end{align*}
We prove the following result.

\begin{lemma}\label{continuityTheorem}
    Under the assumptions of Corollary \ref{existenceEllipticSystem} and Corollary \ref{ExistenceTransportEquation}, there exists a constant $C>0$ only depending on $a_{\mathbf{u}}^{I}, a_{\mathbf{u}}^{II}, a_{\theta}^{I}, a_{\mathbf{u}}^{II}, c_{v}, R, \mu, \kappa, \overline{\rho}, \overline{\theta}$ such that the solution $(\mathcal{D},\boldsymbol{\mathcal{V}},\mathcal{T})$ of \eqref{DSNS}, \eqref{DBC} satisfies
    \begin{equation}\label{EstimateContinuity}
        \|\mathcal{D}\|_{L^{2}(\Omega)}^2+\|\boldsymbol{\mathcal{V}}\|_{H^{1}(\Omega)}^2+\|\mathcal{T}\|_{H^{1}(\Omega)}^2\leq C(\|\mathcal{D}'\|_{L^{2}(\Omega)}^2+\|\boldsymbol{\mathcal{V}}'\|_{H^{1}(\Omega)}^2+\|\mathcal{T}'\|_{H^{1}(\Omega)}^2).
    \end{equation}
\end{lemma}

\begin{proof}
    For the estimate of $\boldsymbol{\mathcal{V}}$ and $\mathcal{T}$, using the variational form, we have
    \begin{align*}
        &\underbrace{\int_{\Omega}\frac{2}{3R}\sqrt{\frac{2}{R}}a_{\mathbf{u}}^{I}a_{\mathbf{u}}^{II}\frac{\mu(\tilde{\theta})^{2}}{\sqrt{\tilde{\theta}}}\nabla\boldsymbol{\mathcal{V}}:\left(\nabla\boldsymbol{\mathcal{V}}+\nabla\boldsymbol{\mathcal{V}}^{\mathsf{T}}-\frac{2}{3}\div\boldsymbol{\mathcal{V}}\mathbb{I}_{3}\right)d\mathbf{x}}_{I_{1}} +\underbrace{\int_{\Omega}\frac{32}{75R^2}\sqrt{\frac{2}{R}}a_{\theta}^{I}a_{\theta}^{II}\frac{\kappa(\tilde{\theta})^{2}}{\tilde{\theta}\sqrt{\tilde{\theta}}}\nabla\mathcal{T}\cdot\nabla\mathcal{T} d\mathbf{x}}_{I_{2}} \\
        &+\underbrace{\int_{\Omega}-\frac{4}{5R^2}a_{\mathbf{u}}^{II}a_{\theta}^I\frac{\mu(\tilde{\theta})\kappa(\tilde{\theta})}{\mu(\overline{\theta})\tilde{\theta}}\nabla\mathcal{T}\cdot\mathfrak{n}\left(\nabla\mathcal{V}_{1}+\partial_{1}\boldsymbol{\mathcal{V}}-\frac{2}{3}(\div\boldsymbol{\mathcal{V}})\mathbf{e}_{1}\right)d\mathbf{x} }_{I_{3}} \\
        &+\underbrace{\int_{\Omega}\frac{2}{3R}\sqrt{\frac{2}{R}}a_{\mathbf{u}}^{I}a_{\mathbf{u}}^{II}\nabla\left(\frac{\mu(\tilde{\theta})^{2}}{\sqrt{\tilde{\theta}}}\right)\otimes\boldsymbol{\mathcal{V}}:\left(\nabla\boldsymbol{\mathcal{V}}+\nabla\boldsymbol{\mathcal{V}}^{\mathsf{T}}-\frac{2}{3}\div\boldsymbol{\mathcal{V}}\mathbb{I}_{3}\right)d\mathbf{x}}_{I_{4}} \nonumber\\
        &+\underbrace{\int_{\Omega}\frac{32}{75R^2}\sqrt{\frac{2}{R}}a_{\theta}^{I}a_{\theta}^{II}\nabla\left(\frac{\kappa(\tilde{\theta})\kappa(\tilde{\theta})}{\tilde{\theta}\sqrt{\tilde{\theta}}}\right)\mathcal{T}\cdot\nabla\mathcal{T} d\mathbf{x}}_{I_{5}} \\
        &+\underbrace{\int_{\Omega}-\frac{4}{5R^2}a_{\mathbf{u}}^{II}a_{\theta}^I\nabla\left(\frac{\mu(\tilde{\theta})\kappa(\tilde{\theta})}{\mu(\overline{\theta})\tilde{\theta}}\right)\cdot\mathcal{T}\mathfrak{n}\left(\nabla\mathcal{V}_{1}+\partial_{1}\boldsymbol{\mathcal{V}}-\frac{2}{3}(\div\boldsymbol{\mathcal{V}})\mathbf{e}_{1}\right)d\mathbf{x}}_{I_{6}} \\
        &+\underbrace{\int_{\Omega}R\overline{\rho}\sqrt{\frac{2}{R}}\frac{2}{3R}a_{\mathbf{u}}^{I}a_{\mathbf{u}}^{II}\frac{\mu(\tilde{\theta})^{2}}{\mu(\overline{\theta})\sqrt{\tilde{\theta}}}\boldsymbol{\mathcal{V}}\cdot\nabla\mathcal{T} d\mathbf{x}}_{I_{7}} +\underbrace{\int_{\Omega}\frac{32}{75R^2}\sqrt{\frac{2}{R}}a_{\theta}^{I}a_{\theta}^{II}\frac{\kappa(\tilde{\theta})^{2}}{\kappa(\overline{\theta})\tilde{\theta}\sqrt{\tilde{\theta}}}(R\overline{\rho}\overline{\theta}+\mathcal{R}\overline{\rho})\mathcal{T}\div\boldsymbol{\mathcal{V}}d\mathbf{x}}_{I_{8}} \\
        &+\underbrace{\int_{\Omega}-\frac{4}{5R}\overline{\rho}a_{\mathbf{u}}^{II}a_{\theta}^I\frac{\mu(\tilde{\theta})\kappa(\tilde{\theta})}{\mu(\overline{\theta})\tilde{\theta}}\mathcal{T}\mathbf{n}\cdot\nabla\mathcal{T} d\mathbf{x}}_{I_{9}} \\
        &+\underbrace{\sum_{p=0,1}\int_{\mathbb{T}_{p}^{2}}\frac{2}{3R}a_{\mathbf{u}}^{II}\mu(\theta_{w})(\varphi'_{1}+\overline{\rho})|\boldsymbol{\mathcal{V}}|^{2}dx_{2}dx_{3}}_{I_{10}} +\underbrace{\sum_{p=0,1}\int_{\mathbb{T}_{p}^{2}}\frac{16}{15R}a_{\theta}^{I}\frac{\kappa(\theta_{w})}{\theta_{w}}(\varphi'+\overline{\rho})|\mathcal{T}|^{2}dx_{2}dx_{3}}_{I_{11}} \\
        &+\underbrace{\sum_{p=0,1}\int_{\mathbb{T}_{p}^{2}}\frac{2}{3R}a_{\mathbf{u}}^{II}\mu(\theta_{w})(\varphi'_{1}+\overline{\rho})\boldsymbol{\mathcal{V}}\cdot(\varphi'_{1}-\varphi'_{2})\mathbf{u}_{2}dx_{2}dx_{3}}_{I_{12}} +\underbrace{\sum_{p=0,1}\int_{\mathbb{T}_{p}^{2}}\frac{16}{15R}a_{\theta}^{I}\frac{\kappa(\theta_{w})}{\theta_{w}}\mathcal{T}(\varphi'_{1}-\varphi'_{2})\zeta_{2} dx_{2}dx_{3}}_{I_{13}} \\
        = 
        &\underbrace{\int_{\Omega}\frac{2}{3R}\sqrt{\frac{2}{R}}a_{\mathbf{u}}^{I}a_{\mathbf{u}}^{II}\frac{\mu(\tilde{\theta})^{2}}{\mu(\overline{\theta})\sqrt{\tilde{\theta}}}\boldsymbol{\mathcal{V}}\cdot\boldsymbol{\mathcal{F}}d\mathbf{x}}_{I_{14}} +\underbrace{\int_{\Omega}\frac{32}{75R^2}\sqrt{\frac{2}{R}}a_{\theta}^{I}a_{\theta}^{II}\frac{\kappa(\tilde{\theta})^{2}}{\kappa(\overline{\theta})\tilde{\theta}\sqrt{\tilde{\theta}}}\mathcal{T} \mathcal{G}d\mathbf{x}}_{I_{15}} \\
        &+\underbrace{\int_{\Omega}-\frac{4}{5R^2}a_{\mathbf{u}}^{II}a_{\theta}^I\frac{\mu(\tilde{\theta})\kappa(\tilde{\theta})}{\mu(\overline{\theta})\tilde{\theta}}\mathcal{T}\mathbf{n}\cdot\boldsymbol{\mathcal{F}}(\varphi',\mathbf{u}',\zeta')d\mathbf{x}}_{I_{16}}.
    \end{align*}

    Note that for $I_{1}$, $I_{2}$ and $I_{3}$, by H\"{o}lder's inequality and \eqref{assumption2} in assumption (A), 
    \begin{align*}
        I_{1}+I_{2}+I_{3}\geq&\left(\frac{1}{3R}\sqrt{\frac{2}{R}}a_{\mathbf{u}}^{I}a_{\mathbf{u}}^{II}\frac{\mu(\overline{\theta})^{2}}{\sqrt{\overline{\theta}}}-C\delta\right)\int_{\Omega}\left|\nabla\boldsymbol{\mathcal{V}}+\nabla\boldsymbol{\mathcal{V}}^{\mathsf{T}}-\frac{2}{3}\div\boldsymbol{\mathcal{V}}\mathbb{I}_{3}\right|^2d\mathbf{x} \nonumber\\
        &+\left(\frac{32}{75R^2}\sqrt{\frac{2}{R}}a_{\theta}^{I}a_{\theta}^{II}\frac{\kappa(\overline{\theta})^{2}}{\overline{\theta}\sqrt{\overline{\theta}}}-C\delta\right)\int_{\Omega}|\nabla\mathcal{T}|^2d\mathbf{x} \nonumber\\
        &-\left(\frac{4}{5R^2}a_{\mathbf{u}}^{II}a_{\theta}^I\frac{\mu(\overline{\theta})\kappa(\overline{\theta})}{\mu(\overline{\theta})\overline{\theta}}+C\delta\right)\int_{\Omega}|\nabla\mathcal{T}|\cdot\left|\nabla\boldsymbol{\mathcal{V}}+\nabla\boldsymbol{\mathcal{V}}^{\mathsf{T}}-\frac{2}{3}\div\boldsymbol{\mathcal{V}}\mathbb{I}_{3}\right|d\mathbf{x} \nonumber\\
        \geq&C_{0}\left(\left\|\nabla\boldsymbol{\mathcal{V}}+\nabla\boldsymbol{\mathcal{V}}^{\mathsf{T}}-\frac{2}{3}\div\boldsymbol{\mathcal{V}}\mathbb{I}_{3}\right\|_{L^{2}(\Omega)}^2+\|\nabla\mathcal{T}\|_{L^{2}(\Omega)}^2\right),
    \end{align*}
    where $C_{0}>0$ is a constant only depending on $a_{\mathbf{u}}^{I}, a_{\mathbf{u}}^{II}, a_{\theta}^{I}, a_{\theta}^{II},\overline{\theta}$. Applying Korn's type inequality \eqref{KornInequality},
    \begin{equation}\label{EstimateContinuityI1I2I3}
        I_{1}+I_{2}+I_{3}\geq C_{0}(\|\nabla\boldsymbol{\mathcal{V}}\|_{L^{2}(\Omega)}^2+\|\nabla\mathcal{T}\|_{L^{2}(\Omega)}^2).
    \end{equation}

    For $I_{4}$, $I_{5}$ and $I_{6}$, by H\"{o}lder's inequality, 
    \begin{equation}\label{EstimateContinuityI4I5I6}
        |I_{4}|+|I_{5}|+|I_{6}|\leq C_{1}\delta(\|\boldsymbol{\mathcal{V}}\|_{H^{1}(\Omega)}^2+\|\mathcal{T}\|_{H^{1}(\Omega)}^2),
    \end{equation}
    where $C_{1}>0$ is a constant only depending on $a_{\mathbf{u}}^{I},a_{\mathbf{u}}^{II},a_{\theta}^{I},a_{\theta}^{II},\overline{\theta}$.

    For $I_{7}$ and $I_{8}$, note that
    \begin{equation*}
        R\overline{\rho}\sqrt{\frac{2}{R}}\frac{2}{3R}a_{\mathbf{u}}^{I}a_{\mathbf{u}}^{II}\mu(\overline{\theta})^{2}\overline{\theta}^{-\frac{1}{2}}=R(\overline{\rho}\overline{\theta}+\mathcal{R}\overline{\rho})\frac{32}{75R^2}\sqrt{\frac{2}{R}}a_{\theta}^{I}a_{\theta}^{II}\mu(\overline{\theta})\kappa(\overline{\theta})\overline{\theta}^{-\frac{3}{2}},
    \end{equation*}
    and it then follows by H\"{o}lder's inequality, 
    \begin{align}\label{EstimateContinuityI7I8}
        |I_{7}+I_{8}|\leq&C\delta\left(\int_{\Omega}|\boldsymbol{\mathcal{V}}\cdot\nabla\mathcal{T}|d\mathbf{x}+\int_{\Omega}|\mathcal{T}\div\boldsymbol{\mathcal{V}}|d\mathbf{x}\right) \nonumber\\
        \leq&C_{2}\delta\left(\|\boldsymbol{\mathcal{V}}\|_{H^{1}(\Omega)}^{2}+\|\mathcal{T}\|_{H^{1}(\Omega)}^{2}\right),
    \end{align}
    where $C_{2}>0$ is a constant only depending on $a_{\mathbf{u}}^{I},a_{\mathbf{u}}^{II},a_{\theta}^{I},a_{\theta}^{II},\overline{\rho},\overline{\theta}$.

    For $I_{9}$, by divergence theorem,
    \begin{align}\label{EstimateContinuityI9}
        I_{9}=&\int_{\Omega}\frac{2}{5R}\overline{\rho}a_{\mathbf{u}}^{II}a_{\theta}^I\frac{\mu(\tilde{\theta})\kappa(\tilde{\theta})}{\mu(\overline{\theta})\tilde{\theta}}\mathfrak{n}|\mathcal{T}|^2 d\mathbf{x}+\int_{\Omega}\frac{2}{5R}\overline{\rho}a_{\mathbf{u}}^{II}a_{\theta}^I\partial_{1}\left(\frac{\mu(\tilde{\theta})\kappa(\tilde{\theta})}{\mu(\overline{\theta})\tilde{\theta}}\right)\cdot\mathfrak{n}|\mathcal{T}|^2 d\mathbf{x} \nonumber\\
        &+\sum_{p=0,1}\int_{\mathbb{T}_{p}^{2}}-\frac{2}{5R}\overline{\rho}a_{\mathbf{u}}^{II}a_{\theta}^I\frac{\mu(\theta_{w})\kappa(\theta_{w})}{\mu(\overline{\theta})\theta_{w}}\mathcal{T}^2 dx_{2}dx_{3} \nonumber\\
        \geq&\left(\frac{4}{5R}a_{\mathbf{u}}^{II}a_{\theta}^I\overline{\rho}\frac{\kappa(\overline{\theta})}{\overline{\theta}}-C_{3}\delta\right)\|\mathcal{T}\|_{L^{2}(\Omega)}^2-(\frac{2}{5R}\overline{\rho}a_{\mathbf{u}}^{II}a_{\theta}^I\frac{\kappa(\overline{\theta})}{\overline{\theta}}+C_{3}\delta)\|\mathcal{T}\|_{L^{2}(\partial\Omega)}^2,
    \end{align}
    where $C_{3}>0$ is a constant only depending on $a_{\mathbf{u}}^{I},a_{\mathbf{u}}^{II},a_{\theta}^{I},a_{\theta}^{II},\overline{\rho},\overline{\theta}$.

    For $I_{10}$, $I_{11}$, $I_{12}$ and $I_{13}$, 
    \begin{align}\label{EstimateContinuityI10I11I12I13}
        I_{10}+I_{11}+I_{12}+I_{13}\geq&\frac{2}{3R}a_{\mathbf{u}}^{II}\mu(\overline{\theta})(\overline{\rho}-C_{4}\delta-C_{4}\varepsilon)\|\boldsymbol{\mathcal{V}}\|_{L^{2}(\partial\Omega)}^2 \nonumber\\
        &+\frac{16}{15R}a_{\theta}^{I}\frac{\kappa(\overline{\theta})}{\overline{\theta}}(\overline{\rho}-C_{4}\delta-C_{4}\varepsilon)\|\mathcal{T}\|_{L^{2}(\partial\Omega)}^2 \nonumber\\
        &-C_{4}(\|\mathbf{u}_{2}\|_{L^{\infty}(\Omega)}+\|\zeta_{2}\|_{L^{\infty}(\Omega)})\|\mathcal{D}'\|_{L^{2}(\partial\Omega)}^2.
    \end{align}
    where $C_{4}>0$ is a constant only depending on $a_{\mathbf{u}}^{I},a_{\mathbf{u}}^{II},a_{\theta}^{I},a_{\theta}^{II},\overline{\rho},\overline{\theta}$.

    For $I_{14}$, $I_{15}$, $I_{16}$, by H\"{o}lder's inequality, 
    \begin{equation}\label{EstimateContinuityI14I15I16}
        |I_{14}|+|I_{15}|+|I_{16}|\leq\frac{C_{0}}{2}(\|\boldsymbol{\mathcal{V}}\|_{L^{2}(\Omega)}^2+\|\mathcal{T}\|_{L^{2}(\Omega)}^2)+C_{5}(\|\mathbf{F'}\|_{L^{2}(\Omega)}^2+\|G'\|_{L^{2}(\Omega)}^2),
    \end{equation}
    where $C_{5}>0$ is a constant only depending on $a_{\mathbf{u}}^{I},a_{\mathbf{u}}^{II},a_{\theta}^{I},a_{\theta}^{II},\overline{\theta}$.

    In the following, we estimate $\|\boldsymbol{\mathcal{F}}\|_{L^{2}(\Omega)}^2$ and $\|\boldsymbol{\mathcal{G}}\|_{L^{2}(\Omega)}^2$. For $\|\boldsymbol{\mathcal{F}}\|_{L^{2}(\Omega)}^2$, 
    \begin{align*}
        &\|\boldsymbol{\mathcal{F}}\|_{L^{2}(\Omega)}^2 \\
        \leq&C_{6}( \underbrace{\|\varphi'_{1}\mathbf{u}'_{1}\cdot\nabla\mathbf{u}'_{1}-\varphi'_{2}\mathbf{u}'_{2}\cdot\nabla\mathbf{u}'_{2}\|_{L^{2}(\Omega)}^2}_{J_1} +\underbrace{\|\mathbf{u}'_{1}\cdot\nabla\mathbf{u}'_{1}-\mathbf{u}'_{2}\cdot\nabla\mathbf{u}'_{2}\|_{L^{2}(\Omega)}^2}_{J_2} \\
        &+\underbrace{\|\varphi'_{1}\nabla\zeta'_{1}-\varphi'_{2}\nabla\zeta'_{2}\|_{L^{2}(\Omega)}^2}_{J_3} +\|\varphi'_{1}-\varphi'_{2}\|_{L^{2}(\Omega)}^2 +\underbrace{\|\zeta'_{1}\nabla\varphi'_{1}-\zeta'_{2}\nabla\varphi'_{2}\|_{L^{2}(\Omega)}^2}_{J_{4}} +\|\nabla\varphi'_{1}-\nabla\varphi'_{2}\|_{L^{2}(\Omega)}^2) \\
        &+\underbrace{\left\|\div\left((\mu(\theta_{1}')-\mu(\overline{\theta}))\mathbb{S}(\mathbf{u}_{1}')\right)-\div\left((\mu(\theta_{2}')-\mu(\overline{\theta}))\mathbb{S}(\mathbf{u}_{2}')\right)\right\|_{L^{2}(\Omega)}^2}_{J_{5}},
    \end{align*}
    where $C_{6}>0$ is a constant only depending on $C_{v}, R, \overline{\rho}, \overline{\theta}$.

    For $J_1$, 
    \begin{align}\label{EstimateContinuityJ1}
        J_1\leq&\|(\varphi'_{1}-\varphi'_{2})\mathbf{u}'_{1}\cdot\nabla\mathbf{u}'_{1}\|_{L^{2}(\Omega)}^2 +\|\varphi'_{2}(\mathbf{u}'_{1}-\mathbf{u}'_{2})\cdot\nabla\mathbf{u}'_{2}\|_{L^{2}(\Omega)}^2 +\|\varphi'_{2}\mathbf{u}'_{2}\cdot\nabla(\mathbf{u}'_{1}-\mathbf{u}'_{2})\|_{L^{2}(\Omega)}^2 \nonumber\\
        \leq&\|\mathbf{u}'_{1}\|_{L^{\infty}(\Omega)}^2\|\nabla\mathbf{u}'_{1}\|_{L^{\infty}}^2\|\varphi'_{1}-\varphi'_{2}\|_{L^{2}(\Omega)}^{2} +\|\varphi'_{2}\|_{L^{\infty}(\Omega)}^2\|\nabla\mathbf{u}'_{2}\|_{L^{\infty}(\Omega)}^2\|\mathbf{u}'_{1}-\mathbf{u}'_{2}\|_{L^{2}(\Omega)}^{2} \nonumber\\
        &+\|\varphi'_{2}\|_{L^{\infty}(\Omega)}^2\|\mathbf{u}'_{2}\|_{L^{\infty}(\Omega)}^2\|\nabla(\mathbf{u}'_{1}-\mathbf{u}'_{2})\|_{L^{2}(\Omega)}^{2} \nonumber\\
        \leq&C_{7}(\|\mathbf{u}_{1}\|_{H^{3}(\Omega)}^{4}\|\varphi_{1}'-\varphi_{2}'\|_{L^{2}(\Omega)}^{2}+\|\varphi_{2}'\|_{H^{2}(\Omega)}^{2}\|\mathbf{u}_{2}'\|_{H^{3}(\Omega)}^{2}\|\mathbf{u}_{1}'-\mathbf{u}_{2}'\|_{H^{1}(\Omega)}).
    \end{align}

    For $J_2$,
    \begin{align}\label{EstimateContinuityJ2}
        J_2\leq&\|(\mathbf{u}'_{1}-\mathbf{u}'_{2})\cdot\nabla\mathbf{u}'_{1}\|_{L^{2}(\Omega)}^2 +\|\mathbf{u}'_{2}\nabla(\mathbf{u}'_{1}-\mathbf{u}'_{2})\|_{L^{2}(\Omega)}^2 \nonumber\\
        \leq&\|\nabla\mathbf{u}'_{1}\|_{L^{\infty}(\Omega)}^2\|\mathbf{u}'_{1}-\mathbf{u}'_{2}\|_{L^{2}(\Omega)}^{2} +\|\mathbf{u}'_{2}\|_{L^{\infty}(\Omega)}^2\|\nabla(\mathbf{u}'_{1}-\mathbf{u}'_{2})\|_{L^{2}(\Omega)}^{2} \nonumber\\ 
        \leq&C_{8}(\|\mathbf{u}_{1}'\|_{H^{2}(\Omega)}^{2}+\|\mathbf{u}_{2}'\|_{H^{2}(\Omega)}^{2})\|\mathbf{u}_{1}'-\mathbf{u}_{2}'\|_{H^{1}(\Omega)}^{2}.
    \end{align}

    For $J_3$,
    \begin{align}\label{EstimateContinuityJ3}
        J_3\leq&\|(\varphi'_{1}-\varphi'_{2})\nabla\zeta'_{1}\|_{L^{2}(\Omega)}^2 +\|\varphi'_{2}\nabla(\zeta'_{1}-\zeta'_{2})\|_{L^{2}(\Omega)}^2 \nonumber\\
        \leq&\|\nabla\zeta'_{1}\|_{L^{\infty}(\Omega)}^2\|\varphi'_{1}-\varphi'_{2}\|_{L^{2}(\Omega)}^{2} +\|\varphi'_{2}\|_{L^{\infty}(\Omega)}^2\|\nabla(\zeta'_{1}-\zeta'_{2})\|_{L^{2}(\Omega)}^{2} \nonumber\\
        \leq&C_{9}(\|\zeta_{1}\|_{H^{3}(\Omega)}^{2}\|\varphi_{1}'-\varphi_{2}'\|_{L^{2}(\Omega)}+\|\varphi_{2}'\|_{H^{3}(\Omega)}\|\zeta_{1}'-\zeta_{2}'\|_{H^{1}(\Omega)}^{2}).
    \end{align}

    For $J_{4}$,
    \begin{align}\label{EstimateContinuityJ4}
        J_{4}\leq&\|(\zeta'_{1}-\zeta'_{2})\nabla\varphi'_{1}\|_{L^{2}(\Omega)}^2 +\|\zeta'_{2}\nabla(\varphi'_{1}-\varphi'_{2})\|_{L^{2}(\Omega)}^2 \nonumber\\
        \leq&\|(\zeta'_{1}-\zeta'_{2})\|_{L^4(\Omega)}^2\|\nabla\varphi'_{1}\|_{L^4(\Omega)}^2 +\|\nabla\zeta'_{2}\|_{L^{\infty}(\Omega)}^{2}\|(\varphi'_{1}-\varphi'_{2})\|_{L^{2}(\Omega)}^{2} \nonumber\\
        \leq&C_{10}(\|\varphi'_{1}\|_{H^{2}(\Omega)}^2\|(\zeta'_{1}-\zeta'_{2})\|_{H^{1}(\Omega)}^2 +\|\zeta'_{2}\|_{H^{3}(\Omega)}^2\|\varphi'_{1}-\varphi'_{2}\|_{L^{2}(\Omega)}^{2}),
    \end{align}

    For $J_{5}$,
    \begin{align*}
        J_{5}\leq&\|\div((\mu(\theta_{1}')-\mu(\theta_{2}'))\mathbb{S}(\mathbf{u}_{1}'))\|_{L^{2}(\Omega)}^{2} +\|\div(\mu(\theta_{2}')(\mathbb{S}(\mathbf{u}_{1}')-\mathbb{S}(\mathbf{u}_{2}')))\|_{L^{2}(\Omega)}^{2} \nonumber\\
        \leq&C_{11}(\|\zeta_{1}'-\zeta_{2}'\|_{L^{4}(\Omega)}^{2}\|\div\mathbb{S}(\mathbf{u}_{1}')\|_{L^{4}(\Omega)}^{2}+\|\mathbb{S}(\mathbf{u}_{1}')\|_{L^{\infty}(\Omega)}^{2}\|\nabla\zeta_{1}'-\nabla\zeta_{2}'\|_{L^{2}(\Omega)}^{2} \nonumber\\
        &+\|\zeta_{2}'\|_{L^{\infty}(\Omega)}^{2}\|\div\mathbb{S}(\mathbf{u}_{1}')-\div\mathbb{S}(\mathbf{u}_{1}')\|_{L^{2}(\Omega)}^{2}+\|\nabla\zeta_{1}'\|_{L^{\infty}(\Omega)}^{2}\|\mathbb{S}(\mathbf{u}_{1}')-\mathbb{S}(\mathbf{u}_{2}')\|_{L^{2}(\Omega)}^{2}) \nonumber\\
        \leq&C_{11}(\|\mathbf{u}_{1}'\|_{H^{3}(\Omega)}^{2}\|\zeta_{1}'-\zeta_{2}'\|_{H^{1}(\Omega)}^{2}+\|\zeta_{2}'\|_{H^{3}(\Omega)}^{2}\|\mathbf{u}_{1}'-\mathbf{u}_{2}'\|_{H^{2}(\Omega)}^{2}),
    \end{align*}
    where $C_{11}>0$ is a constant only depending on $\mu, \overline{\theta}$. Since
    \begin{equation*}
        \|\mathbf{u}_{1}'-\mathbf{u}_{2}'\|_{H^{2}(\Omega)}^{2}\leq C_{11}\|\mathbf{u}_{1}'-\mathbf{u}_{2}'\|_{H^{1}(\Omega)}\|\mathbf{u}_{1}'-\mathbf{u}_{2}'\|_{H^{3}(\Omega)}\leq(\|\mathbf{u}_{1}'\|_{H^{3}(\Omega)}+\|\mathbf{u}_{2}'\|_{H^{3}(\Omega)})\|\mathbf{u}_{1}'-\mathbf{u}_{2}'\|_{H^{1}(\Omega)}^{2},
    \end{equation*}
    it holds that
    \begin{equation}\label{EstimateContinuityJ5}
        J_{5}\leq C_{11}(\|\mathbf{u}_{1}'\|_{H^{3}(\Omega)}^{2}\|\zeta_{1}'-\zeta_{2}'\|_{H^{1}(\Omega)}^{2}+\|\zeta_{2}'\|_{H^{3}(\Omega)}^{2}(\|\mathbf{u}_{1}'\|_{H^{3}(\Omega)}+\|\mathbf{u}_{2}'\|_{H^{3}(\Omega)})\|\mathbf{u}_{1}'-\mathbf{u}_{2}'\|_{H^{1}(\Omega)}^{2}).
    \end{equation}

    Combining \eqref{EstimateContinuityJ1}, \eqref{EstimateContinuityJ2}, \eqref{EstimateContinuityJ3}, \eqref{EstimateContinuityJ4}, \eqref{EstimateContinuityJ5}, we have
    \begin{equation}\label{EstimateContinuityF}
        \|\boldsymbol{\mathcal{F}}\|_{L^{2}(\Omega)}^2\leq C_{12}(\|\mathbf{u}'_{1}-\mathbf{u}'_{2}\|_{H^{1}(\Omega)}^2+\|\zeta'_{1}-\zeta'_{2}\|_{H^{1}(\Omega)}^2+\|\varphi'_{1}-\varphi'_{2}\|_{L^{2}(\Omega)}^{2}),
    \end{equation}
    where $C_{12}>0$ is a constant only depending on $c_{v}, R, \mu, \overline{\rho}, \overline{\theta}$.

    For $\|\mathcal{G}\|_{L^{2}(\Omega)}^2$,
    \begin{align*}
        &\|\mathcal{G}\|_{L^{2}(\Omega)}^2\\
        \leq&C_{13}( \underbrace{\|\varphi'_{1}\mathbf{u}'_{1}\cdot\nabla\zeta'_{1}-\varphi'_{2}\mathbf{u}'_{2}\cdot\nabla\zeta'_{2}\|_{L^{2}(\Omega)}^2}_{K_1} +\underbrace{\|\mathbf{u}'_{1}\cdot\nabla\zeta'_{1}-\mathbf{u}'_{2}\cdot\nabla\zeta'_{2}\|_{L^{2}(\Omega)}^2}_{K_2} +\underbrace{\|\varphi'_{1}\mathbf{u}'_{1}-\varphi'_{2}\mathbf{u}'_{2}\|_{L^{2}(\Omega)}^2}_{K_3} \\
        &+\|\mathbf{u}'_{1}-\mathbf{u}'_{2}\|_{L^{2}(\Omega)}^2
        +\underbrace{\|\varphi'_{1}\zeta'_{1}\div\mathbf{u}'_{1}-\varphi'_{2}\zeta'_{2}\div\mathbf{u}'_{2}\|_{L^{2}(\Omega)}^2}_{K_{4}} +\underbrace{\|\zeta'_{1}\div\mathbf{u}'_{1}-\zeta'_{2}\div\mathbf{u}'_{2}\|_{L^{2}(\Omega)}^2}_{K_{5}} \\
        & +\underbrace{\|\varphi'_{1}\div\mathbf{u}'_{1}-\varphi'_{2}\div\mathbf{u}'_{2}\|_{L^{2}(\Omega)}^2}_{K_{6}} +\underbrace{\|\div\mathbf{u}'_{1}-\div\mathbf{\mathbf{u}}'_{2}\|_{L^{2}(\Omega)}^2}_{K_{7}}\big)
        \\
        & +\underbrace{\left\|\div\left([\kappa(\theta_{1}')-\kappa(\overline{\theta})]\nabla\zeta'_{1}\right)-\div\left([\kappa(\theta_{2}')-\kappa(\overline{\theta})]\nabla\zeta'_{2}\right)\right\|_{L^{2}(\Omega)}^2}_{K_{8}} +\underbrace{\|\div(\kappa(\theta_{1}')\nabla\tilde{\theta})-\div(\kappa(\theta_{2}')\nabla\tilde{\theta})\|_{L^{2}(\Omega)}^2}_{K_{9}}\\
        &+\underbrace{\left\|\nabla\mathbf{u}'_{1}:\mu(\theta_{1}')\mathbb{S}(\mathbf{u}_{1}')-\nabla\mathbf{u}'_{2}:\mu(\theta_{2}')\mathbb{S}(\mathbf{u}_{2}')\right\|_{L^{2}(\Omega)}^2}_{K_{10}} +\underbrace{\|\div(\varphi_{1}'\mathbf{u}_{1}'-\varphi_{2}'\mathbf{u}_{2}')\|_{L^{2}(\Omega)}^2}_{K_{11}},
    \end{align*}
    where $C_{13}>0$ is a constant only depending on $C_{v}, R, \overline{\rho}, \overline{\theta}$.

    For $K_1$,
    \begin{align}\label{EstimateContinutiyK1}
        K_1\leq&\|(\varphi'_{1}-\varphi'_{2})\mathbf{u}'_{1}\cdot\nabla\zeta'_{1}\|_{L^{2}(\Omega)}^2 +\|\varphi'_{2}(\mathbf{u}'_{1}-\mathbf{u}'_{2})\cdot\nabla\zeta'_{1}\|_{L^{2}(\Omega)}^2 +\|\varphi'_{2}\mathbf{u}'_{1}\cdot(\nabla\zeta'_{1}-\nabla\zeta'_{2})\|_{L^{2}(\Omega)}^2 \nonumber\\
        \leq&\|\mathbf{u}'_{1}\|_{L^{\infty}(\Omega)}^2\|\nabla\zeta'_{1}\|_{L^{\infty}(\Omega)}^2\|\varphi'_{1}-\varphi'_{2}\|_{L^{2}(\Omega)}^{2} +\|\varphi'_{2}\|_{L^{\infty}(\Omega)}^2\|\nabla\zeta'_{1}\|_{L^{\infty}(\Omega)}^2\|\mathbf{u}'_{1}-\mathbf{u}'_{2}\|_{L^{2}(\Omega)}^{2} \nonumber\\
        &+\|\varphi'_{2}\|_{L^{\infty}(\Omega)}^2\|\mathbf{u}'_{1}\|_{L^{\infty}}^2\|\nabla\zeta'_{1}-\nabla\zeta'_{2}\|_{L^{2}(\Omega)}^{2} \nonumber\\
        \leq&C_{14}(\|\mathbf{u}_{1}'\|_{H^{2}(\Omega)}^{2}\|\zeta_{1}'\|_{H^{3}(\Omega)}^{2}\|\varphi_{1}'-\varphi_{2}'\|_{L^{2}(\Omega)}^{2}+\|\varphi_{2}'\|_{H^{2}(\Omega)}^{2}\|\zeta_{1}'\|_{H^{2}(\Omega)}\|\mathbf{u}_{1}'-\mathbf{u}_{2}'\|_{L^{2}(\Omega)}^{2} \nonumber\\
        &+\|\varphi_{2}'\|_{H^{2}(\Omega)}^{2}\|\mathbf{u}_{1}'\|_{H^{2}(\Omega)}^{2}\|\zeta_{1}'-\zeta_{2}'\|_{H^{1}(\Omega)}^{2}).
    \end{align}

    For $K_2$,
    \begin{align}\label{EstimateContinuityK2}
        K_2\leq&\|(\mathbf{u}'_{1}-\mathbf{u}'_{2})\nabla\zeta'_{1}\|_{L^{2}(\Omega)}^2 +\|\mathbf{u}'_{1}(\nabla\zeta'_{1}-\nabla\zeta'_{2})\|_{L^{2}(\Omega)}^2 \nonumber\\
        \leq&\|\nabla\zeta'_{1}\|_{L^{\infty}(\Omega)}^2\|\mathbf{u}'_{1}-\mathbf{u}'_{2}\|_{L^{2}(\Omega)}^{2} +\|\mathbf{u}'_{1}\|_{L^{\infty}(\Omega)}^2\|\nabla\zeta'_{1}-\nabla\zeta'_{2}\|_{L^{2}(\Omega)}^{2} \nonumber\\
        \leq&C_{15}(\|\zeta_{1}'\|_{H^{2}(\Omega)}^{2}\|\mathbf{u}_{1}'-\mathbf{u}_{2}'\|_{H^{1}(\Omega)}+\|\mathbf{u}_{1}'\|_{H^{2}(\Omega)}^{2}\|\zeta_{1}'-\zeta_{2}'\|_{H^{1}(\Omega)}^{2}).
    \end{align}

    For $K_3$,
    \begin{align}\label{EstimateContinuityK3}
        K_3\leq&\|(\varphi'_{1}-\varphi'_{2})\mathbf{u}'_{1}\|_{L^{2}(\Omega)}^2 +\|\varphi'_{2}(\mathbf{u}'_{1}-\mathbf{u}'_{2})\|_{L^{2}(\Omega)}^2 \nonumber\\
        \leq&\|\mathbf{u}'_{1}\|_{L^{\infty}(\Omega)}^2\|\varphi'_{1}-\varphi'_{2}\|_{L^{2}(\Omega)}^{2} +\|\varphi'_{2}\|_{L^{\infty}(\Omega)}^2\|\mathbf{u}'_{1}-\mathbf{u}'_{2}\|_{L^{2}(\Omega)}^{2} \nonumber\\
        \leq&C_{16}(\|\mathbf{u}'_{1}\|_{H^{2}(\Omega)}^2\|\varphi'_{1}-\varphi'_{2}\|_{L^{2}(\Omega)}^{2} +\|\varphi'_{2}\|_{H^{2}(\Omega)}^2\|\mathbf{u}'_{1}-\mathbf{u}'_{2}\|_{L^{2}(\Omega)}^{2}).
    \end{align}

    For $K_{4}$,
    \begin{align}\label{EstimateContinuityK4}
        K_{4}\leq&\|(\varphi'_{1}-\varphi'_{2})\zeta'_{1}\div\mathbf{u}'_{1}\|_{L^{2}(\Omega)}^2 +\|\varphi'_{2}(\zeta'_{1}-\zeta'_{2})\div\mathbf{u}'_{1}\|_{L^{2}(\Omega)}^2 +\|\varphi'_{2}\zeta'_{2}(\div\mathbf{u}'_{1}-\div\mathbf{u}'_{2})\|_{L^{2}(\Omega)}^2 \nonumber\\
        \leq&\|\zeta'_{1}\|_{L^{\infty}(\Omega)}^2\|\div\mathbf{u}'_{1}\|_{L^{\infty}(\Omega)}^2\|\varphi'_{1}-\varphi'_{2}\|_{L^{2}(\Omega)}^{2} +\|\varphi'_{2}\|_{L^{\infty}(\Omega)}^2\|\div\mathbf{u}'_{1}\|_{L^{\infty}(\Omega)}^2\|\zeta'_{1}-\zeta'_{2}\|_{L^{2}(\Omega)}^{2} \nonumber\\
        &+\|\varphi'_{2}\|_{L^{\infty}(\Omega)}^2\|\zeta'_{2}\|_{L^{\infty}(\Omega)}^2\|\div\mathbf{u}'_{1}-\div\mathbf{u}'_{2}\|_{L^{2}(\Omega)}^{2} \nonumber\\
        \leq&C_{17}(\|\zeta'_{1}\|_{H^{2}(\Omega)}^2\|\mathbf{u}'_{1}\|_{H^{3}(\Omega)}^2\|\varphi'_{1}-\varphi'_{2}\|_{L^{2}(\Omega)}^{2} +\|\varphi'_{2}\|_{H^{2}(\Omega)}^2\|\mathbf{u}'_{1}\|_{H^{3}(\Omega)}^2\|\zeta'_{1}-\zeta'_{2}\|_{L^{2}(\Omega)}^{2} \nonumber\\
        &+\|\varphi'_{2}\|_{H^{2}(\Omega)}^2\|\zeta'_{2}\|_{H^{2}(\Omega)}^2\|\mathbf{u}'_{1}-\mathbf{u}'_{2}\|_{H^{1}(\Omega)}^{2}).
    \end{align}

    For $K_{5}$,
    \begin{align}\label{EstimateContinuityK5}
        K_{5}\leq&\|(\zeta'_{1}-\zeta'_{2})\div\mathbf{u}'_{1}\|_{L^{2}(\Omega)}^2 +\zeta'_{2}(\div\mathbf{u}'_{1}-\div\mathbf{u}'_{2})\|_{L^{2}(\Omega)}^2 \nonumber\\
        \leq&\|\div\mathbf{u}'_{1}\|_{L^{\infty}(\Omega)}^2\|\zeta'_{1}-\zeta'_{2}\|_{L^{2}(\Omega)}^{2} +\|\zeta'_{2}\|_{L^{\infty}(\Omega)}^2\|\div\mathbf{u}'_{1}-\div\mathbf{u}'_{2}\|_{L^{2}(\Omega)}^{2} \nonumber\\
        \leq&C_{18}(\|\mathbf{u}_{1}'\|_{H^{2}(\Omega)}^{2}\|\zeta_{1}'-\zeta_{2}'\|_{L^{2}(\Omega)}^{2} +\|\zeta_{2}'\|_{H^{2}(\Omega)}^{2}\|\mathbf{u}_{1}'-\mathbf{u}_{2}'\|_{H^{1}(\Omega)}^{2}).
    \end{align}

    For $K_{6}$,
    \begin{align}\label{EstimateContinuityK6}
        K_{6}\leq&\|(\varphi'_{1}-\varphi'_{2})\div\mathbf{u}'_{1}\|_{L^{2}(\Omega)}^2 +\|\varphi'_{2}\div(\mathbf{u}'_{1}-\mathbf{u}'_{2})\|_{L^{2}(\Omega)}^2 \nonumber\\
        \leq&C_{19}(\|\div\mathbf{u}'_{1}\|_{L^{\infty}}^2\|\varphi'_{1}-\varphi'_{2}\|_{L^{2}(\Omega)}^{2} +\|\varphi'_{2}\|_{L^{\infty}(\Omega)}^2\|\div\mathbf{u}'_{1}-\div\mathbf{u}'_{2}\|_{L^{2}(\Omega)}^{2} \nonumber\\
        \leq&C_{19}(\|\mathbf{u}_{1}'\|_{H^{2}(\Omega)}^{3}\|\varphi_{1}'-\varphi_{2}'\|_{L^{2}(\Omega)}^{2}+\|\varphi_{2}'\|_{H^{2}(\Omega)}^{2}\|\mathbf{u}_{1}'-\mathbf{u}_{2}'\|_{H^{1}(\Omega)}^{2}).
    \end{align}

    For $K_{7}$,
    \begin{equation}\label{EstimateContinuityK7}
        K_{7}\leq C_{20}\|\nabla\mathbf{u}'_{1}-\nabla\mathbf{\mathbf{u}}'_{2}\|_{L^{2}(\Omega)}^{2}.
    \end{equation}

    For $K_{8}$,
    \begin{align}\label{EstimateContinuityK8}
        K_{8}\leq&C_{21}(\|(\zeta_{1}'-\zeta_{2}')\nabla\zeta_{1}'\|_{L^{2}(\Omega)}^{2} +\|\zeta_{2}'(\nabla\zeta_{1}'-\nabla\zeta_{2}')\|_{L^{2}(\Omega)}^{2}) \nonumber\\
        \leq&C_{21}(\|\nabla\zeta_{1}'\|_{L^{\infty}}^{2}\|\zeta_{1}'-\zeta_{2}'\|_{L^{2}(\Omega)}^{2} +\|\zeta_{2}'\|_{L^{\infty}}^{2}\|\nabla\zeta_{1}'-\nabla\zeta_{2}'\|_{L^{2}(\Omega)}^{2}) \nonumber\\
        \leq&C_{21}(\|\zeta_{1'}\|_{H^{3}(\Omega)}^{2}+\|\zeta_{2}'\|_{H^{2}(\Omega)}^{2})\|\zeta_{1}'-\zeta_{2}'\|_{H^{1}(\Omega)}^{2},
    \end{align}
    where $C_{21}>0$ is a constant only depending on $\kappa, \overline{\theta}$.

    For $K_{9}$,
    \begin{equation}\label{EstimateContinuityK9}
        K_{9}\leq C_{22}(\|\zeta_{1}'-\zeta_{2}'\|_{L^{2}(\Omega)}^{2}+\|\nabla\zeta_{1}'-\nabla\zeta_{2}'\|_{L^{2}(\Omega)}^{2}),
    \end{equation}
    where $C_{22}>0$ is a constant only depending on $\kappa, \overline{\theta}$.

    For $K_{10}$,
    \begin{align}\label{EstimateContinuityK10}
        K_{10}\leq&\left\|(\nabla\mathbf{u}'_{1}-\nabla\mathbf{u}'_{2}):\mu(\theta_{1}')\mathbb{S}(\mathbf{u}_{1}')\right\|_{L^{2}(\Omega)}^{2} +\left\|\nabla\mathbf{u}'_{2}:(\mu(\theta_{1}')-\mu(\theta_{2}')\mathbb{S}(\mathbf{u}_{1}')\right\| \nonumber\\
        &+\left\|\nabla\mathbf{u}'_{2}:\mu(\theta_{2}')(\mathbb{S}(\mathbf{u}_{1}')-\mathbb{S}(\mathbf{u}_{2}'))\right\|_{L^{2}(\Omega)}^{2} \nonumber\\
        \leq&C_{23}((\|\zeta'_{1}\|_{L^{\infty}(\Omega)}^2+\|\tilde{\theta}\|_{L^{\infty}(\Omega)}^2)\|\nabla\mathbf{u}'_{1}\|_{L^{\infty}(\Omega)}^2\|\nabla\mathbf{u}'_{1}-\nabla\mathbf{u}'_{2}\|_{L^{2}(\Omega)}^{2}) \nonumber\\
        &+\|\nabla\mathbf{u}'_{2}\|_{L^{\infty}(\Omega)}^2\|\nabla\mathbf{u}'_{1}\|_{L^{\infty}(\Omega)}^2\|\zeta'_{1}-\zeta'_{2}\|_{L^{2}(\Omega)}^{2} \nonumber\\
        &+\|\nabla\mathbf{u}'_{2}\|_{L^{\infty}(\Omega)}^2(\|\zeta'_{2}\|_{L^{\infty}(\Omega)}^2+\|\tilde{\theta}\|_{L^{\infty}(\Omega)}^2)\|\nabla(\mathbf{u}'_{1}-\mathbf{u}'_{2})\|_{L^{2}(\Omega)}^{2} \nonumber\\
        \leq&C_{23}((\|\zeta_{1}'\|_{H^{2}(\Omega)}^{2}+\|\zeta_{2}'\|_{H^{2}(\Omega)}^{2})(\|\mathbf{u}_{1}'\|_{H^{3}(\Omega)}^{2}+\|\mathbf{u}_{2}'\|_{H^{3}(\Omega)}^{2})\|\mathbf{u}_{1}'-\mathbf{u}_{2}'\|_{L^{2}(\Omega)}^{2} \nonumber\\
        &+\|\mathbf{u}_{2}'\|_{H^{3}(\Omega)}^{2}\|\mathbf{u}_{1}'\|_{H^{3}(\Omega)}^{2}\|\zeta_{1}'-\zeta_{2}'\|_{H^{1}(\Omega)}^{2}),
    \end{align}
    where $C_{23}>0$ is a constant only depending on $\mu, \overline{\theta}$.

    For $K_{11}$,
    \begin{align}\label{EstimateContinuityK11}
        K_{11}\leq&\|(\varphi_{1}'-\varphi_{2}')\div\mathbf{u}_{1}'\|_{L^{2}(\Omega)}^{2} +\|\varphi_{2}'(\div\mathbf{u}_{1}'-\div\mathbf{u}_{2}')\|_{L^{2}(\Omega)}^{2} +\|(\nabla\varphi_{1}'-\nabla\varphi_{2}')\cdot\mathbf{u}_{1}'\|_{L^{2}(\Omega)}^{2} \nonumber\\
        &+\|\nabla\varphi_{2}'\cdot(\mathbf{u}_{1}'-\mathbf{u}_{2}')\|_{L^{2}(\Omega)}^{2} \nonumber\\
        \leq&\|\div\mathbf{u}_{1}'\|_{L^{\infty}(\Omega)}^{2}\|\varphi_{1}'-\varphi_{2}'\|_{L^{2}(\Omega)}^{2} +\|\varphi_{2}'\|_{L^{\infty}(\Omega)}^{2}\|\div\mathbf{u}_{1}'-\div\mathbf{u}_{2}'\|_{L^{2}(\Omega)}^{2} \nonumber\\
        &+\|\mathbf{u}_{1}'\|_{L^{\infty}(\Omega)}^{2}\|\nabla\varphi_{1}'-\nabla\varphi_{2}'\|_{L^{2}(\Omega)}^{2} +\|\nabla\varphi_{2}'\|_{L^{4}(\Omega)})^{2}\|\mathbf{u}_{1}'-\mathbf{u}_{2}'\|_{L^{4}(\Omega)}^{2} \nonumber\\
        \leq&C_{24}(\|\mathbf{u}_{1}'\|_{H^{2}(\Omega)}^{2}\|\varphi_{1}'-\varphi_{2}'\|_{L^{2}(\Omega)}^{2} +\|\varphi_{2}'\|_{H^{2}(\Omega)}\|\mathbf{u}_{1}'-\mathbf{u}_{2}'\|_{H^{1}(\Omega)}^{2} \nonumber\\
        &+\|\mathbf{u}_{1}'\|_{H^{2}(\Omega)}^{2}\|\varphi_{1}'-\varphi_{2}'\|_{H^{1}(\Omega)}^{2}+\|\varphi_{2}'\|_{H^{2}(\Omega)}^{2}\|\mathbf{u}_{1}'-\mathbf{u}_{2}'\|_{H^{1}(\Omega)}^{2}),
    \end{align}
    where $C_{24}>0$ is a constant only depending on $\mu, \overline{\theta}$.

    Combining \eqref{EstimateContinutiyK1}, \eqref{EstimateContinuityK2}, \eqref{EstimateContinuityK3}, \eqref{EstimateContinuityK4}, \eqref{EstimateContinuityK5}, \eqref{EstimateContinuityK6}, \eqref{EstimateContinuityK7}, \eqref{EstimateContinuityK8}, \eqref{EstimateContinuityK9}, \eqref{EstimateContinuityK10}, \eqref{EstimateContinuityK11}, we have
    \begin{equation}\label{EstimateContinuityG}
        \|\mathcal{G}\|_{L^{2}(\Omega)}^2\leq C_{25}(\|\mathbf{u}'_{1}-\mathbf{u}'_{2}\|_{H^{1}(\Omega)}^2+\|\zeta'_{1}-\zeta'_{2}\|_{H^{1}(\Omega)}^2+\|\varphi'_{1}-\varphi'_{2}\|_{L^{2}(\Omega)}^{2}),
    \end{equation}
    where $C_{25}>0$ is a constant only depending on $c_{v}, R, \mu, \kappa, \overline{\rho}, \overline{\theta}$.

    Therefore, Combining \eqref{EstimateContinuityI1I2I3}, \eqref{EstimateContinuityI4I5I6}, \eqref{EstimateContinuityI7I8}, \eqref{EstimateContinuityI9}, \eqref{EstimateContinuityI10I11I12I13}, \eqref{EstimateContinuityI14I15I16} together with \eqref{EstimateContinuityG}, \eqref{EstimateContinuityF}, by \eqref{assumption1} in assumption (A), then choosing $\delta, \varepsilon>0$ sufficiently small, we obtain
    \begin{equation}\label{EstimateContinuityVelocityTemperautre}
        \|\boldsymbol{\mathcal{V}}\|_{H^{1}(\Omega)}^2+\|\mathcal{T}\|_{H^{1}(\Omega)}^2\leq C_{26}(\|\mathcal{D}'\|_{L^{2}(\Omega)}^2+\|\boldsymbol{\mathcal{V}}'\|_{H^{1}(\Omega)}^2+\|\mathcal{T}'\|_{H^{1}(\Omega)}^2),
    \end{equation}
    where $C_{26}>0$ is a constant only depending on $c_{v}, R, \mu, \kappa, \overline{\rho}, \overline{\theta}$.

    For the estimate of $\mathcal{D}$, multiplying $\eqref{DSNS}_{1}$ with $\mathcal{D}$ and integrating the resultant over $\Omega$, it then follows
    \begin{equation*}
        \int_{\Omega}\mathcal{D}^2d\mathbf{x}+\int_{\Omega}h\mathcal{D}\div(\mathcal{D}\mathbf{u}_{1})d\mathbf{x}+\int_{\Omega}h\mathcal{D}\div(\varphi_{1}\mathbf{V})d\mathbf{x}=\int_{\Omega}\mathcal{D}\mathcal{H}d\mathbf{x}.
    \end{equation*}
    
    We compute
    \begin{equation*}
        |\int_{\Omega}h\mathcal{D}\div(\mathcal{D}\mathbf{u}_{1})d\mathbf{x}|\leq C_{27}h\int_{\Omega}\mathcal{D}^2|\div\mathbf{u}_{1}|d\mathbf{x}\leq C_{27}h\|\mathbf{u}\|_{3}\|\mathcal{D}\|_{L^{2}(\Omega)}^{2},
    \end{equation*}
    \begin{align*}
        |\int_{\Omega}h\mathcal{D}\div(\varphi_{1}\mathbf{V})d\mathbf{x}|\leq&Ch\left(\int_{\Omega}|\mathcal{D}\nabla\varphi_{1}\cdot\boldsymbol{\mathcal{V}}|d\mathbf{x}+\int_{\Omega}|\mathcal{D}\varphi_{1}\div\mathcal{V}|d\mathbf{x}\right)\\
        \leq&C_{28}h\left(\|\mathcal{D}\|_{L^{2}(\Omega)}\|\nabla\varphi_{1}\|_{L^{4}(\Omega)}\|\mathcal{V}\|_{L^{4}(\Omega)}+\|\mathcal{D}\|_{L^{2}(\Omega)}\|\varphi_{1}\|_{L^{\infty}(\Omega)}\|\div\mathcal{V}\|_{L^{2}(\Omega)}\right)\\
        \leq&C_{28}h\left(\|\mathcal{D}\|_{L^{2}(\Omega)}\|\varphi_{1}\|_{H^{2}(\Omega)}\|\nabla\mathcal{V}\|_{L^{2}(\Omega)}+\|\mathcal{D}\|_{L^{2}(\Omega)}\|\varphi_{1}\|_{H^{2}(\Omega)}\|\div\mathcal{V}\|_{L^{2}(\Omega)}\right),
    \end{align*}
    and
    \begin{equation*}
        |\int_{\Omega}\mathcal{D}\mathcal{H}d\mathbf{x}|\leq \frac{1}{2}\|\mathcal{D}\|_{L^{2}(\Omega)}^{2}+C_{29}\|\mathcal{H}\|_{L^{2}(\Omega)}^{2}.
    \end{equation*}
    Therefore it holds that
    \begin{equation}\label{EstimateContinuityDensity}
        \|\mathcal{D}\|_{L^{2}(\Omega)}^2\leq C_{30}(\|\mathcal{D}'\|_{L^{2}(\Omega)}^2+\|\boldsymbol{\mathcal{V}}'\|_{H^{1}(\Omega)}^2+\|\mathcal{T}'\|_{H^{1}(\Omega)}^2).
    \end{equation}

    Combing \eqref{EstimateContinuityVelocityTemperautre} and \eqref{EstimateContinuityDensity}, we obtain the desired estimate \eqref{EstimateContinuity}.
\end{proof}

\begin{remark}
    The inequality \eqref{EstimateContinuity} implies that $\mathds{T}$ is continuous in $L_{0}^{2}\times V^{1}(\Omega)\times H^{1}(\Omega)$.
\end{remark}

\section{Boundedness of solution operator $\mathds{T}$}
In this section, we prove that the operator $\mathds{T}$ is bounded. 

\begin{theorem}\label{boundednessTheorem}
    Under the assumption of Corollary \ref{existenceEllipticSystem} and Corollary \ref{ExistenceTransportEquation}, there exists a constant $\varepsilon_{0}>0$ such that for $h, \varepsilon\in[0,\varepsilon_{0})$, if the wall temperature $\theta_{w}\in H^{\vartheta}(\partial\Omega)$ with $\|\theta_{w}-\overline{\theta}\|_{H^{\vartheta}(\partial\Omega)}\leq \varepsilon$ where $\vartheta>\frac{7}{2}$, and $(\varphi',\mathbf{u}',\zeta')\in (H^{2}(\Omega)\cap L_{0}^{2}(\Omega))\times V^{3}(\Omega)\times H^{3}(\Omega)$ with $N_{2}(\varphi',\mathbf{u}',\zeta')\leq\varepsilon$, then $N_{2}(\varphi,\mathbf{u},\zeta)\leq\varepsilon$.
\end{theorem}

\begin{remark}
    Theorem \ref{boundednessTheorem} implies that operator $\mathds{T}$ is bounded on $H^{2}(\Omega)\cap H^{3}(\Omega)\cap H^{3}(\Omega)$.
\end{remark}

The above result will be proved after we derive a series of estimates in the following.

\begin{lemma}\label{basicEstimateLemma}
    Under the assumption of Corollary \ref{existenceEllipticSystem} and Corollary \ref{ExistenceTransportEquation}, for an arbitrary fixed constant $h>0$, there exist constants $\delta_{0}, \varepsilon_{0}>0$ such that for $\varepsilon\in[0,\delta_{0}), \varepsilon\in[0,\varepsilon_{0})$, if the wall temperature $\theta_{w}\in H^{\vartheta}(\partial\Omega)$ with $\|\theta_{w}-\overline{\theta}\|_{H^{\vartheta}(\partial\Omega)}\leq \delta$ where $\vartheta>\frac{7}{2}$, and $(\varphi',\mathbf{u}',\zeta')\in (H^{2}(\Omega)\cap L_{0}^{2}(\Omega))\times V^{3}(\Omega)\times H^{3}(\Omega)$ with $N_{2}(\varphi',\mathbf{u}',\zeta')\leq\varepsilon$, then there exist constants $\eta, C>0$ only depending on $a_{\mathbf{u}}^{I}, a_{\mathbf{u}}^{II}, a_{\theta}^{I}, a_{\theta}^{II}, c_{v}, R, \overline{\rho}, \overline{\theta}$ such that
    \begin{align}\label{basicEstimate}
        &\frac{R}{4h}\frac{\overline{\theta}}{\overline{\rho}}\sqrt{\frac{2}{R}}\frac{2}{3R}a_{\mathbf{u}}^{I}a_{\mathbf{u}}^{II}\frac{\mu(\overline{\theta})}{\sqrt{\overline{\theta}}}(\|\varphi\|_{L^{2}(\Omega)}^2-\|\varphi'\|_{L^{2}(\Omega)}^2+\|\varphi-\varphi'\|_{L^{2}(\Omega)}^2) \nonumber\\
        &+\left(\frac{1}{3R}\sqrt{\frac{2}{R}}a_{\mathbf{u}}^{I}a_{\mathbf{u}}^{II}\frac{\mu(\overline{\theta})^{2}}{\sqrt{\overline{\theta}}}-C\delta\right)\left\|\nabla\mathbf{u}+\nabla\mathbf{u}^{\mathsf{T}}-\frac{2}{3}\div\mathbf{u}\mathbb{I}_{3}\right\|_{L^{2}(\Omega)}^{2} \nonumber\\
        &+\left(\frac{32}{75R^2}\sqrt{\frac{2}{R}}a_{\theta}^{I}a_{\theta}^{II}\frac{\kappa(\overline{\theta})^{2}}{\overline{\theta}\sqrt{\overline{\theta}}}-C\delta\right)\|\nabla\zeta\|_{L^{2}(\Omega)}^2 \nonumber\\
        &+\left(\frac{2}{3R}a_{\mathbf{u}}^{II}\mu(\overline{\theta})\overline{\rho}-C\delta-C\varepsilon\right)\|\mathbf{u}\|_{L^{2}(\partial\Omega)}^2+\left(\frac{16}{15R}a_{\theta}^{I}\frac{\kappa(\overline{\theta})}{\overline{\theta}}\overline{\rho}-C\delta-C\varepsilon\right)\|\zeta\|_{L^{2}(\partial\Omega)}^2 \nonumber\\
        \leq&\left(\frac{2}{5R^2}a_{\mathbf{u}}^{II}a_{\theta}^I\frac{\kappa(\overline{\theta})}{\overline{\theta}}+C\delta\right)\left\|\nabla\mathbf{u}+\nabla\mathbf{u}^{\mathsf{T}}-\frac{2}{3}\div\mathbf{u}\mathbb{I}_{3}\right\|_{L^{2}(\Omega)}\|\nabla\zeta\|_{L^{2}(\Omega)} \nonumber\\
        &+\left(R\frac{4}{5R^2}a_{\mathbf{u}}^{II}a_\theta^I\kappa(\overline{\theta})+C\delta\right)\|\zeta\|_{L^{2}(\Omega)}\|\partial_{1}\varphi'\|_{L^{2}(\Omega)} \nonumber\\
        &+C\delta(\|\nabla\varphi'\|_{L^{2}(\Omega)}^{2}+\|\mathbf{u}\|_{L^{2}(\Omega)}^{2})+Ch\|\nabla\mathbf{u}\|_{L^{2}(\Omega)}^{2} \nonumber\\
        &+\eta(\|\mathbf{u}\|_{L^{2}(\Omega)}+\|\zeta\|_{L^{2}(\Omega)}^{2})+ \frac{C}{\eta}(1+N_{1}(\varphi',\mathbf{u}',\zeta'))^{2}(\sqrt{\delta}+N_{2}(\varphi',\mathbf{u}',\zeta'))^{4}.
    \end{align}
\end{lemma}

\begin{proof}
    Multiplying both sides of $\eqref{LSNS}_{1}$ by $R\frac{\overline{\theta}}{\overline{\rho}}\sqrt{\frac{2}{R}}\frac{2}{3R}a_{\mathbf{u}}^{I}a_{\mathbf{u}}^{II}\frac{\mu(\overline{\theta})}{\sqrt{\overline{\theta}}}\varphi$, we then have
    \begin{align*}
        &\underbrace{\int_{\Omega}R\frac{\overline{\theta}}{\overline{\rho}}\sqrt{\frac{2}{R}}\frac{2}{3R}a_{\mathbf{u}}^{I}a_{\mathbf{u}}^{II}\frac{\mu(\overline{\theta})}{\sqrt{\overline{\theta}}}\varphi\frac{\varphi-\varphi'}{h}d\mathbf{x}}_{I_1}+\underbrace{\int_{\Omega}R\frac{\overline{\theta}}{\overline{\rho}}\sqrt{\frac{2}{R}}\frac{2}{3R}a_{\mathbf{u}}^{I}a_{\mathbf{u}}^{II}\frac{\mu(\overline{\theta})}{\sqrt{\overline{\theta}}}\varphi\div(\varphi\mathbf{u})d\mathbf{x}}_{I_2} \\
        &+\underbrace{\int_{\Omega}R\overline{\theta}\sqrt{\frac{2}{R}}\frac{2}{3R}a_{\mathbf{u}}^{I}a_{\mathbf{u}}^{II}\frac{\mu(\overline{\theta})}{\sqrt{\overline{\theta}}}\varphi\div\mathbf{u}d\mathbf{x}}_{I_3}=0.
    \end{align*}
    And by the variational form,
    \begin{align*}
        &\underbrace{\int_{\Omega}\frac{2}{3R}\sqrt{\frac{2}{R}}a_{\mathbf{u}}^{I}a_{\mathbf{u}}^{II}\frac{\mu(\tilde{\theta})^{2}}{\sqrt{\tilde{\theta}}}\nabla\mathbf{u}:\left(\nabla\mathbf{u}+\nabla\mathbf{u}^{\mathsf{T}}-\frac{2}{3}\div\mathbf{u}\mathbb{I}_{3}\right)d\mathbf{x}}_{I_{4}} +\underbrace{\int_{\Omega}\frac{32}{75R^2}\sqrt{\frac{2}{R}}a_{\theta}^{I}a_{\theta}^{II}\frac{\kappa(\tilde{\theta})^{2}}{\tilde{\theta}\sqrt{\tilde{\theta}}}\nabla\zeta\cdot\nabla\zeta d\mathbf{x}}_{I_{5}} \\
        &+\underbrace{\int_{\Omega}-\frac{4}{5R^2}a_{\mathbf{u}}^{II}a_{\theta}^I\frac{\mu(\tilde{\theta})\kappa(\tilde{\theta})}{\mu(\overline{\theta})\tilde{\theta}}\nabla\zeta\cdot\mathfrak{n}\left(\nabla u_{1}+\partial_{1}\mathbf{u}-\frac{2}{3}(\div\mathbf{u})\mathbf{e}_{1}\right)d\mathbf{x}}_{I_{6}} \\
        &+\underbrace{\int_{\Omega}\frac{2}{3R}\sqrt{\frac{2}{R}}a_{\mathbf{u}}^{I}a_{\mathbf{u}}^{II}\nabla\left(\frac{\mu(\tilde{\theta})^{2}}{\sqrt{\tilde{\theta}}}\right)\otimes\mathbf{u}:\left(\nabla\mathbf{u}+\nabla\mathbf{u}^{\mathsf{T}}-\frac{2}{3}\div\mathbf{u}\mathbb{I}_{3}\right)d\mathbf{x}}_{I_{7}} \nonumber\\
        &+\underbrace{\int_{\Omega}\frac{32}{75R^2}\sqrt{\frac{2}{R}}a_{\theta}^{I}a_{\theta}^{II}\nabla\left(\frac{\kappa(\tilde{\theta})\kappa(\tilde{\theta})}{\tilde{\theta}\sqrt{\tilde{\theta}}}\right)\zeta\cdot\nabla\zeta d\mathbf{x}}_{I_{8}} \\
        &+\underbrace{\int_{\Omega}-\frac{4}{5R^2}a_{\mathbf{u}}^{II}a_{\theta}^I\nabla\left(\frac{\mu(\tilde{\theta})\kappa(\tilde{\theta})}{\mu(\overline{\theta})\tilde{\theta}}\right)\cdot\zeta\mathfrak{n}\left(\nabla u_{1}+\partial_{1}\mathbf{u}-\frac{2}{3}(\div\mathbf{u})\mathbf{e}_{1}\right)d\mathbf{x}}_{I_{9}} \\
        &+\underbrace{\int_{\Omega}R\overline{\rho}\sqrt{\frac{2}{R}}\frac{2}{3R}a_{\mathbf{u}}^{I}a_{\mathbf{u}}^{II}\frac{\mu(\tilde{\theta})^{2}}{\mu(\overline{\theta})\sqrt{\tilde{\theta}}}\mathbf{u}\cdot\nabla\zeta d\mathbf{x}}_{I_{10}} +\underbrace{\int_{\Omega}\frac{32}{75R^2}\sqrt{\frac{2}{R}}a_{\theta}^{I}a_{\theta}^{II}\frac{\kappa(\tilde{\theta})^{2}}{\kappa(\overline{\theta})\tilde{\theta}\sqrt{\tilde{\theta}}}(R\overline{\rho}\overline{\theta}+\mathcal{R}\overline{\rho})\zeta\div\mathbf{u}d\mathbf{x}}_{I_{11}} \\
        &+\underbrace{\int_{\Omega}-\frac{4}{5R}\overline{\rho}a_{\mathbf{u}}^{II}a_{\theta}^I\frac{\mu(\tilde{\theta})\kappa(\tilde{\theta})}{\mu(\overline{\theta})\tilde{\theta}}\zeta\mathfrak{n}\partial_{1}\zeta d\mathbf{x}}_{I_{12}} \\
        &+\underbrace{\sum_{p=0,1}\int_{\mathbb{T}_{p}^{2}}\frac{2}{3R}a_{\mathbf{u}}^{II}\mu(\theta_{w})(\varphi'+\overline{\rho})|\mathbf{u}|^{2}dx_{2}dx_{3}}_{I_{13}} +\underbrace{\sum_{p=0,1}\int_{\mathbb{T}_{p}^{2}}\frac{16}{15R}a_{\theta}^{I}\frac{\kappa(\theta_{w})}{\theta_{w}}(\varphi'+\overline{\rho})|\zeta|^{2}dx_{2}dx_{3}}_{I_{14}} \\
        &+\underbrace{\sum_{p=0,1}\int_{\mathbb{T}_{p}^{2}}\frac{8}{15R^2}a_{\mathbf{u}}^{II}a_{\theta}^{I}\mathbf{u}\cdot\nabla\tilde{\theta}dx_{2}dx_{3}}_{I_{15}} +\underbrace{\sum_{p=0,1}\int_{\mathbb{T}_{p}^{2}}\frac{32}{75R^2}a_{\theta}^{I}a_{\theta}^{II}\frac{\kappa(\theta_{w})^2}{\theta_{w}\sqrt{\theta_{w}}}\zeta\mathfrak{n}\partial_{1}\tilde{\theta}dx_{2}dx_{3}}_{I_{16}} \\
        = 
        &\underbrace{\int_{\Omega}\frac{2}{3R}\sqrt{\frac{2}{R}}a_{\mathbf{u}}^{I}a_{\mathbf{u}}^{II}\frac{\mu(\tilde{\theta})^{2}}{\mu(\overline{\theta})\sqrt{\tilde{\theta}}}\mathbf{u}\cdot\mathbf{F}(\varphi',\mathbf{u}',\zeta')d\mathbf{x}}_{I_{16}}+\underbrace{\int_{\Omega}-R\tilde{\theta}\sqrt{\frac{2}{R}}\frac{2}{3R}a_{\mathbf{u}}^{I}a_{\mathbf{u}}^{II}\frac{\mu(\tilde{\theta})^{2}}{\mu(\overline{\theta})\sqrt{\tilde{\theta}}}\mathbf{u}\cdot\nabla\varphi'd\mathbf{x}}_{I_{17}} \\
        &+\underbrace{\int_{\Omega}\frac{32}{75R^2}\sqrt{\frac{2}{R}}a_{\theta}^{I}a_{\theta}^{II}\frac{\kappa(\tilde{\theta})^{2}}{\kappa(\overline{\theta})\tilde{\theta}\sqrt{\tilde{\theta}}}\zeta G(\varphi',\mathbf{u}',\zeta')d\mathbf{x}}_{I_{18}}+\underbrace{\int_{\Omega}-\frac{32}{75R^2}\sqrt{\frac{2}{R}}a_{\theta}^{I}a_{\theta}^{II}\frac{\kappa(\tilde{\theta})^{2}}{\kappa(\overline{\theta})\tilde{\theta}\sqrt{\tilde{\theta}}}\mathcal{R}\zeta\div(\varphi'\mathbf{u}')d\mathbf{x}}_{I_{19}} \\
        &+\underbrace{\int_{\Omega}-\frac{4}{5R^2}a_{\mathbf{u}}^{II}a_{\theta}^I\frac{\mu(\tilde{\theta})\kappa(\tilde{\theta})}{\mu(\overline{\theta})\tilde{\theta}}\zeta\mathfrak{n}F_{1}(\varphi',\mathbf{u}',\zeta')d\mathbf{x}}_{I_{20}} +\underbrace{\int_{\Omega}\frac{4}{5R}\overline{\theta}a_{\mathbf{u}}^{II}a_{\theta}^I\frac{\mu(\tilde{\theta})\kappa(\tilde{\theta})}{\mu(\overline{\theta})\tilde{\theta}}\zeta\mathfrak{n}\partial_{1}\varphi' d\mathbf{x}}_{I_{21}}.
    \end{align*}
  
    For $I_{1}$, 
    \begin{equation*}
        I_{1}\geq \frac{R}{2h}\frac{\overline{\theta}}{\overline{\rho}}\sqrt{\frac{2}{R}}\frac{2}{3R}a_{\mathbf{u}}^{I}a_{\mathbf{u}}^{II}\frac{\mu(\overline{\theta})}{\sqrt{\overline{\theta}}}(\|\varphi\|_{L^{2}(\Omega)}^{2}-\|\varphi'\|_{L^{2}(\Omega)}^{2}+\|\varphi-\varphi'\|_{L^{2}(\Omega)}^{2}).
    \end{equation*}

    For $I_{2}$, 
    \begin{equation*}
        I_{2}
        \leq C_{1}\int_{\Omega}|\div\mathbf{u}\cdot\varphi^2|d\mathbf{x}\\
        \leq C_{1}\|\varphi\|_{L^3(\Omega)}\|\varphi\|_{L^6(\Omega)}\|\nabla\mathbf{u}\|_{L^2(\Omega)}\\
        \leq C_{1}\|\varphi\|_{H^{1}(\Omega)}^{2}\|\nabla\mathbf{u}\|_{L^{2}(\Omega)},
    \end{equation*}
    where $C_{1}>0$ is a constant only depending on $R, a_{\mathbf{u}}^{I}, a_{\mathbf{u}}^{II}, \overline{\rho}, \overline{\theta}$.

    For $I_{3}$ and $I_{17}$, 
    \begin{align*}
        &|I_{3}+I_{17}| \\
        \leq&\delta\int_{\Omega}|\mathbf{u}\cdot\nabla\varphi'|d\mathbf{x}+R\overline{\theta}\sqrt{\frac{2}{R}}\frac{2}{3R}a_{\mathbf{u}}^{I}a_{\mathbf{u}}^{II}\frac{\mu(\overline{\theta})}{\sqrt{\overline{\theta}}}\int_{\Omega}|(\varphi-\varphi')\div\mathbf{u}|d\mathbf{x} \\
        \leq&\delta(\|\nabla\varphi'\|_{L^{2}(\Omega)}^{2}+\|\mathbf{u}\|_{L^{2}(\Omega)}^{2})+\frac{h}{R}\overline{\rho}\overline{\theta}\sqrt{\frac{2}{R}}\frac{2}{3R}a_{\mathbf{u}}^{I}a_{\mathbf{u}}^{II}\frac{\mu(\overline{\theta})}{\sqrt{\overline{\theta}}}\|\nabla\mathbf{u}\|_{L^{2}(\Omega)}^{2} \\
        &+\frac{R}{4h}\frac{\overline{\theta}}{\overline{\rho}}\sqrt{\frac{2}{R}}\frac{2}{3R}a_{\mathbf{u}}^{I}a_{\mathbf{u}}^{II}\frac{\mu(\overline{\theta})}{\sqrt{\overline{\theta}}}\|\varphi-\varphi'\|_{L^{2}(\Omega)}^{2}.
    \end{align*}

    The rest of the terms can be estimated similarly as the proof of Lemma \ref{weakExistenceEllipticSystem}, and we omit them for brevity. After collecting all of the results, we can obtain the desired estimate \eqref{basicEstimate}.
\end{proof}

\begin{lemma}\label{H1ConormalBasicEstimateLemma}
    Under the assumption of Corollary \ref{existenceEllipticSystem} and Corollary \ref{ExistenceTransportEquation}, for an arbitrary fixed constant $h>0$, there exist constants $\delta_{0}, \varepsilon_{0}>0$ such that for $\varepsilon\in[0,\delta_{0}), \varepsilon\in[0,\varepsilon_{0})$, if the wall temperature $\theta_{w}\in H^{\vartheta}(\partial\Omega)$ with $\|\theta_{w}-\overline{\theta}\|_{H^{\vartheta}(\partial\Omega)}\leq \delta$ where $\vartheta>\frac{7}{2}$, and $(\varphi',\mathbf{u}',\zeta')\in (H^{2}(\Omega)\cap L_{0}^{2}(\Omega))\times V^{3}(\Omega)\times H^{3}(\Omega)$ with $N_{2}(\varphi',\mathbf{u}',\zeta')\leq\varepsilon$, then there exist constants $\eta, C>0$ only depending on $a_{\mathbf{u}}^{I}, a_{\mathbf{u}}^{II}, a_{\theta}^{I}, a_{\theta}^{II}, c_{v}, R, \overline{\rho}, \overline{\theta}$ such that
    \begin{align}\label{H1ConormalBasicEstimate}
        &\frac{R}{4h}\frac{\overline{\theta}}{\overline{\rho}}\sqrt{\frac{2}{R}}\frac{2}{3R}a_{\mathbf{u}}^{I}a_{\mathbf{u}}^{II}\frac{\mu(\overline{\theta})}{\sqrt{\overline{\theta}}}(\|\partial_{i}\varphi\|_{L^{2}(\Omega)}^2-\|\partial_{i}\varphi'\|_{L^{2}(\Omega)}^2+\|\partial_{i}\varphi-\partial_{i}\varphi'\|_{L^{2}(\Omega)}^2) \nonumber\\
        &+\left(\frac{1}{3R}\sqrt{\frac{2}{R}}a_{\mathbf{u}}^{I}a_{\mathbf{u}}^{II}\frac{\mu(\overline{\theta})^{2}}{\sqrt{\overline{\theta}}}-C\delta\right)\left\|\partial_{i}\nabla\mathbf{u}+\partial_{i}\nabla\mathbf{u}^{\mathsf{T}}-\frac{2}{3}\partial_{i}\div\mathbf{u}\mathbb{I}_{3}\right\|_{L^{2}(\Omega)}^{2} \nonumber\\
        &+\left(\frac{32}{75R^2}\sqrt{\frac{2}{R}}a_{\theta}^{I}a_{\theta}^{II}\frac{\kappa(\overline{\theta})^{2}}{\overline{\theta}\sqrt{\overline{\theta}}}-C\delta\right)\|\nabla\partial_{i}\zeta\|_{L^{2}(\Omega)}^2 \nonumber\\
        &+\left(\frac{2}{3R}a_{\mathbf{u}}^{II}\mu(\overline{\theta})\overline{\rho}-C\delta-C\varepsilon\right)\|\partial_{i}\mathbf{u}\|_{L^{2}(\partial\Omega)}^2+\left(\frac{16}{15R}a_{\theta}^{I}\frac{\kappa(\overline{\theta})}{\overline{\theta}}\overline{\rho}-C\delta-C\varepsilon\right)\|\partial_{i}\zeta\|_{L^{2}(\partial\Omega)}^2 \nonumber\\
        \leq&\left(\frac{2}{5R^2}a_{\mathbf{u}}^{II}a_{\theta}^I\frac{\kappa(\overline{\theta})}{\overline{\theta}}+C\delta\right)\left\|\partial_{i}\nabla\mathbf{u}+\partial_{i}\nabla\mathbf{u}^{\mathsf{T}}-\frac{2}{3}\partial_{i}\div\mathbf{u}\mathbb{I}_{3}\right\|_{L^{2}(\Omega)}\|\partial_{i}\nabla\zeta\|_{L^{2}(\Omega)} \nonumber\\
        &+\left(R\frac{4}{5R^2}a_{\mathbf{u}}^{II}a_\theta^I\kappa(\overline{\theta})+C\delta\right)\|\partial_{i}^{2}\zeta\|_{L^{2}(\Omega)}\|\partial_{1}\varphi'\|_{L^{2}(\Omega)} \nonumber\\
        &+C\delta(\|\nabla\partial_{i}\varphi'\|_{L^{2}(\Omega)}^{2}+\|\partial_{i}\mathbf{u}\|_{L^{2}(\Omega)}^{2})+Ch\|\nabla\partial_{i}\mathbf{u}\|_{L^{2}(\Omega)}^{2} \nonumber\\
        &+\eta\sum_{j=1,2}(\|\partial_{i}^{j}\mathbf{u}\|_{L^{2}(\Omega)}^{2}+\|\partial_{i}^{j}\zeta\|_{L^{2}(\Omega)}^{2})+\frac{C}{\eta}(1+N_{1}(\varphi,\mathbf{u},\zeta))^{2}(\sqrt{\delta}+N_{2}(\varphi,\mathbf{u},\zeta))^{4},
    \end{align}
    for $i=2,3$.
\end{lemma}

\begin{proof}
    Applying $\partial_{i}$ to both sides of $\eqref{LSNS}_{1}$ and multiplying the resultant by $R\frac{\overline{\theta}}{\overline{\rho}}\sqrt{\frac{2}{R}}\frac{2}{3R}a_{\mathbf{u}}^{I}a_{\mathbf{u}}^{II}\frac{\mu(\overline{\theta})}{\sqrt{\overline{\theta}}}\partial_{i}\varphi$,
    \begin{align*}
        &\underbrace{\int_{\Omega}R\frac{\overline{\theta}}{\overline{\rho}}\sqrt{\frac{2}{R}}\frac{2}{3R}a_{\mathbf{u}}^{I}a_{\mathbf{u}}^{II}\frac{\mu(\overline{\theta})}{\sqrt{\overline{\theta}}}\partial_{i}\varphi\frac{\partial_{i}\varphi-\partial_{i}\varphi'}{h}d\mathbf{x}}_{I_1}+\underbrace{\int_{\Omega}R\frac{\overline{\theta}}{\overline{\rho}}\sqrt{\frac{2}{R}}\frac{2}{3R}a_{\mathbf{u}}^{I}a_{\mathbf{u}}^{II}\frac{\mu(\overline{\theta})}{\sqrt{\overline{\theta}}}\partial_{i}\varphi\div\partial_{i}(\varphi\mathbf{u})d\mathbf{x}}_{I_2} \\
        &+\underbrace{\int_{\Omega}R\overline{\theta}\sqrt{\frac{2}{R}}\frac{2}{3R}a_{\mathbf{u}}^{I}a_{\mathbf{u}}^{II}\frac{\mu(\overline{\theta})}{\sqrt{\overline{\theta}}}\partial_{i}\varphi\div\partial_{i}\mathbf{u}d\mathbf{x}}_{I_3}=0.
    \end{align*}
    And by the variational form,
    \begin{align*}
        &\underbrace{\int_{\Omega}\frac{2}{3R}\sqrt{\frac{2}{R}}a_{\mathbf{u}}^{I}a_{\mathbf{u}}^{II}\frac{\mu(\tilde{\theta})^{2}}{\sqrt{\tilde{\theta}}}\nabla\partial_{i}\mathbf{u}:\left(\nabla\partial_{i}\mathbf{u}+\nabla\partial_{i}\mathbf{u}^{\mathsf{T}}-\frac{2}{3}\div\partial_{i}\mathbf{u}\mathbb{I}_{3}\right)d\mathbf{x}}_{I_{4}} \nonumber\\
        &+\underbrace{\int_{\Omega}\frac{32}{75R^2}\sqrt{\frac{2}{R}}a_{\theta}^{I}a_{\theta}^{II}\frac{\kappa(\tilde{\theta})^{2}}{\tilde{\theta}\sqrt{\tilde{\theta}}}\nabla\partial_{i}\zeta\cdot\nabla\partial_{i}\zeta d\mathbf{x}}_{I_{5}} \\
        &+\underbrace{\int_{\Omega}-\frac{4}{5R^2}a_{\mathbf{u}}^{II}a_{\theta}^I\frac{\mu(\tilde{\theta})\kappa(\tilde{\theta})}{\mu(\overline{\theta})\tilde{\theta}}\nabla\partial_{i}\zeta\cdot\mathfrak{n}\left(\nabla\partial_{i}u_{1}+\partial_{i}\partial_{1}\mathbf{u}-\frac{2}{3}(\partial_{i}\div\mathbf{u})\mathbf{e}_{1}\right)d\mathbf{x}}_{I_{6}} \\
        &+\underbrace{\int_{\Omega}\frac{2}{3R}\sqrt{\frac{2}{R}}a_{\mathbf{u}}^{I}a_{\mathbf{u}}^{II}\nabla\left(\frac{\mu(\tilde{\theta})^{2}}{\sqrt{\tilde{\theta}}}\right)\otimes\partial_{i}\mathbf{u}:\left(\nabla\partial_{i}\mathbf{u}+\nabla\partial_{i}\mathbf{u}^{\mathsf{T}}-\frac{2}{3}\div\partial_{i}\mathbf{u}\mathbb{I}_{3}\right)d\mathbf{x}}_{I_{7}} \nonumber\\ &+\underbrace{\int_{\Omega}\frac{32}{75R^2}\sqrt{\frac{2}{R}}a_{\theta}^{I}a_{\theta}^{II}\nabla\left(\frac{\kappa(\tilde{\theta})\kappa(\tilde{\theta})}{\tilde{\theta}\sqrt{\tilde{\theta}}}\right)\partial_{i}\zeta\cdot\nabla\partial_{i}\zeta d\mathbf{x}}_{I_{8}} \\
        &+\underbrace{\int_{\Omega}-\frac{4}{5R^2}a_{\mathbf{u}}^{II}a_{\theta}^I\nabla\left(\frac{\mu(\tilde{\theta})\kappa(\tilde{\theta})}{\mu(\overline{\theta})\tilde{\theta}}\right)\cdot\partial_{i}\zeta\cdot\mathfrak{n}\left(\partial_{i}\nabla u_{1}+\partial_{i}\nabla\mathbf{u}-\frac{2}{3}(\partial_{i}\div\mathbf{u})\mathbf{e}_{1}\right)d\mathbf{x}}_{I_{9}} \\
        &+\underbrace{\int_{\Omega}R\overline{\rho}\sqrt{\frac{2}{R}}\frac{2}{3R}a_{\mathbf{u}}^{I}a_{\mathbf{u}}^{II}\frac{\mu(\tilde{\theta})^{2}}{\mu(\overline{\theta})\sqrt{\tilde{\theta}}}\partial_{i}\mathbf{u}\cdot\nabla\partial_{i}\zeta d\mathbf{x}}_{I_{10}} +\underbrace{\int_{\Omega}\frac{32}{75R^2}\sqrt{\frac{2}{R}}a_{\theta}^{I}a_{\theta}^{II}\frac{\kappa(\tilde{\theta})^{2}}{\kappa(\overline{\theta})\tilde{\theta}\sqrt{\tilde{\theta}}}(R\overline{\rho}\overline{\theta}+\mathcal{R}\overline{\rho})\partial_{i}\zeta\div\partial_{i}\mathbf{u}d\mathbf{x}}_{I_{11}} \\
        &+\underbrace{\int_{\Omega}-\frac{4}{5R}\overline{\rho}a_{\mathbf{u}}^{II}a_{\theta}^I\frac{\mu(\tilde{\theta})\kappa(\tilde{\theta})}{\mu(\overline{\theta})\tilde{\theta}}\partial_{i}\zeta\cdot\mathfrak{n}\partial_{i}\partial_{1}\zeta d\mathbf{x}}_{I_{12}} \\
        &+\underbrace{\sum_{p=0,1}\int_{\mathbb{T}_{p}^{2}}\frac{2}{3R}a_{\mathbf{u}}^{II}\mu(\theta_{w})(\varphi'+\overline{\rho})|\partial_{i}\mathbf{u}|^{2}dx_{2}dx_{3}}_{I_{13}} +\underbrace{\sum_{p=0,1}\int_{\mathbb{T}_{p}^{2}}\frac{16}{15R}a_{\theta}^{I}\frac{\kappa(\theta_{w})}{\theta_{w}}(\varphi'+\overline{\rho})|\partial_{i}\zeta|^{2}dx_{2}dx_{3}}_{I_{14}} \\
        &+\underbrace{\sum_{p=0,1}\int_{\mathbb{T}_{p}^{2}}\frac{16}{15R}a_{\theta}^{I}\frac{\kappa(\theta_{w})}{\theta_{w}}(\varphi'+\overline{\rho})\partial_{i}\zeta(\partial_{i}\tilde{\theta}-\partial_{i}\theta_{w})dx_{2}dx_{3}}_{I_{15}} \\
        &+\underbrace{\sum_{p=0,1}\int_{\mathbb{T}_{p}^{2}}\frac{2}{3R}a_{\mathbf{u}}^{II}\mu(\theta_{w})\partial_{i}\varphi'\cdot\partial_{i}\mathbf{u}\cdot\mathbf{u}dx_{2}dx_{3}}_{I_{16}} +\underbrace{\sum_{p=0,1}\int_{\mathbb{T}_{p}^{2}}\frac{16}{15R}a_{\theta}^{I}\frac{\kappa(\theta_{w})}{\theta_{w}}\partial_{i}\varphi'\cdot\zeta\partial_{i}\zeta dx_{2}dx_{3}}_{I_{17}} \\
        &+\underbrace{\sum_{p=0,1}\int_{\mathbb{T}_{p}^{2}}\frac{8}{15R^2}a_{\mathbf{u}}^{II}a_{\theta}^{I}\partial_{i}\mathbf{u}\cdot\nabla\partial_{i}\tilde{\theta}dx_{2}dx_{3}}_{I_{18}} +\underbrace{\sum_{p=0,1}\int_{\mathbb{T}_{p}^{2}}\frac{32}{75R^2}a_{\theta}^{I}a_{\theta}^{II}\frac{\kappa(\theta_{w})^2}{\theta_{w}\sqrt{\theta_{w}}}\partial_{i}\zeta\cdot\mathfrak{n}\partial_{i}\partial_{1}\tilde{\theta}dx_{2}dx_{3}}_{I_{19}} \\
        = 
        &\underbrace{\int_{\Omega}\frac{2}{3R}\sqrt{\frac{2}{R}}a_{\mathbf{u}}^{I}a_{\mathbf{u}}^{II}\frac{\mu(\tilde{\theta})^{2}}{\mu(\overline{\theta})\sqrt{\tilde{\theta}}}\partial_{i}\mathbf{u}\cdot\partial_{i}\mathbf{F}(\varphi',\mathbf{u}',\zeta')d\mathbf{x}}_{I_{20}}+\underbrace{\int_{\Omega}-R\tilde{\theta}\sqrt{\frac{2}{R}}\frac{2}{3R}a_{\mathbf{u}}^{I}a_{\mathbf{u}}^{II}\frac{\mu(\tilde{\theta})^{2}}{\mu(\overline{\theta})\sqrt{\tilde{\theta}}}\partial_{i}\mathbf{u}\cdot\nabla\partial_{i}\varphi'd\mathbf{x}}_{I_{21}} \\
        &+\underbrace{\int_{\Omega}\frac{32}{75R^2}\sqrt{\frac{2}{R}}a_{\theta}^{I}a_{\theta}^{II}\frac{\kappa(\tilde{\theta})^{2}}{\kappa(\overline{\theta})\tilde{\theta}\sqrt{\tilde{\theta}}}\partial_{i}\zeta\cdot\partial_{i} G(\varphi',\mathbf{u}',\zeta')d\mathbf{x}}_{I_{22}}+\underbrace{\int_{\Omega}-\frac{32}{75R^2}\sqrt{\frac{2}{R}}a_{\theta}^{I}a_{\theta}^{II}\frac{\kappa(\tilde{\theta})^{2}}{\kappa(\overline{\theta})\tilde{\theta}\sqrt{\tilde{\theta}}}\mathcal{R}\partial_{i}\zeta\cdot\partial_{i}\div(\varphi'\mathbf{u}')d\mathbf{x}}_{I_{23}} \\
        &+\underbrace{\int_{\Omega}-\frac{4}{5R^2}a_{\mathbf{u}}^{II}a_{\theta}^I\frac{\mu(\tilde{\theta})\kappa(\tilde{\theta})}{\mu(\overline{\theta})\tilde{\theta}}\partial_{i}\zeta\cdot\mathfrak{n}\partial_{i}F_{1}(\varphi',\mathbf{u}',\zeta')d\mathbf{x}}_{I_{24}} +\underbrace{\int_{\Omega}\frac{4}{5R}\overline{\theta}a_{\mathbf{u}}^{II}a_{\theta}^I\frac{\mu(\tilde{\theta})\kappa(\tilde{\theta})}{\mu(\overline{\theta})\tilde{\theta}}\partial_{i}\zeta\cdot\mathfrak{n}\partial_{i}\partial_{1}\varphi' d\mathbf{x}}_{I_{25}}.
    \end{align*}
    
    For $I_{1}$, 
    \begin{equation*}
        I_{1}\geq \frac{R}{2h}\frac{\overline{\theta}}{\overline{\rho}}\sqrt{\frac{2}{R}}\frac{2}{3R}a_{\mathbf{u}}^{I}a_{\mathbf{u}}^{II}\frac{\mu(\overline{\theta})}{\sqrt{\overline{\theta}}}\left(\|\partial_{i}\varphi\|_{L^{2}(\Omega)}^{2}-\|\partial_{i}\varphi'\|_{L^{2}(\Omega)}^{2}+\|\partial_{i}\varphi-\partial_{i}\varphi'\|_{L^{2}(\Omega)}^{2}\right).
    \end{equation*}

    For $I_{2}$, 
    \begin{equation*}
        I_{2}
        \leq C_{1}\int_{\Omega}|\div\partial_{i}\mathbf{u}|\cdot|\partial_{i}\varphi|^2d\mathbf{x}\\
        \leq C_{1}\|\partial_{i}\varphi\|_{L^3(\Omega)}\|\partial_{i}\varphi\|_{L^6(\Omega)}\|\nabla\partial_{i}\mathbf{u}\|_{L^2(\Omega)}\\
        \leq C_{1}\|\partial_{i}\varphi\|_{H^{1}(\Omega)}^{2}\|\nabla\partial_{i}\mathbf{u}\|_{L^{2}(\Omega)},
    \end{equation*}
    where $C_{1}>0$ is a constant only depending on $R, a_{\mathbf{u}}^{I}, a_{\mathbf{u}}^{II}, \overline{\rho}, \overline{\theta}$.

    For $I_{3}$ and $I_{17}$, 
    \begin{align*}
        &|I_{3}+I_{17}| \\
        \leq&\varepsilon\int_{\Omega}|\partial_{i}\mathbf{u}\cdot\nabla\partial_{i}\varphi'|d\mathbf{x}+R\overline{\theta}\sqrt{\frac{2}{R}}\frac{2}{3R}a_{\mathbf{u}}^{I}a_{\mathbf{u}}^{II}\frac{\mu(\overline{\theta})}{\sqrt{\overline{\theta}}}\int_{\Omega}|(\partial_{i}\varphi-\partial_{i}\varphi')\div\partial_{i}\mathbf{u}|d\mathbf{x} \\
        \leq&\varepsilon(\|\nabla\partial_{i}\varphi'\|_{L^{2}(\Omega)}^{2}+\|\partial_{i}\mathbf{u}\|_{L^{2}(\Omega)}^{2})+\frac{h}{R}\overline{\rho}\overline{\theta}\sqrt{\frac{2}{R}}\frac{2}{3R}a_{\mathbf{u}}^{I}a_{\mathbf{u}}^{II}\frac{\mu(\overline{\theta})}{\sqrt{\overline{\theta}}}\|\nabla\partial_{i}\mathbf{u}\|_{L^{2}(\Omega)}^{2} \\
        &+\frac{R}{4h}\frac{\overline{\theta}}{\overline{\rho}}\sqrt{\frac{2}{R}}\frac{2}{3R}a_{\mathbf{u}}^{I}a_{\mathbf{u}}^{II}\frac{\mu(\overline{\theta})}{\sqrt{\overline{\theta}}}\|\partial_{i}\varphi-\partial_{i}\varphi'\|_{L^{2}(\Omega)}^{2}.
    \end{align*}
   
    For $I_{20}$,
    \begin{align*}
        |I_{20}|\leq&\left|\int_{\Omega}\frac{2}{3R}\sqrt{\frac{2}{R}}a_{\mathbf{u}}^{I}a_{\mathbf{u}}^{II}\frac{\mu(\tilde{\theta})^{2}}{\mu(\overline{\theta})\sqrt{\tilde{\theta}}}\partial_{i}^2\mathbf{u}\cdot\mathbf{F}(\varphi',\mathbf{u}',\zeta')d\mathbf{x}\right| \\
        &+\left|\int_{\Omega}\frac{2}{3R}\sqrt{\frac{2}{R}}a_{\mathbf{u}}^{I}a_{\mathbf{u}}^{II}\partial_{i}\left(\frac{\mu(\tilde{\theta})^{2}}{\mu(\overline{\theta})\sqrt{\tilde{\theta}}}\right)\cdot\partial_{i}\mathbf{u}\cdot\mathbf{F}(\varphi',\mathbf{u}',\zeta')d\mathbf{x}\right| \\
        \leq&\eta\sum_{j=1,2}\|\partial_{i}^{j}\mathbf{u}\|_{L^{2}(\Omega)}^{2}+\frac{C_{2}}{\eta}\|\mathbf{F}\|_{L^{2}(\Omega)}^{2},
    \end{align*}
    where $C_{2}>0$ is a constant only depending on $R, a_{\mathbf{u}}^{I}, a_{\mathbf{u}}^{II}, \mu,  \overline{\theta}$.
   
    For $I_{22}$,
    \begin{align*}
        |I_{22}|\leq&\left|\int_{\Omega}\frac{32}{75R^2}\sqrt{\frac{2}{R}}a_{\theta}^{I}a_{\theta}^{II}\frac{\kappa(\tilde{\theta})^{2}}{\kappa(\overline{\theta})\tilde{\theta}\sqrt{\tilde{\theta}}}\partial_{i}^{2}\zeta\cdot G(\varphi',\mathbf{u}',\zeta')d\mathbf{x}\right| \\
        &+\left|\int_{\Omega}\frac{32}{75R^2}\sqrt{\frac{2}{R}}a_{\theta}^{I}a_{\theta}^{II}\partial_{i}\left(\frac{\kappa(\tilde{\theta})^{2}}{\kappa(\overline{\theta})\tilde{\theta}\sqrt{\tilde{\theta}}}\right)\cdot\partial_{i}\zeta\cdot G(\varphi',\mathbf{u}',\zeta')d\mathbf{x}\right| \\
        \leq&\eta\sum_{j=1,2}\|\partial_{i}^{j}\zeta\|_{L^{2}(\Omega)}^{2}+\frac{C_{3}}{\eta}\|G\|_{L^{2}(\Omega)}^{2},
    \end{align*}
    where $C_{3}>0$ is a constant only depending on $R, a_{\theta}^{I}, a_{\theta}^{II}, \kappa, \overline{\theta}$.
   
    For $I_{23}$,
    \begin{align*}
        |I_{23}|\leq&\left|\int_{\Omega}-\frac{32}{75R^2}\sqrt{\frac{2}{R}}a_{\theta}^{I}a_{\theta}^{II}\frac{\kappa(\tilde{\theta})^{2}}{\kappa(\overline{\theta})\tilde{\theta}\sqrt{\tilde{\theta}}}\mathcal{R}\partial_{i}^{2}\zeta\cdot\div(\varphi'\mathbf{u}')d\mathbf{x}\right| \\
        &+\left|\int_{\Omega}-\frac{32}{75R^2}\sqrt{\frac{2}{R}}a_{\theta}^{I}a_{\theta}^{II}\partial_{i}\left(\frac{\kappa(\tilde{\theta})^{2}}{\kappa(\overline{\theta})\tilde{\theta}\sqrt{\tilde{\theta}}}\right)\cdot\mathcal{R}\partial_{i}\zeta\cdot\div(\varphi'\mathbf{u}')d\mathbf{x}\right| \\
        \leq&\eta\sum_{j=1,2}\|\partial_{i}^{j}\zeta\|_{L^{2}(\Omega)}^{2}+\frac{C_{4}}{\eta}\|\div(\varphi'\mathbf{u}')\|_{L^{2}(\Omega)}^{2},
    \end{align*}
    where $C_{4}>0$ is a constant only depending on $R, a_{\theta}^{I}, a_{\theta}^{II}, \overline{\rho}, \overline{\theta}$.
   
    For $I_{24}$,
    \begin{align*}
        |I_{24}|\leq&\left|\int_{\Omega}-\frac{4}{5R^2}a_{\mathbf{u}}^{II}a_{\theta}^I\frac{\mu(\tilde{\theta})\kappa(\tilde{\theta})}{\mu(\overline{\theta})\tilde{\theta}}\partial_{i}^2\zeta\cdot\mathfrak{n}F_{1}(\varphi',\mathbf{u}',\zeta')d\mathbf{x}\right| \\
        &+\left|\int_{\Omega}-\frac{4}{5R^2}a_{\mathbf{u}}^{II}a_{\theta}^I\partial_{i}\left(\frac{\mu(\tilde{\theta})\kappa(\tilde{\theta})}{\mu(\overline{\theta})\tilde{\theta}}\right)\cdot\partial_{i}\zeta\cdot\mathfrak{n}F_{1}(\varphi',\mathbf{u}',\zeta')d\mathbf{x}\right| \\
        \leq&\eta\sum_{j=1,2}\|\partial_{i}^{j}\zeta\|_{L^{2}(\Omega)}^{2}+\frac{C_{5}}{\eta}\|F_{1}\|_{L^{2}(\Omega)}^{2},
    \end{align*}
    where $C_{5}>0$ is a constant only depending on $R, a_{\mathbf{u}}^{II}, a_{\theta}^{I}, \mu, \kappa, \overline{\theta}$.
   
    For $I_{25}$,
    \begin{align*}
        |I_{25}| \leq&\left|\int_{\Omega}R\tilde{\theta}\frac{4}{5R^2}a_{\mathbf{u}}^{II}a_{\theta}^{I}\mathfrak{n}\frac{\mu(\tilde{\theta})\kappa(\tilde{\theta})\tilde{\theta}^{-1}}{\mu(\overline{\theta})}\partial_{i}^2\zeta\cdot\partial_1\varphi'd\mathbf{x}\right| 
        +\left|\int_{\Omega}R\tilde{\theta}\frac{4}{5R^2}a_{\mathbf{u}}^{II}a_{\theta}^{I}\partial_{i}\left(\mathfrak{n}\frac{\mu(\tilde{\theta})\kappa(\tilde{\theta})\tilde{\theta}^{-1}}{\mu(\overline{\theta})}\right)\cdot\partial_{i}\zeta\cdot\partial_1\varphi'd\mathbf{x}\right| \\
        \leq&(R\frac{4}{5R^2}a_{\mathbf{u}}^{II}a_{\theta}^{I}\kappa(\overline{\theta})+C\delta)\|\partial_{i}^2\zeta\|_{L^{2}(\Omega)}\|\partial_{1}\varphi'\|_{L^{2}(\Omega)}+C_{6}\delta(\|\partial_{i}\zeta\|_{L^{2}(\Omega)}^{2}+\|\partial_{1}\varphi'\|_{L^{2}(\Omega)}^{2}),
    \end{align*}
    where $C_{6}>0$ is a constant only depending on $R, a_{\mathbf{u}}^{II}, a_{\theta}^{I}, \mu, \kappa, \overline{\theta}$.

    The rest of the terms can be estimated similarly as the proof of Lemma \ref{regularityEllipticSystem}, and we omit them for brevity. After collecting all of the results, we can obtain the desired estimate \eqref{H1ConormalBasicEstimate}.
\end{proof}

\begin{lemma}\label{H2ConormalBasicEstimateLemma}
    Under the assumption of Corollary \ref{existenceEllipticSystem} and Corollary \ref{ExistenceTransportEquation}, for an arbitrary fixed constant $h>0$, there exist constants $\delta_{0}, \varepsilon_{0}>0$ such that for $\varepsilon\in[0,\delta_{0}), \varepsilon\in[0,\varepsilon_{0})$, if the wall temperature $\theta_{w}\in H^{\vartheta}(\partial\Omega)$ with $\|\theta_{w}-\overline{\theta}\|_{H^{\vartheta}(\partial\Omega)}\leq \delta$ where $\vartheta>\frac{7}{2}$, and $(\varphi',\mathbf{u}',\zeta')\in (H^{2}(\Omega)\cap L_{0}^{2}(\Omega))\times V^{3}(\Omega)\times H^{3}(\Omega)$ with $N_{2}(\varphi',\mathbf{u}',\zeta')\leq\varepsilon$, then there exist constants $\eta, C>0$ only depending on $a_{\mathbf{u}}^{I}, a_{\mathbf{u}}^{II}, a_{\theta}^{I}, a_{\theta}^{II}, c_{v}, R, \overline{\rho}, \overline{\theta}$ such that
    \begin{align}\label{H2ConormalBasicEstimate}
        &\frac{R}{4h}\frac{\overline{\theta}}{\overline{\rho}}\sqrt{\frac{2}{R}}\frac{2}{3R}a_{\mathbf{u}}^{I}a_{\mathbf{u}}^{II}\frac{\mu(\overline{\theta})}{\sqrt{\overline{\theta}}}(\|\partial_{i}^2\varphi\|_{L^{2}(\Omega)}^2-\|\partial_{i}^2\varphi'\|_{L^{2}(\Omega)}^2+\|\partial_{i}^2\varphi-\partial_{i}^2\varphi'\|_{L^{2}(\Omega)}^2) \nonumber\\
        &+\left(\frac{1}{3R}\sqrt{\frac{2}{R}}a_{\mathbf{u}}^{I}a_{\mathbf{u}}^{II}\frac{\mu(\overline{\theta})^{2}}{\sqrt{\overline{\theta}}}-C\delta\right)\left\|\partial_{i}^2\nabla\mathbf{u}+\partial_{i}^{2}\nabla\mathbf{u}^{\mathsf{T}}-\frac{2}{3}\partial_{i}^{2}\div\mathbf{u}\mathbb{I}_{3}\right\|_{L^{2}(\Omega)}^2 \nonumber\\
        &+\left(\frac{32}{75R^2}\sqrt{\frac{2}{R}}a_{\theta}^{I}a_{\theta}^{II}\frac{\kappa(\overline{\theta})^{2}}{\overline{\theta}\sqrt{\overline{\theta}}}-C\delta\right)\|\nabla\partial_{i}^2\zeta\|_{L^{2}(\Omega)}^2 \nonumber\\
        &+\left(\frac{2}{3R}a_{\mathbf{u}}^{II}\mu(\overline{\theta})\overline{\rho}-C\delta-C\varepsilon\right)\|\partial_{i}^2\mathbf{u}\|_{L^{2}(\partial\Omega)}^2+\left(\frac{16}{15R}a_{\theta}^{I}\frac{\kappa(\overline{\theta})}{\overline{\theta}}\overline{\rho}-C\delta-C\varepsilon\right)\|\partial_{i}^2\zeta\|_{L^{2}(\partial\Omega)}^2 \nonumber\\
        \leq&\left(\frac{2}{5R^2}a_{\mathbf{u}}^{II}a_{\theta}^I\frac{\kappa(\overline{\theta})}{\overline{\theta}}+C\delta\right)\left\|\partial_{i}^{2}\nabla\mathbf{u}+\partial_{i}^{2}\nabla\mathbf{u}^{\mathsf{T}}-\frac{2}{3}\partial_{i}^{2}\div\mathbf{u}\mathbb{I}_{3}\right\|_{L^{2}(\Omega)}\|\partial_{i}^{2}\nabla\zeta\|_{L^{2}(\Omega)} \nonumber\\
        &+\left(R\frac{4}{5R^2}a_{\mathbf{u}}^{II}a_\theta^I\kappa(\overline{\theta})+C\delta\right)\|\partial_{i}^{2}\zeta\|_{L^{2}(\Omega)}\|\partial_{i}^{2}\partial_{1}\varphi'\|_{L^{2}(\Omega)} \nonumber\\
        &+C\delta(\|\nabla\partial_{i}^{2}\varphi'\|_{L^{2}(\Omega)}^{2}+\|\partial_{i}\mathbf{u}\|_{L^{2}(\Omega)}^{2})+Ch\|\nabla\partial_{i}^{2}\mathbf{u}\|_{L^{2}(\Omega)}^{2} \nonumber\\
        &+\eta\sum_{j=2,3}(\|\partial_{i}^{j}\mathbf{u}\|_{L^{2}(\Omega)}^{2}+\|\partial_{i}^{j}\zeta\|_{L^{2}(\Omega)}^{2})+\frac{C}{\eta}(1+N_{2}(\varphi,\mathbf{u},\zeta))^{2}(\sqrt{\delta}+N_{2}(\varphi,\mathbf{u},\zeta))^{4},
    \end{align}
    for $i=2,3$.
\end{lemma}
\begin{proof}
    Applying $\partial_{i}^{2}$ to both sides of $\eqref{LSNS}_{1}$ and multiplying the resultant by $R\frac{\overline{\theta}}{\overline{\rho}}\sqrt{\frac{2}{R}}\frac{2}{3R}a_{\mathbf{u}}^{I}a_{\mathbf{u}}^{II}\frac{\mu(\overline{\theta})}{\sqrt{\overline{\theta}}}\partial_{i}^{2}\varphi$,
    \begin{align*}
        &\underbrace{\int_{\Omega}R\frac{\overline{\theta}}{\overline{\rho}}\sqrt{\frac{2}{R}}\frac{2}{3R}a_{\mathbf{u}}^{I}a_{\mathbf{u}}^{II}\frac{\mu(\overline{\theta})}{\sqrt{\overline{\theta}}}\partial_{i}^{2}\varphi\frac{\partial_{i}^{2}\varphi-\partial_{i}^{2}\varphi'}{h}d\mathbf{x}}_{I_1}+\underbrace{\int_{\Omega}R\frac{\overline{\theta}}{\overline{\rho}}\sqrt{\frac{2}{R}}\frac{2}{3R}a_{\mathbf{u}}^{I}a_{\mathbf{u}}^{II}\frac{\mu(\overline{\theta})}{\sqrt{\overline{\theta}}}\partial_{i}^{2}\varphi\div\partial_{i}^{2}(\varphi\mathbf{u})d\mathbf{x}}_{I_2} \\
        &+\underbrace{\int_{\Omega}R\overline{\theta}\sqrt{\frac{2}{R}}\frac{2}{3R}a_{\mathbf{u}}^{I}a_{\mathbf{u}}^{II}\frac{\mu(\overline{\theta})}{\sqrt{\overline{\theta}}}\partial_{i}^{2}\varphi\div\partial_{i}^{2}\mathbf{u}d\mathbf{x}}_{I_3}=0.
    \end{align*}
    And by the variational form,
    \begin{align*}
        &\underbrace{\int_{\Omega}\frac{2}{3R}\sqrt{\frac{2}{R}}a_{\mathbf{u}}^{I}a_{\mathbf{u}}^{II}\frac{\mu(\tilde{\theta})^{2}}{\sqrt{\tilde{\theta}}}\nabla\partial_{i}^2\mathbf{u}:\left(\nabla\partial_{i}^2\mathbf{u}+\nabla\mathbf{u}^{\mathsf{T}}-\frac{2}{3}\div\mathbf{u}\mathbb{I}_{3}\right)d\mathbf{x}}_{I_{4}} \nonumber\\
        &+\underbrace{\int_{\Omega}\frac{32}{75R^2}\sqrt{\frac{2}{R}}a_{\theta}^{I}a_{\theta}^{II}\frac{\kappa(\tilde{\theta})^{2}}{\tilde{\theta}\sqrt{\tilde{\theta}}}\nabla\partial_{i}^2\zeta\cdot\nabla\partial_{i}^2\zeta d\mathbf{x}}_{I_{5}} \\
        &+\underbrace{\int_{\Omega}-\frac{4}{5R^2}a_{\mathbf{u}}^{II}a_{\theta}^I\frac{\mu(\tilde{\theta})\kappa(\tilde{\theta})}{\mu(\overline{\theta})\tilde{\theta}}\nabla\partial_{i}^2\zeta\cdot\mathfrak{n}\left(\partial_{i}^2\nabla u+\partial_{i}^2\nabla\mathbf{u}-\frac{2}{3}(\partial_{i}^2\div\mathbf{u})\mathbf{e}_{1}\right)d\mathbf{x}}_{I_{6}} \\
        &+\underbrace{\int_{\Omega}\frac{2}{3R}\sqrt{\frac{2}{R}}a_{\mathbf{u}}^{I}a_{\mathbf{u}}^{II}\nabla\left(\frac{\mu(\tilde{\theta})^{2}}{\sqrt{\tilde{\theta}}}\right)\otimes\partial_{i}^2\mathbf{u}:\left(\nabla\partial_{i}^2\mathbf{u}+\nabla\partial_{i}^2\mathbf{u}^{\mathsf{T}}-\frac{2}{3}\div\partial_{i}^2\mathbf{u}\mathbb{I}_{3}\right)d\mathbf{x}}_{I_{7}} \nonumber\\
        &+\underbrace{\int_{\Omega}\frac{32}{75R^2}\sqrt{\frac{2}{R}}a_{\theta}^{I}a_{\theta}^{II}\nabla\left(\frac{\kappa(\tilde{\theta})\kappa(\tilde{\theta})}{\tilde{\theta}\sqrt{\tilde{\theta}}}\right)\partial_{i}^2\zeta\cdot\nabla\partial_{i}^2\zeta d\mathbf{x}}_{I_{8}} \\
        &+\underbrace{\int_{\Omega}-\frac{4}{5R^2}a_{\mathbf{u}}^{II}a_{\theta}^I\nabla\left(\frac{\mu(\tilde{\theta})\kappa(\tilde{\theta})}{\mu(\overline{\theta})\tilde{\theta}}\right)\partial_{i}^2\zeta\cdot\mathfrak{n}\left(\partial_{i}^2\nabla u_{1}+\partial_{i}^2\nabla\mathbf{u}-\frac{2}{3}(\partial_{i}^2\div\mathbf{u})\mathbf{e}_{1}\right)d\mathbf{x}}_{I_{9}} \\
        &+\underbrace{\int_{\Omega}R\overline{\rho}\sqrt{\frac{2}{R}}\frac{2}{3R}a_{\mathbf{u}}^{I}a_{\mathbf{u}}^{II}\frac{\mu(\tilde{\theta})^{2}}{\mu(\overline{\theta})\sqrt{\tilde{\theta}}}\partial_{i}^2\mathbf{u}\cdot\nabla\partial_{i}^2\zeta d\mathbf{x}}_{I_{10}} +\underbrace{\int_{\Omega}\frac{32}{75R^2}\sqrt{\frac{2}{R}}a_{\theta}^{I}a_{\theta}^{II}\frac{\kappa(\tilde{\theta})^{2}}{\kappa(\overline{\theta})\tilde{\theta}\sqrt{\tilde{\theta}}}(R\overline{\rho}\overline{\theta}+\mathcal{R}\overline{\rho})\partial_{i}^2\zeta\cdot\partial_{i}^2\div\mathbf{u}d\mathbf{x}}_{I_{11}} \\
        &+\underbrace{\int_{\Omega}-\frac{4}{5R}\overline{\rho}a_{\mathbf{u}}^{II}a_{\theta}^I\frac{\mu(\tilde{\theta})\kappa(\tilde{\theta})}{\mu(\overline{\theta})\tilde{\theta}}\partial_{i}\zeta\cdot\mathfrak{n}\partial_{i}\partial_{1}\zeta d\mathbf{x}}_{I_{12}} \\
        &+\underbrace{\sum_{p=0,1}\int_{\mathbb{T}_{p}^{2}}\frac{2}{3R}a_{\mathbf{u}}^{II}\mu(\theta_{w})(\varphi'+\overline{\rho})|\partial_{i}^2\mathbf{u}|^{2}dx_{2}dx_{3}}_{I_{13}} +\underbrace{\sum_{p=0,1}\int_{\mathbb{T}_{p}^{2}}\frac{16}{15R}a_{\theta}^{I}\frac{\kappa(\theta_{w})}{\theta_{w}}(\varphi'+\overline{\rho})|\partial_{i}^2\zeta|^{2}dx_{2}dx_{3}}_{I_{14}} \\
        &+\underbrace{\sum_{p=0,1}\int_{\mathbb{T}_{p}^{2}}\frac{16}{15R}a_{\theta}^{I}\frac{\kappa(\theta_{w})}{\theta_{w}}(\varphi'+\overline{\rho})\partial_{i}^2\zeta\cdot(\partial_{i}^2\tilde{\theta}-\partial_{i}^2\theta_{w})dx_{2}dx_{3}}_{I_{15}} \\
        &+\underbrace{\sum_{p=0,1}\int_{\mathbb{T}_{p}^{2}}\frac{2}{3R}a_{\mathbf{u}}^{II}\mu(\theta_{w})\partial_{i}\varphi'\cdot\partial_{i}^2\mathbf{u}\cdot\partial_{i}\mathbf{u}dx_{2}dx_{3}}_{I_{16}} \\ &+\underbrace{\sum_{p=0,1}\int_{\mathbb{T}_{p}^{2}}\frac{16}{15R}a_{\theta}^{I}\frac{\kappa(\theta_{w})}{\theta_{w}}\partial_{i}^2\zeta\cdot\partial_{i}\varphi'\cdot(\partial_{i}\zeta+\partial_{i}\tilde{\theta}-\partial_{i}\theta_{w})dx_{2}dx_{3}}_{I_{17}} \\
        &+\underbrace{\sum_{p=0,1}\int_{\mathbb{T}_{p}^{2}}\frac{2}{3R}a_{\mathbf{u}}^{II}\mu(\theta_{w})\partial_{i}^2\varphi'\cdot\partial_{i}^2\mathbf{u}\cdot\mathbf{u}dx_{2}dx_{3}}_{I_{18}} +\underbrace{\sum_{p=0,1}\int_{\mathbb{T}_{p}^{2}}\frac{16}{15R}a_{\theta}^{I}\frac{\kappa(\theta_{w})}{\theta_{w}}\partial_{i}^2\zeta\cdot\partial_{i}^2\varphi'\zeta dx_{2}dx_{3}}_{I_{19}} \\
        &+\underbrace{\sum_{p=0,1}\int_{\mathbb{T}_{p}^{2}}\frac{8}{15R^2}a_{\mathbf{u}}^{II}a_{\theta}^{I}\partial_{i}^2\mathbf{u}\cdot\nabla\partial_{i}^2\tilde{\theta}dx_{2}dx_{3}}_{I_{20}} +\underbrace{\sum_{p=0,1}\int_{\mathbb{T}_{p}^{2}}\frac{32}{75R^2}a_{\theta}^{I}a_{\theta}^{II}\frac{\kappa(\theta_{w})^2}{\theta_{w}\sqrt{\theta_{w}}}\partial_{i}^2\zeta\cdot\mathfrak{n}\partial_{i}^2\partial_{1}\tilde{\theta}dx_{2}dx_{3}}_{I_{21}} \\
        = 
        &\underbrace{\int_{\Omega}\frac{2}{3R}\sqrt{\frac{2}{R}}a_{\mathbf{u}}^{I}a_{\mathbf{u}}^{II}\frac{\mu(\tilde{\theta})^{2}}{\mu(\overline{\theta})\sqrt{\tilde{\theta}}}\partial_{i}^2\mathbf{u}\cdot\partial_{i}^2\mathbf{F}(\varphi',\mathbf{u}',\zeta')d\mathbf{x}}_{I_{22}}+\underbrace{\int_{\Omega}-R\tilde{\theta}\sqrt{\frac{2}{R}}\frac{2}{3R}a_{\mathbf{u}}^{I}a_{\mathbf{u}}^{II}\frac{\mu(\tilde{\theta})^{2}}{\mu(\overline{\theta})\sqrt{\tilde{\theta}}}\partial_{i}^2\mathbf{u}\cdot\nabla\partial_{i}^2\varphi'd\mathbf{x}}_{I_{23}} \\
        &+\underbrace{\int_{\Omega}\frac{32}{75R^2}\sqrt{\frac{2}{R}}a_{\theta}^{I}a_{\theta}^{II}\frac{\kappa(\tilde{\theta})^{2}}{\kappa(\overline{\theta})\tilde{\theta}\sqrt{\tilde{\theta}}}\partial_{i}^2\zeta\cdot\partial_{i}^2 G(\varphi',\mathbf{u}',\zeta')d\mathbf{x}}_{I_{24}} \\
        &+\underbrace{\int_{\Omega}-\frac{32}{75R^2}\sqrt{\frac{2}{R}}a_{\theta}^{I}a_{\theta}^{II}\frac{\kappa(\tilde{\theta})^{2}}{\kappa(\overline{\theta})\tilde{\theta}\sqrt{\tilde{\theta}}}\mathcal{R}\partial_{i}^2\zeta\cdot\div\partial_{i}^2(\varphi'\mathbf{u}')d\mathbf{x}}_{I_{25}} \\
        &+\underbrace{\int_{\Omega}-\frac{4}{5R^2}a_{\mathbf{u}}^{II}a_{\theta}^I\frac{\mu(\tilde{\theta})\kappa(\tilde{\theta})}{\mu(\overline{\theta})\tilde{\theta}}\partial_{i}^2\zeta\cdot\mathfrak{n}\partial_{i}^2F_{1}(\varphi',\mathbf{u}',\zeta')d\mathbf{x}}_{I_{26}} +\underbrace{\int_{\Omega}\frac{4}{5R}\overline{\theta}a_{\mathbf{u}}^{II}a_{\theta}^I\frac{\mu(\tilde{\theta})\kappa(\tilde{\theta})}{\mu(\overline{\theta})\tilde{\theta}}\partial_{i}^2\zeta\cdot\mathfrak{n}\partial_{i}^2\partial_{1}\varphi' d\mathbf{x}}_{I_{27}}.
    \end{align*}
  
    For $I_{1}$, 
    \begin{equation*}
        I_{1}\geq \frac{R}{2h}\frac{\overline{\theta}}{\overline{\rho}}\sqrt{\frac{2}{R}}\frac{2}{3R}a_{\mathbf{u}}^{I}a_{\mathbf{u}}^{II}\frac{\mu(\overline{\theta})}{\sqrt{\overline{\theta}}}\left(\|\partial_{i}^{2}\varphi\|_{L^{2}(\Omega)}^{2}-\|\partial_{i}^{2}\varphi'\|_{L^{2}(\Omega)}^{2}+\|\partial_{i}^{2}\varphi-\partial_{i}^{2}\varphi'\|_{L^{2}(\Omega)}^{2}\right).
    \end{equation*}

    For $I_{2}$, 
    \begin{equation*}
        I_{2}
        \leq C_{1}\int_{\Omega}|\div\partial_{i}^{2}\mathbf{u}\cdot|\partial_{i}^{2}\varphi|^2|d\mathbf{x}\\
        \leq C_{1}\|\partial_{i}^{2}\varphi\|_{L^3(\Omega)}\|\partial_{i}\varphi\|_{L^6(\Omega)}\|\nabla\partial_{i}^{2}\mathbf{u}\|_{L^2(\Omega)}\\
        \leq C_{1}\|\partial_{i}^{2}\varphi\|_{H^{1}(\Omega)}^{2}\|\nabla\partial_{i}^{2}\mathbf{u}\|_{L^{2}(\Omega)},
    \end{equation*}
    where $C_{1}>0$ is a constant only depending on $R, a_{\mathbf{u}}^{I}, a_{\mathbf{u}}^{II}, \overline{\rho}, \overline{\theta}$.

    For $I_{3}$ and $I_{17}$, 
    \begin{align*}
        &|I_{3}+I_{17}| \\
        \leq&\varepsilon\int_{\Omega}|\partial_{i}^{2}\mathbf{u}\cdot\nabla\partial_{i}^{2}\varphi'|d\mathbf{x}+R\overline{\theta}\sqrt{\frac{2}{R}}\frac{2}{3R}a_{\mathbf{u}}^{I}a_{\mathbf{u}}^{II}\frac{\mu(\overline{\theta})}{\sqrt{\overline{\theta}}}\int_{\Omega}|(\partial_{i}^{2}\varphi-\partial_{i}^{2}\varphi')\div\partial_{i}\mathbf{u}|d\mathbf{x} \\
        \leq&\varepsilon(\|\nabla\partial_{i}^{2}\varphi'\|_{L^{2}(\Omega)}^{2}+\|\partial_{i}^{2}\mathbf{u}\|_{L^{2}(\Omega)}^{2})+\frac{h}{R}\overline{\rho}\overline{\theta}\sqrt{\frac{2}{R}}\frac{2}{3R}a_{\mathbf{u}}^{I}a_{\mathbf{u}}^{II}\frac{\mu(\overline{\theta})}{\sqrt{\overline{\theta}}}\|\nabla\partial_{i}^{2}\mathbf{u}\|_{L^{2}(\Omega)}^{2} \\
        &+\frac{R}{4h}\frac{\overline{\theta}}{\overline{\rho}}\sqrt{\frac{2}{R}}\frac{2}{3R}a_{\mathbf{u}}^{I}a_{\mathbf{u}}^{II}\frac{\mu(\overline{\theta})}{\sqrt{\overline{\theta}}}\|\partial_{i}^{2}\varphi-\partial_{i}^{2}\varphi'\|_{L^{2}(\Omega)}^{2}.
    \end{align*}
  
    For $I_{18}$, by Plancherel theorem, the trace theorem and H\"{o}lder's inequality,
    \begin{align*}
        |I_{18}|=&\frac{2}{3R}a_{\mathbf{u}}^{II}\left|\sum_{m=-\infty}^{+\infty}4\pi^2m^2[\varphi']^\wedge(m)\cdot\left[\mu(\theta_w)\partial_{i}^2\mathbf{u}\cdot\mathbf{u}\right]^\wedge(m)\right|\\
        \leq&C_{2}\left(\sum_{m=-\infty}^{+\infty}\left|m^\frac{3}{2}[\varphi']^\wedge(m)\right|^2\right)^\frac{1}{2}\left(\sum_{m=-\infty}^{+\infty}\left|m^\frac{1}{2}\left[\mu(\theta_w)\partial_{i}^2\mathbf{u}\cdot\mathbf{u}\right]^\wedge(m)\right|^2\right)^\frac{1}{2}\\
        \leq&C_{2}\|\varphi'\|_{H^\frac{3}{2}(\partial\Omega)}\|\mu(\theta_w)\partial_{i}^2\mathbf{u}\cdot\mathbf{u}\|_{H^{\frac{1}{2}}(\partial\Omega)}\\
        \leq&C_{2}\|\varphi'\|_{H^{2}(\Omega)}\|\partial_{i}^2\mathbf{u}\cdot\mathbf{u}\|_{H^{1}(\Omega)}\\
        \leq&C_{2}\|\varphi'\|_{H^{2}(\Omega)}\|\mathbf{u}\|_{H^{3}(\Omega)}^{2},
    \end{align*}
    where $C_{2}>0$ is a constant only depending on $R, a_{\mathbf{u}}^{II}, \mu, \overline{\theta}$.

    For $I_{19}$, by Plancherel theorem, the trace theorem and H\"{o}lder's inequality,
    \begin{align*}
        |I_{19}|=&\frac{16}{15R}a_{\theta}^{I}\left|\sum_{m=-\infty}^{+\infty}4\pi^2m^2[\varphi']^\wedge(m)\cdot\left[\kappa(\theta_w)\theta_w^{-1}\partial_{i}^2\zeta\cdot\zeta\right]^\wedge(m)\right|\\
        \leq&C_{3}\left(\sum_{m=-\infty}^{+\infty}\left|m^\frac{3}{2}[\varphi']^\wedge(m)\right|^2\right)^\frac{1}{2}\left(\sum_{m=-\infty}^{+\infty}\left|m^\frac{1}{2}\left[\kappa(\theta_w)\theta_w^{-1}\partial_{i}^2\zeta\cdot\zeta\right]^\wedge(m)\right|^2\right)^\frac{1}{2}\\
        \leq&C_{3}\|\varphi'\|_{H^{\frac{3}{2}}(\partial\Omega)}\|\kappa(\theta_w)\theta_w^{-1}\partial_{i}^2\zeta\cdot\zeta\|_{H^{\frac{1}{2}}(\partial\Omega)}\\
        \leq&C_{3}\|\varphi'\|_{H^{2}(\Omega)}\|\partial_{i}^2\zeta\cdot\zeta\|_{H^{1}(\Omega)}\\
        \leq&C_{3}\|\varphi'\|_{H^{2}(\Omega)}\|\zeta\|_{H^{3}(\Omega)}^{2},
    \end{align*}
    where $C_{3}>0$ is a constant only depending on $R, a_{\theta}^{I}, \kappa, \overline{\theta}$.

    For $I_{20}$, by Plancherel theorem, the trace theorem and H\"{o}lder's inequality,
    \begin{align*}
        |I_{20}|=&\frac{16}{15R}a_{\theta}^{I}\left|\sum_{m=-\infty}^{+\infty}4\pi^2m^2[\varphi']^\wedge(m)\cdot\left[\kappa(\theta_w)\theta_w^{-1}(\tilde{\theta}-\theta_w)\partial_{i}^2\zeta\right]^\wedge(m)\right|\\
        \leq&C_{4}\left(\sum_{m=-\infty}^{+\infty}\left|m^\frac{3}{2}[\varphi']^\wedge(m)\right|^2\right)^\frac{1}{2}\left(\sum_{m=-\infty}^{+\infty}\left|m^\frac{1}{2}\left[\kappa(\theta_w)\theta_w^{-1}(\tilde{\theta}-\theta_w)\partial_{i}^2\zeta\right]^\wedge(m)\right|^2\right)^\frac{1}{2}\\
        \leq&C_{4}\|\varphi'\|_{H^{\frac{3}{2}}(\partial\Omega)}\|\kappa(\theta_w)\theta_w^{-1}(\tilde{\theta}-\theta_w)\partial_{i}^2\zeta\|_{H^{\frac{1}{2}}(\partial\Omega)}\\
        \leq&C_{4}\varepsilon\|\varphi'\|_{H^{2}(\Omega)}\|\partial_{i}^2\zeta\|_{H^{1}(\Omega)}\\
        \leq&C_{4}\varepsilon\|\varphi'\|_{H^{2}(\Omega)}\|\zeta\|_{H^{3}(\Omega)},
    \end{align*}
    where $C_{4}>0$ is a constant only depending on $R, a_{\theta}^{I}, \kappa, \overline{\theta}$.

    For $I_{22}$,
    \begin{align*}
        |I_{22}|\leq&\left|\int_{\Omega}\frac{2}{3R}\sqrt{\frac{2}{R}}a_{\mathbf{u}}^{I}a_{\mathbf{u}}^{II}\frac{\mu(\tilde{\theta})^{2}}{\mu(\overline{\theta})\sqrt{\tilde{\theta}}}\partial_{i}^3\mathbf{u}\cdot\partial_{i}\mathbf{F}(\varphi',\mathbf{u}',\zeta')d\mathbf{x}\right| \\
        &+\left|\int_{\Omega}\frac{2}{3R}\sqrt{\frac{2}{R}}a_{\mathbf{u}}^{I}a_{\mathbf{u}}^{II}\partial_{i}\left(\frac{\mu(\tilde{\theta})^{2}}{\mu(\overline{\theta})\sqrt{\tilde{\theta}}}\right)\cdot\partial_{i}^2\mathbf{u}\cdot\partial_{i}\mathbf{F}(\varphi',\mathbf{u}',\zeta')d\mathbf{x}\right| \\
        \leq&\eta\sum_{j=2,3}\|\partial_{i}^{j}\mathbf{u}\|_{L^{2}(\Omega)}^{2}+\frac{C_{5}}{\eta}\|\partial_{i}\mathbf{F}\|_{L^{2}(\Omega)}^{2},
    \end{align*}
    where $C_{5}>0$ is a constant only depending on $R, a_{\mathbf{u}}^{I}, a_{\mathbf{u}}^{II}, \mu, \overline{\theta}$.

    For $I_{24}$,
    \begin{align*}
        |I_{24}|\leq&\left|\int_{\Omega}\frac{32}{75R^2}\sqrt{\frac{2}{R}}a_{\theta}^{I}a_{\theta}^{II}\frac{\kappa(\tilde{\theta})^{2}}{\kappa(\overline{\theta})\tilde{\theta}\sqrt{\tilde{\theta}}}\partial_{i}^{3}\zeta\cdot\partial_{i}G(\varphi',\mathbf{u}',\zeta')d\mathbf{x}\right| \\
        &+\left|\int_{\Omega}\frac{32}{75R^2}\sqrt{\frac{2}{R}}a_{\theta}^{I}a_{\theta}^{II}\partial_{i}\left(\frac{\kappa(\tilde{\theta})^{2}}{\kappa(\overline{\theta})\tilde{\theta}\sqrt{\tilde{\theta}}}\right)\cdot\partial_{i}^2\zeta\cdot\partial_{i} G(\varphi',\mathbf{u}',\zeta')d\mathbf{x}\right| \\
        \leq&\eta\sum_{j=2,3}\|\partial_{i}^{j}\zeta\|_{L^{2}(\Omega)}^{2}+\frac{C_{6}}{\eta}\|\partial_{i}G\|_{L^{2}(\Omega)}^{2},
    \end{align*}
    where $C_{6}>0$ is a constant only depending on $R, a_{\theta}^{I}, a_{\theta}^{II}, \kappa, \overline{\theta}$.

    For $I_{25}$,
    \begin{align*}
        |I_{25}|\leq&\left|\int_{\Omega}-\frac{32}{75R^2}\sqrt{\frac{2}{R}}a_{\theta}^{I}a_{\theta}^{II}\frac{\kappa(\tilde{\theta})^{2}}{\kappa(\overline{\theta})\tilde{\theta}\sqrt{\tilde{\theta}}}\mathcal{R}\partial_{i}^{3}\zeta\cdot\div\partial_{i}(\varphi'\mathbf{u}')d\mathbf{x}\right| \\
        &+\left|\int_{\Omega}-\frac{32}{75R^2}\sqrt{\frac{2}{R}}a_{\theta}^{I}a_{\theta}^{II}\partial_{i}\left(\frac{\kappa(\tilde{\theta})^{2}}{\kappa(\overline{\theta})\tilde{\theta}\sqrt{\tilde{\theta}}}\right)\cdot\mathcal{R}\partial_{i}^2\zeta\cdot\div\partial_{i}(\varphi'\mathbf{u}')d\mathbf{x}\right| \\
        \leq&\eta\sum_{j=2,3}\|\partial_{i}^{j}\zeta\|_{L^{2}(\Omega)}^{2}+\frac{C_{7}}{\eta}\|\partial_{i}\div(\varphi'\mathbf{u}')\|_{L^{2}(\Omega)}^{2},
    \end{align*}
    where $C_{7}>0$ is a constant only depending on $R, a_{\theta}^{I}, a_{\theta}^{II}, \kappa, \overline{\theta}$.
 
    For $I_{26}$,
    \begin{align*}
        |I_{26}|\leq&\left|\int_{\Omega}-\frac{4}{5R^2}a_{\mathbf{u}}^{II}a_{\theta}^I\frac{\mu(\tilde{\theta})\kappa(\tilde{\theta})}{\mu(\overline{\theta})\tilde{\theta}}\partial_{i}^2\zeta\cdot\mathfrak{n}\partial_{i}^2F_{1}(\varphi',\mathbf{u}',\zeta')d\mathbf{x}\right| \\
        &+\left|\int_{\Omega}-\frac{4}{5R^2}a_{\mathbf{u}}^{II}a_{\theta}^I\frac{\mu(\tilde{\theta})\kappa(\tilde{\theta})}{\mu(\overline{\theta})\tilde{\theta}}\partial_{i}^2\zeta\cdot\mathfrak{n}\partial_{i}^2F_{1}(\varphi',\mathbf{u}',\zeta')d\mathbf{x}\right| \\
        \leq&\eta\sum_{j=2,3}\|\partial_{i}^{j}\zeta\|_{L^{2}(\Omega)}^{2}+\frac{C_{8}}{\eta}\|\partial_{i}F_{1}\|_{L^{2}(\Omega)}^{2},
    \end{align*}
    where $C_{7}>0$ is a constant only depending on $R, a_{\mathbf{u}}^{II}, a_{\theta}^{I}, \mu, \kappa, \overline{\theta}$.
 
    For $I_{27}$, 
    \begin{align*}
        |I_{27}|\leq&\left|\int_{\Omega}-R\tilde{\theta}\frac{4}{5R^2}a_{\mathbf{u}}^{II}a_{\theta}^{I}\mathfrak{n}\frac{\mu(\tilde{\theta})\kappa(\tilde{\theta})\tilde{\theta}^{-1}}{\mu(\overline{\theta})}\partial_{i}^3\zeta\cdot\partial_1\partial_{i}\varphi'd\mathbf{x}\right| \\
        &+\left|\int_{\Omega}-R\tilde{\theta}\frac{4}{5R^2}a_{\mathbf{u}}^{II}a_{\theta}^{I}\partial_{i}\left(\mathfrak{n}\frac{\mu(\tilde{\theta})\kappa(\tilde{\theta})\tilde{\theta}^{-1}}{\mu(\overline{\theta})}\right)\cdot\partial_{i}^2\zeta\cdot\partial_1\partial_{i}\varphi'd\mathbf{x}\right| \\
        \leq&\left(R\frac{4}{5R^2}a_{\mathbf{u}}^{II}a_\theta^I\kappa(\overline{\theta})+C\delta\right)\int_{\Omega}|\partial_{i}^3\zeta\cdot\partial_1\partial_{i}\varphi'|d\mathbf{x} +C_{8}\delta(\|\partial_{i}^2\zeta\|_{L^{2}(\Omega)}^{2}+\|\partial_{i}\partial_{1}\varphi'\|_{L^{2}(\Omega)}^{2}),
    \end{align*}
    where $C_{8}>0$ is a constant only depending on $R, a_{\mathbf{u}}^{II}, a_{\theta}^{I}, \mu, \kappa, \overline{\theta}$.

    The rest of the terms can be estimated similarly as the proof of Lemma \ref{regularityEllipticSystem}, and we omit them for brevity. After collecting all of the results, we can obtain the desired estimate \eqref{H2ConormalBasicEstimate}.
\end{proof}

\begin{lemma}\label{densityBasicEstimateLemma}
    Under the assumption of Corollary \ref{existenceEllipticSystem} and Corollary \ref{ExistenceTransportEquation}, there exist two constants $\eta, C>0$ such that
    \begin{align}\label{densityBasicEstimate}
        &\frac{4}{3}\frac{\mu(\overline{\theta})}{\overline{\rho}}\frac{1}{2h}(\|\partial_{1}\varphi\|_{L^{2}(\Omega)}^{2}-\|\partial_{1}\varphi'\|_{L^{2}(\Omega)}^{2}) \nonumber\\
        &+\left(\frac{4}{3}\frac{\mu(\overline{\theta})}{\overline{\rho}}\frac{1}{2h}-\frac{R\overline{\theta}}{2}\right)\|\partial_{1}\varphi-\partial_{1}\varphi'\|_{L^{2}(\Omega)}^{2} \nonumber\\
        &+\left(\frac{R\overline{\theta}}{2}-\eta\right)\|\partial_{1}\varphi\|_{L^{2}(\Omega)}^{2}+\frac{R\overline{\theta}}{2}\|\partial_{1}\varphi'\|_{L^{2}(\Omega)}^{2} \nonumber\\
        \leq&\mu(\overline{\theta})\|\partial_{1}\varphi\|_{L^{2}(\Omega)}\sum_{i=2,3}\|\nabla\partial_{i}\mathbf{u}\|_{L^{2}(\Omega)} +R\overline{\rho}\|\partial_{1}\varphi\|_{L^{2}(\Omega)}\|\partial_{1}\zeta\|_{L^{2}(\Omega)} \nonumber\\ 
        &+\frac{C}{\eta}(1+N_{1}(\varphi',\mathbf{u}',\zeta'))^{2}(\sqrt{\delta}+N_{2}(\varphi',\mathbf{u}',\zeta'))^{4}.
    \end{align}
\end{lemma}
\begin{proof}
    Applying $\partial_{1}$ to $\eqref{LSNS}_{1}$ and multiplying the resultant by $\frac{4}{3}\frac{\mu(\overline{\theta})}{\overline{\rho}}$,
    \begin{equation*}
        \frac{4}{3}\frac{\mu(\overline{\theta})}{\overline{\rho}}\frac{\partial_{1}\varphi-\partial_{1}\varphi'}{h}+\frac{4}{3}\frac{\mu(\overline{\theta})}{\overline{\rho}}\div\partial_{1}(\varphi\mathbf{u})+\frac{4}{3}\frac{\mu(\overline{\theta})}{\overline{\rho}}\partial_{1}(\overline{\rho}\div\mathbf{u})=0,
    \end{equation*}
    and taking the inner product between $\eqref{LSNS}_{2}$ and $\mathbf{e}_{1}$,
    \begin{equation*}
        -\div\left[\mu(\overline{\theta})\left(\nabla u_{1}+\partial_{1}\mathbf{u}-\frac{2}{3}(\div\mathbf{u})\mathbf{e}_{1}\right)\right]+R\overline{\rho}\partial_{1}\zeta=F_{1}(\varphi',\mathbf{u}',\zeta')-R\overline{\theta}\partial_{1}\varphi'.
    \end{equation*}
    It then follows that
    \begin{align*}
        &\frac{4}{3}\frac{\mu(\overline{\theta})}{\overline{\rho}}\frac{\partial_{1}\varphi-\partial_{1}\varphi'}{h}+R\overline{\theta}\partial_{1}\varphi'\\
        =&-\frac{4}{3}\frac{\mu(\overline{\theta})}{\overline{\rho}}\div\partial_{1}(\varphi\mathbf{u})-\frac{4}{3}\frac{\mu(\overline{\theta})}{\overline{\rho}}\partial_{1}(\overline{\rho}\div\mathbf{u})\\
        &+\mu(\overline{\theta})\left(\Delta u_{1}+\frac{1}{3}\partial_{1}\div\mathbf{u}\right)-R\overline{\rho}\partial_{1}\zeta+F_{1}(\varphi',\mathbf{u}',\zeta').
    \end{align*}
    Multiplying the above equation by $\partial_{1}\varphi$ and integrating the resultant over $\Omega$, we have
    \begin{align*}
        &\underbrace{\int_{\Omega}\partial_{1}\varphi\left(\frac{4}{3}\frac{\mu(\overline{\theta})}{\overline{\rho}}\frac{\partial_{1}\varphi-\partial_{1}\varphi'}{h}+R\overline{\theta}\partial_{1}\varphi'\right)d\mathbf{x}}_{I_{1}} \\
        =&\underbrace{\int_{\Omega}-\frac{4}{3}\frac{\mu(\overline{\theta})}{\overline{\rho}}\partial_{1}\varphi\div\partial_{1}(\varphi\mathbf{u})d\mathbf{x}}_{I_{2}}+\underbrace{\int_{\Omega}-\frac{4}{3}\frac{\mu(\overline{\theta})}{\overline{\rho}}\partial_{1}\varphi\partial_{1}(\overline{\rho}\div\mathbf{u})d\mathbf{x}}_{I_{3}} \\
        &+\underbrace{\int_{\Omega}\mu(\overline{\theta})\partial_{1}\varphi\cdot\left(\Delta u_{1}+\frac{1}{3}\partial_{1}\div\mathbf{u}\right)d\mathbf{x}}_{I_{4}} +\underbrace{\int_{\Omega}-R\overline{\rho}\partial_{1}\varphi\cdot\partial_{1}\zeta d\mathbf{x}}_{I_{5}}+\underbrace{\int_{\Omega}\partial_{1}\varphi\cdot F_1(\varphi',\mathbf{u}',\zeta')d\mathbf{x}}_{I_{6}}.
    \end{align*}
  
    For $I_{1}$,
    \begin{align}\label{densityBasicEstimateI1}
        I_{1}=&\int_{\Omega}\frac{4}{3}\frac{\mu(\overline{\theta})}{\overline{\rho}}\frac{1}{h}|\partial_{1}\varphi|^2+\left(-\frac{4}{3}\frac{\mu(\overline{\theta})}{\overline{\rho}}\frac{1}{h}+R\overline{\theta}\right)\partial_{1}\varphi\cdot\partial_{1}\varphi'd\mathbf{x} \nonumber\\
        \geq&\frac{4}{3}\frac{\mu(\overline{\theta})}{\overline{\rho}}\frac{1}{2h}\left(\int_{\Omega}|\partial_{1}\varphi|^2d\mathbf{x}-\int_{\Omega}|\partial_{1}\varphi'|^2\right)d\mathbf{x} \nonumber\\
        &+\left(\frac{4}{3}\frac{\mu(\overline{\theta})}{\overline{\rho}}\frac{1}{2h}-\frac{R\overline{\theta}}{2}\right)\int_{\Omega}\left|\partial_{1}\varphi-\partial_{1}\varphi'\right|^2d\mathbf{x} \nonumber\\
        &+\frac{R\overline{\theta}}{2}\left(\int_{\Omega}|\partial_{1}\varphi|^2d\mathbf{x}+\int_{\Omega}|\partial_{1}\varphi'|^2d\mathbf{x}\right).
    \end{align}
  
    For $I_{2}$,
    \begin{align*}
        |I_{2}|
        \leq&\left|\int_{\Omega}\frac{4}{3}\frac{\mu(\overline{\theta})}{\overline{\rho}}\cdot\frac{1}{2}(\partial_{1}\varphi)^2\div\mathbf{u}d\mathbf{x}\right| +\left|\int_{\Omega}\frac{4}{3}\frac{\mu(\overline{\theta})}{\overline{\rho}}\nabla\partial_{1}\varphi\cdot\varphi\partial_{1}\mathbf{u}d\mathbf{x}\right| \nonumber\\
        \leq&C(\|\div\mathbf{u}\|_{L^{\infty}(\Omega)}\|\partial_{1}\varphi\|_{L^{2}(\Omega)}^{2}+\|\varphi\|_{L^{4}(\Omega)}\|\nabla\partial_{1}\varphi\|_{L^{2}(\Omega)}\|\partial_{1}\mathbf{u}\|_{L^{4}(\Omega)}) \nonumber\\
        \leq&C\|\mathbf{u}\|_{H^{2}(\Omega)}\|\varphi\|_{H^{2}(\Omega)}^{2}.
    \end{align*}
    Then by Corollary \ref{existenceEllipticSystem} and Corollary \ref{ExistenceTransportEquation}, 
    \begin{equation}\label{densityBasicEsitmateI2}
        |I_{2}|\leq C(1+N_{1}(\varphi',\mathbf{u}',\zeta'))^{3}(\sqrt{\delta}+N_{2}(\varphi',\mathbf{u}',\zeta'))^{6}.
    \end{equation}
  
    For $I_{3}$ and $I_{4}$,
    \begin{align}\label{densityBasicEstiamteI3I4}
        |I_{3}+I_{4}|=&\left|\int_{\Omega}\partial_{1}\varphi\left[-\frac{4}{3}\frac{\mu(\overline{\theta})}{\overline{\rho}}\partial_{1}(\overline{\rho}\div\mathbf{u})+\mu(\overline{\theta})\left(\Delta u_{1}+\frac{1}{3}\partial_{1}\div\mathbf{u}\right)\right]d\mathbf{x}\right| \nonumber\\
        =&\left|\int_{\Omega}\mu(\overline{\theta})\partial_{1}\varphi\cdot(\Delta u_{1}-\partial_{1}\div\mathbf{u})d\mathbf{x}\right| \nonumber\\
        \leq&\mu(\overline{\theta})\|\partial_{1}\varphi\|_{L^{2}(\Omega)}\sum_{i=2,3}\|\nabla\partial_{i}\mathbf{u}\|_{L^{2}(\Omega)}.
    \end{align}
  
    For $I_{5}$,
    \begin{equation}\label{densityBasicEsitmateI5}
        |I_{5}|\leq\left(R\overline{\rho}+\varepsilon\right)\|\partial_{1}\varphi\|_{L^{2}(\Omega)}\|\partial_{1}\zeta\|_{L^{2}(\Omega)}.
    \end{equation}
  
    For $I_{6}$,
    \begin{equation}\label{densityBasicEstimateI6}
        |I_{6}|\leq \eta\|\partial_{1}\varphi\|_{L^{2}(\Omega)}^{2}+\frac{C}{\eta}\|\mathbf{F}'\|_{L^{2}(\Omega)}^{2}.
    \end{equation}

    Combining  
    \eqref{densityBasicEstimateI1}, \eqref{densityBasicEsitmateI2}, \eqref{densityBasicEstiamteI3I4}, \eqref{densityBasicEsitmateI5}, \eqref{densityBasicEstimateI6}, 
    we obtain the desired estimate \eqref{densityBasicEstimate}.
\end{proof}

\begin{lemma}\label{densityH1EstimateLemma}
    Under the assumption of Corollary \ref{existenceEllipticSystem} and Corollary \ref{ExistenceTransportEquation}, there exist constants $\eta, C>0$ such that
    \begin{align}\label{densityH1Estimate}
        &\frac{4}{3}\frac{\mu(\overline{\theta})}{\overline{\rho}}\frac{1}{2h}(\|\partial_{1}\partial_{i}\varphi\|_{L^{2}(\Omega)}^{2}-\|\partial_{1}\partial_{i}\varphi'\|_{L^{2}(\Omega)}^{2}) \nonumber\\
        &+\left(\frac{4}{3}\frac{\mu(\overline{\theta})}{\overline{\rho}}\frac{1}{2h}-\frac{R\overline{\theta}}{2}\right)\left\|\partial_{1}\partial_{i}\varphi-\partial_{1}\partial_{i}\varphi'\right\|_{L^{2}(\Omega)}^{2} \nonumber\\
        &+\left(\frac{R\overline{\theta}}{2}-\eta\right)\|\partial_{1}\partial_{i}\varphi\|_{L^{2}(\Omega)}^{2}+\frac{R\overline{\theta}}{2}\|\partial_{1}\partial_{i}\varphi'\|_{L^{2}(\Omega)}^{2} \nonumber\\
        \leq&\mu(\overline{\theta})\|\partial_{i}\partial_{1}\varphi\|_{L^{2}(\Omega)}\sum_{j=2,3}\|\partial_{j}\nabla^{2}\mathbf{u}\|_{L^{2}(\Omega)} +R\overline{\rho}\|\partial_{i}\partial_{1}\varphi\|_{L^{2}(\Omega)}\|\partial_{i}\partial_{1}\zeta\|_{L^{2}(\Omega)} \nonumber\\
        &+\frac{C}{\eta}(1+N_{2}(\varphi,\mathbf{u},\zeta))^{2}(\sqrt{\delta}+N_{2}(\varphi,\mathbf{u},\zeta))^{4},
    \end{align}
    for $i=1,2,3$.
\end{lemma}
\begin{proof}
    Multiplying $\partial_{1}\partial_{i}\eqref{LSNS}_{1}$ by $\frac{4}{3}\frac{\mu(\overline{\theta})}{\overline{\rho}}$,
    \begin{equation*}
        \frac{4}{3}\frac{\mu(\overline{\theta})}{\overline{\rho}}\frac{\partial_{1}\partial_{i}\varphi-\partial_{1}\partial_{i}\varphi'}{h}+\frac{4}{3}\frac{\mu(\overline{\theta})}{\overline{\rho}}\div\partial_{i}\partial_{1}(\varphi\mathbf{u})+\frac{4}{3}\frac{\mu(\overline{\theta})}{\overline{\rho}}\partial_{1}\partial_{i}(\overline{\rho}\div\mathbf{u})=0,
    \end{equation*}
    and taking the inner product between $\partial_{i}\eqref{LSNS}_{2}$ and $\mathbf{e}_{1}$,
    \begin{equation*}
        -\div\left[\mu(\overline{\theta})\left(\partial_{i}\nabla u_{1}+\partial_{i}\partial_{1}\mathbf{u}-\frac{2}{3}(\partial_{i}\div\mathbf{u})\mathbf{e}_{1}\right)\right]+R\overline{\rho}\partial_{1}\partial_{i}\zeta=\partial_{i} F_{1}(\varphi',\mathbf{u}',\zeta')-R\overline{\theta}\partial_{i}\partial_{1}\varphi'.
    \end{equation*}
    It then follows that
    \begin{align*}
        &\frac{4}{3}\frac{\mu(\overline{\theta})}{\overline{\rho}}\frac{\partial_{1}\partial_{i}\varphi-\partial_{1}\partial_{i}\varphi'}{h}+R\overline{\theta}\partial_{1}\partial_{i}\varphi' \\
        =&-\frac{4}{3}\frac{\mu(\overline{\theta})}{\overline{\rho}}\div\partial_{i}\partial_{1}(\varphi\mathbf{u})-\frac{4}{3}\frac{\mu(\overline{\theta})}{\overline{\rho}}\partial_{1}\partial_{i}(\overline{\rho}\div\mathbf{u}) \\
        &+\mu(\overline{\theta})\mathbf{n}\cdot\left(\partial_{i}\Delta u_{1}+\frac{1}{3}\partial_{i}\partial_{1}\div\mathbf{u}\right) -R\overline{\rho}\partial_{1}\partial_{i}\zeta+\partial_{i} F_{1}(\varphi',\mathbf{u}',\zeta').
    \end{align*}
    Multiplying the above equation by $\partial_{1}\partial\varphi$ and integrating the resultant over $\Omega$, 
    \begin{align*}
        &\underbrace{\int_{\Omega}\partial_{1}\partial_{i}\varphi\cdot\left(\frac{4}{3}\frac{\mu(\overline{\theta})}{\overline{\rho}}\frac{\partial_{1}\partial_{i}\varphi-\partial_{1}\partial_{i}\varphi'}{h}+R\overline{\theta}\partial_{1}\partial_{i}\varphi'\right)d\mathbf{x}}_{I_{1}} \\
        =&\underbrace{\int_{\Omega}-\frac{4}{3}\frac{\mu(\overline{\theta})}{\overline{\rho}}\partial_{1}\partial_{i}\varphi\cdot\div\partial_{i}\partial_{1}(\varphi\mathbf{u})d\mathbf{x}}_{I_{2}}+\underbrace{\int_{\Omega}-\frac{4}{3}\frac{\mu(\overline{\theta})}{\overline{\rho}}\partial_{1}\partial_{i}\varphi\cdot\partial_{1}\partial_{i}(\overline{\rho}\div\mathbf{u})d\mathbf{x}}_{I_{3}} \\
        &+\underbrace{\int_{\Omega}\mu(\overline{\theta})\partial_{1}\partial_{i}\varphi\cdot\left(\partial_{i}\Delta u_{1}+\frac{1}{3}\partial_{i}\partial_{1}\div\mathbf{u}\right)d\mathbf{x}}_{I_{4}} +\underbrace{\int_{\Omega}-R\overline{\rho}\partial_{1}\partial_{i}\varphi\cdot\partial_{1}\partial_{i}\zeta d\mathbf{x}}_{I_{5}} \\
        &+\underbrace{\int_{\Omega}\partial_{1}\partial_{i}\varphi\cdot\partial_{i} F_{1}(\varphi',\mathbf{u}',\zeta')d\mathbf{x}}_{I_{6}}.
    \end{align*}
  
    For $I_{1}$,
    \begin{align}\label{densityH1EstimateI1}
        I_{1}=&\int_{\Omega}\frac{4}{3}\frac{\mu(\overline{\theta})}{\overline{\rho}}\frac{1}{h}|\partial_{1}\partial_{i}\varphi|^2+\left(-\frac{4}{3}\frac{\mu(\overline{\theta})}{\overline{\rho}}\frac{1}{h}+R\overline{\theta}\right)\partial_{1}\partial_{i}\varphi\cdot\partial_{1}\varphi'd\mathbf{x} \nonumber\\
        \geq&\left(\frac{4}{3}\frac{\mu(\overline{\theta})}{\overline{\rho}}-\varepsilon\right)\frac{1}{2h}\left(\int_{\Omega}|\partial_{1}\partial_{i}\varphi|^2d\mathbf{x}-\int_{\Omega}|\partial_{1}\partial_{i}\varphi'|^2d\mathbf{x}\right) \nonumber\\
        &+\left(\frac{4}{3}\frac{\mu(\overline{\theta})}{\overline{\rho}}\frac{1}{2h}-\frac{R\overline{\theta}}{2}-\varepsilon\right)\int_{\Omega}\left|\partial_{1}\partial_{i}\varphi-\partial_{1}\partial_{i}\varphi'\right|^2d\mathbf{x} \nonumber\\
        &+\left(\frac{R\overline{\theta}}{2}-\varepsilon\right)\left(\int_{\Omega}|\partial_{1}\partial_{i}\varphi|^2d\mathbf{x}+\int_{\Omega}|\partial_{1}\partial_{i}\varphi'|^2d\mathbf{x}\right).
    \end{align}
  
    For $I_{2}$,
    \begin{align*}
        &|I_{2}| \nonumber\\
        \leq&\left|\int_{\Omega}\frac{4}{3}\frac{\mu(\overline{\theta})}{\overline{\rho}}\cdot\frac{1}{2}(\partial_{1}\partial_{i}\varphi)^2\div\mathbf{u}d\mathbf{x}\right| +\left|\int_{\Omega}\frac{4}{3}\frac{\mu(\overline{\theta})}{\overline{\rho}}\partial_{1}\partial_{i}\varphi\cdot\partial_{1}\partial_{i}\varphi\cdot\div\mathbf{u}d\mathbf{x}\right| +\left|\int_{\Omega}\frac{4}{3}\frac{\mu(\overline{\theta})}{\overline{\rho}}\partial_{1}\partial_{i}\varphi\cdot\nabla\partial_{i}\varphi\cdot\partial_{1}\mathbf{u}d\mathbf{x}\right| \nonumber\\
        &+\left|\int_{\Omega}\frac{4}{3}\frac{\mu(\overline{\theta})}{\overline{\rho}}\partial_{1}\partial_{i}\varphi\cdot\nabla\partial_{1}\varphi\cdot\partial_{i}\mathbf{u}d\mathbf{x}\right| +\left|\int_{\Omega}\frac{4}{3}\frac{\mu(\overline{\theta})}{\overline{\rho}}\partial_{1}\partial_{i}\varphi\cdot\partial_{i}\varphi\cdot\div\partial_{1}\mathbf{u}d\mathbf{x}\right| \nonumber\\
        &+\left|\int_{\Omega}\frac{4}{3}\frac{\mu(\overline{\theta})}{\overline{\rho}}\partial_{1}\partial_{i}\varphi\cdot\partial_{1}\varphi\cdot\div\partial_{i}\mathbf{u}d\mathbf{x}\right| +\left|\int_{\Omega}\frac{4}{3}\frac{\mu(\overline{\theta})}{\overline{\rho}}\partial_{1}\partial_{i}\varphi\cdot\nabla\varphi\cdot\partial_{1}\partial_{i}\mathbf{u}d\mathbf{x}\right| \nonumber\\
        &+\left|\int_{\Omega}\frac{4}{3}\frac{\mu(\overline{\theta})}{\overline{\rho}}\partial_{1}\partial_{i}\varphi\cdot\varphi\div\partial_{1}\partial_{i}\mathbf{u}d\mathbf{x}\right| \nonumber\\
        \leq&C(\|\mathbf{u}\|_{L^{\infty}(\Omega)}\|\partial_{1}\partial_{i}\varphi\|_{L^{2}(\Omega)}^{2} +\|\div\mathbf{u}\|_{L^{\infty}(\Omega)}\|\partial_{1}\partial_{i}\varphi\|_{L^{2}(\Omega)}^{2} +\|\partial_{1}\mathbf{u}\|_{L^{\infty}(\Omega)}\|\partial_{1}\partial_{i}\varphi\|_{L^{2}(\Omega)}^{2} \nonumber\\
        &+\|\partial_{i}\mathbf{u}\|_{L^{\infty}(\Omega)}\|\partial_{1}\partial_{i}\varphi\|_{L^{2}(\Omega)}\|\nabla\partial_{1}\varphi\|_{L^{2}(\Omega)} +\|\partial_{1}\partial_{i}\varphi\|_{L^2(\Omega)}\|\partial_{i}\varphi\|_{L^4(\Omega)}\|\div\partial_{1}\mathbf{u}\|_{L^4(\Omega)} \nonumber\\
        &+\|\partial_{1}\partial_{i}\varphi\|_{L^2(\Omega)}\|\partial_{1}\varphi\|_{L^4(\Omega)}\|\div\partial_{i}\mathbf{u}\|_{L^4(\Omega)} +\|\partial_{1}\partial_{i}\varphi\|_{L^2(\Omega)}\|\nabla\varphi\|_{L^4(\Omega)}\|\partial_{1}\partial_{i}\mathbf{u}\|_{L^4(\Omega)} \nonumber\\
        &+\|\varphi\|_{L^{\infty}(\Omega)}\|\partial_{1}\partial_{i}\varphi\|_{L^{2}(\Omega)}\|\div\partial_{1}\partial_{i}\mathbf{u}\|_{L^{2}(\Omega)}) \nonumber\\
        \leq&C\|\varphi\|_{H^{2}(\Omega)}^{2}\|\mathbf{u}\|_{H^{3}(\Omega)}.
    \end{align*}
    Then by Corollary \ref{existenceEllipticSystem} and Corollary \ref{ExistenceTransportEquation},
    \begin{equation}\label{densityH1EstiamteI2}
        |I_{2}|\leq C(1+N_{2}(\varphi,\mathbf{u},\zeta))^{3}(\sqrt{\delta}+N_{2}(\varphi,\mathbf{u},\zeta))^{6}.
    \end{equation}
  
    For $I_{3}$ and $I_{4}$,
    \begin{align}\label{densityH1EstiameI3I4}
        |I_{3}+I_{4}|=&\left|\int_{\Omega}\partial_{1}\partial_{i}\varphi\cdot\left[-\frac{4}{3}\frac{\mu(\overline{\theta})}{\overline{\rho}}\partial_{1}\partial_{i}(\overline{\rho}\div\mathbf{u})+\mu(\overline{\theta})\left(\partial_{i}\Delta u_{1}+\frac{1}{3}\partial_{i}\partial_{1}\div\mathbf{u}\right)\right]\right| \nonumber\\
        =&\left|\int_{\Omega}\mu(\overline{\theta})\partial_{1}\partial_{i}\varphi\cdot(\partial_{i}\Delta u_{1}-\partial_{i}\partial_{1}\div\mathbf{u})d\mathbf{x}\right| \nonumber\\
        \leq&\mu(\overline{\theta})\|\partial_{1}\partial_{i}\varphi\|_{L^{2}(\Omega)}\sum_{j=2,3}\|\partial_{j}\nabla^{2}\mathbf{u}\|_{L^{2}(\Omega)}.
    \end{align}
  
    For $I_{5}$,
    \begin{equation}\label{densityH1EstimateI5}
        |I_{5}|\leq R\overline{\rho}\|\partial_{1}\partial_{i}\varphi\|_{L^{2}(\Omega)}\|\partial_{1}\partial_{i}\zeta\|_{L^{2}(\Omega)}.
    \end{equation}
  
    For $I_{6}$,
    \begin{equation}\label{densityH1EstimateI6}
        |I_{6}|\leq \eta\|\partial_{1}\partial_{i}\varphi\|_{L^{2}(\Omega)}^{2}+\frac{C}{\eta}\|\partial_{i}\mathbf{F}\|_{L^{2}(\Omega)}^{2}.
    \end{equation}

    Combining \eqref{densityH1EstimateI1}, \eqref{densityH1EstiamteI2}, \eqref{densityH1EstiameI3I4}, \eqref{densityH1EstimateI5}, \eqref{densityH1EstimateI6}, we obtain the desired estimate  \eqref{densityH1Estimate}.
\end{proof}

\begin{lemma}\label{divergenceBasicEstimateLemma}
    Under the assumption of Corollary \ref{existenceEllipticSystem} and Corollary \ref{ExistenceTransportEquation}, there exist two constants $\eta, C>0$ such that
    \begin{align}\label{divergenceBasicEstimate}
        &\left(\frac{4}{3}\mu(\overline{\theta})-\eta\right)\|\partial_{1}\div\mathbf{u}\|_{L^{2}(\Omega)}^{2} \nonumber\\
        \leq&\mu(\overline{\theta})\|\partial_{1}\div\mathbf{u}\|_{L^{2}(\Omega)}\|\partial_{i}\nabla\mathbf{u}\|_{L^{2}(\Omega)} +R\overline{\rho}\|\partial_{1}\div\mathbf{u}\|_{L^{2}(\Omega)}\|\partial_{1}\zeta\|_{L^{2}(\Omega)} \nonumber\\
        &+R\overline{\theta}\|\partial_{1}\div\mathbf{u}\|_{L^{2}(\Omega)}\|\partial_{1}\varphi'\|_{L^{2}(\Omega)} +\frac{C}{\eta}(1+N_{2}(\varphi,\mathbf{u},\zeta))^{2}(\sqrt{\delta}+N_{2}(\varphi,\mathbf{u},\zeta))^{4}.
    \end{align}
\end{lemma}
\begin{proof}
    Taking the inner product between $\eqref{LSNS}_{2}$ and $\mathbf{e}_{1}$,
    \begin{equation*}
        -\frac{4}{3}\mu(\overline{\theta})\partial_{1}\div\mathbf{u}=\mu(\overline{\theta})\mathbf{n}\cdot(\Delta u_{1}-\partial_{1}\div\mathbf{u}) -R\overline{\rho}\partial_{1}\zeta+F_{1}(\varphi',\mathbf{u}',\zeta')-R\overline{\theta}\partial_{1}\varphi'.
    \end{equation*}
    Multiplying the above equation by $-\partial_{1}\div\mathbf{u}$ and integrating the resultant over $\Omega$,
    \begin{align*}
        &\int_{\Omega}\frac{4}{3}\mu(\overline{\theta})\left|\partial_{1}\div\mathbf{u}\right|^2d\mathbf{x} \\
        =&\underbrace{\int_{\Omega}-\mu(\overline{\theta})\partial_{1}\div\mathbf{u}\cdot(\Delta u_{1}-\partial_{1}\div\mathbf{u})d\mathbf{x}}_{I_{1}}
        +\underbrace{\int_{\Omega}R\overline{\rho}\partial_{1}\div\mathbf{u}\cdot\partial_{1}\zeta d\mathbf{x}}_{I_{2}} \\
        &+\underbrace{\int_{\Omega}-\partial_{1}\div\mathbf{u}\cdot F_{1}(\varphi',\mathbf{u}',\zeta')d\mathbf{x}}_{I_{3}}+\underbrace{\int_{\Omega}R\overline{\theta}\partial_{1}\div\mathbf{u}\cdot\partial_{1}\varphi'd\mathbf{x}}_{I_{4}}.
    \end{align*}
  
    For $I_{1}$,
    \begin{equation}\label{divergenceBasicEstimateI1}
        |I_{1}|\leq \mu(\overline{\theta})\|\partial_{1}\div\mathbf{u}\|_{L^{2}(\Omega)}\|\partial_{i}\nabla\mathbf{u}\|_{L^{2}(\Omega)}.
    \end{equation} 
  
    For $I_{2}$,
    \begin{equation}\label{divergenceBasicEstimateI2}
        |I_{2}|\leq R\overline{\rho}\|\partial_{1}\div\mathbf{u}\|_{L^{2}(\Omega)}\|\partial_{1}\zeta\|_{L^{2}(\Omega)}.
    \end{equation}
  
    For $I_{3}$,
    \begin{equation}\label{divergenceBasicEstimateI3}
        |I_{3}|\leq \eta\|\partial_{1}\div\mathbf{u}\|_{L^{2}(\Omega)}^{2}+\frac{C}{\eta}\|\mathbf{F}\|_{L^{2}(\Omega)}^{2}.
    \end{equation}
  
    For $I_{4}$,
    \begin{equation}\label{divergenceBasicEstimateI4}
        |I_{4}|\leq R\overline{\theta}\|\partial_{1}\div\mathbf{u}\|_{L^{2}(\Omega)}\|\partial_{1}\varphi'\|_{L^{2}(\Omega)}.
    \end{equation}

    Combining \eqref{divergenceBasicEstimateI1}, \eqref{divergenceBasicEstimateI2}, \eqref{divergenceBasicEstimateI3}, \eqref{divergenceBasicEstimateI4}, we obtain the desired estimate \eqref{divergenceBasicEstimate}.
\end{proof}

\begin{lemma}\label{divergenceH1EstimateLemma}
    Under the assumption of Corollary \ref{existenceEllipticSystem} and Corollary \ref{ExistenceTransportEquation}, there exist constants $\eta, C>0$ such that
    \begin{align}\label{divergenceH1Estimate}
        &\left(\frac{4}{3}\mu(\overline{\theta})-\eta\right)\|\partial_{1}\partial_{i}\div\mathbf{u}\|_{L^{2}(\Omega)}^{2} \nonumber\\
        \leq&\mu(\overline{\theta})\|\partial_{1}\partial_{i}\div\mathbf{u}\|_{L^{2}(\Omega)}\|\partial_{i}^{2}\nabla\mathbf{u}\|_{L^{2}(\Omega)} +R\overline{\rho}\|\partial_{1}\partial_{i}\div\mathbf{u}\|_{L^{2}(\Omega)}\|\partial_{1}\partial_{i}\zeta\|_{L^{2}(\Omega)} \nonumber\\
        &+R\overline{\theta}\|\partial_{1}\partial_{i}\div\mathbf{u}\|_{L^{2}(\Omega)}\|\partial_{1}\partial_{i}\varphi'\|_{L^{2}(\Omega)} +\frac{C}{\eta}(1+N_{2}(\varphi,\mathbf{u},\zeta))^{2}(\sqrt{\delta}+N_{2}(\varphi,\mathbf{u},\zeta))^{4},
    \end{align}
    for $i=1,2,3$.
\end{lemma}

\begin{proof}
    Taking the inner product between $\partial\eqref{LSNS}_{2}$ and $\mathbf{e}_{1}$,
    \begin{equation*}
        -\frac{4}{3}\mu(\overline{\theta})\partial_{i}\partial_{1}\div\mathbf{u} =\mu(\overline{\theta})(\partial_{i}\Delta u_{1}-\partial_{i}\partial_{1}\div\mathbf{u}) -R\overline{\rho}\partial_{1}\partial_{i}\zeta+\partial_{i} F_{1}(\varphi',\mathbf{u}',\zeta')-R\overline{\theta}\partial_{i}\partial_{1}\varphi'.
    \end{equation*}
    Multiplying the above equation by $-\partial_{1}\partial_{i}\div\mathbf{u}$ and integrating the resultant over $\Omega$,
    \begin{align*}
        &\int_{\Omega}\frac{4}{3}\mu(\overline{\theta})\left|\partial_{i}\partial_{1}\div\mathbf{u}\right|^2d\mathbf{x} \\
        =&\underbrace{\int_{\Omega}-\mu(\overline{\theta})\partial_{1}\partial_{i}\div\mathbf{u}\cdot(\partial_{i}\Delta u_{1}-\partial_{i}\partial_{1}\div\mathbf{u})d\mathbf{x}}_{I_{1}} \\
        &+\underbrace{\int_{\Omega}R\overline{\rho}\partial_{1}\partial_{i}\div\mathbf{u}\cdot\partial_{1}\partial_{i}\zeta d\mathbf{x}}_{I_{2}}+\underbrace{\int_{\Omega}-\partial_{1}\partial_{i}\div\mathbf{u}\cdot\partial F_{1}(\varphi',\mathbf{u}',\zeta')d\mathbf{x}}_{I_{3}} \\
        &\underbrace{\int_{\Omega}R\overline{\theta}\partial_{1}\partial_{i}\div\mathbf{u}\cdot\partial_{1}\partial_{i}\varphi'd\mathbf{x}}_{I_{4}}.
    \end{align*}
  
    For $I_{1}$,
    \begin{equation}\label{divergenceH1EstimateI1}
        |I_{1}|\leq\mu(\overline{\theta})\|\partial_{1}\partial_{i}\div\mathbf{u}\|_{L^{2}(\Omega)}\|\partial_{i}^{2}\nabla\mathbf{u}\|_{L^{2}(\Omega)}.
    \end{equation}
  
    For $I_{2}$,
    \begin{equation}\label{divergenceH1EstimateI2}
        |I_{2}|\leq R\overline{\rho}\|\partial_{1}\partial_{i}\div\mathbf{u}\|_{L^{2}(\Omega)}\|\partial_{1}\partial_{i}\zeta\|_{L^{2}(\Omega)}.
    \end{equation}
  
    For $I_{3}$,
    \begin{equation}\label{divergenceH1EstimateI3}
        |I_{3}|\leq \eta\|\partial_{1}\partial_{i}\div\mathbf{u}\|_{L^{2}(\Omega)}^{2}+\frac{C}{\eta}\|\partial_{i}\mathbf{F}\|_{L^{2}(\Omega)}^{2}.
    \end{equation}
  
    For $I_{4}$,
    \begin{equation}\label{divergenceH1EstimateI4}
        |I_{4}|\leq R\overline{\theta}\|\partial_{1}\partial_{i}\div\mathbf{u}\|_{L^{2}(\Omega)}\|\partial_{1}\partial_{i}\varphi'\|_{L^{2}(\Omega)}.
    \end{equation}

    Combining \eqref{divergenceH1EstimateI1}, \eqref{divergenceH1EstimateI2}, \eqref{divergenceH1EstimateI3}, \eqref{divergenceH1EstimateI4}, we obtain the desired estimate \eqref{divergenceH1Estimate}.
\end{proof}

Now, we give the proof of Theorem \ref{boundednessTheorem}.

\begin{proof}[Proof of Theorem \ref{boundednessTheorem}]
    Let $\lambda_{0},\lambda_{1}$ be chosen such that the assumption (A) is valid, then combining the results of Lemma \ref{densityBasicEstimateLemma}, Lemma \ref{basicEstimateLemma} and Lemma \ref{H1ConormalBasicEstimateLemma} as $\eqref{densityBasicEstimate}+\lambda_{0}\eqref{basicEstimate}+\lambda_{1}\eqref{H1ConormalBasicEstimate}$, we have
    \begin{align*}
        &\frac{4}{3}\frac{\mu(\overline{\theta})}{\overline{\rho}}\frac{1}{2h}(\|\partial_{1}\varphi\|_{L^{2}(\Omega)}^2-\|\partial_{1}\varphi'\|_{L^{2}(\Omega)}^2)+\left(\frac{4}{3}\frac{\mu(\overline{\theta})}{\overline{\rho}}\frac{1}{2h}-\frac{R\overline{\theta}}{2}\right)\|\partial_{1}\varphi-\partial_{1}\varphi'\|_{L^{2}(\Omega)}^2 \\
        &+(\frac{R\overline{\theta}}{2}-\eta)\|\partial_{1}\varphi\|_{L^{2}(\Omega)}^2+\frac{R\overline{\theta}}{2}\|\partial_{1}\varphi'\|_{L^{2}(\Omega)}^2 \\
        &+\lambda_{0}\frac{R}{4h}\frac{\overline{\theta}}{\overline{\rho}}\sqrt{\frac{2}{R}}\frac{2}{3R}a_{\mathbf{u}}^{I}a_{\mathbf{u}}^{II}\frac{\mu(\overline{\theta})}{\sqrt{\overline{\theta}}}(\|\varphi\|_{L^{2}(\Omega)}^2-\|\varphi'\|_{L^{2}(\Omega)}^2+\|\varphi-\varphi'\|_{L^{2}(\Omega)}^2) \\
        &+\lambda_{0}\left(\frac{1}{3R}\sqrt{\frac{2}{R}}a_{\mathbf{u}}^{I}a_{\mathbf{u}}^{II}\frac{\mu(\overline{\theta})^{2}}{\sqrt{\overline{\theta}}}-C\delta\right)\left\|\nabla\mathbf{u}+\nabla\mathbf{u}^{\mathsf{T}}-\frac{2}{3}\div\mathbf{u}\mathbb{I}_{3}\right\|_{L^{2}(\Omega)}^{2} \\
        &+\lambda_{0}\left(\frac{32}{75R^2}\sqrt{\frac{2}{R}}a_{\theta}^{I}a_{\theta}^{II}\frac{\kappa(\overline{\theta})^{2}}{\overline{\theta}\sqrt{\overline{\theta}}}-C\delta\right)\|\nabla\zeta\|_{L^{2}(\Omega)}^2 \\
        &+\lambda_{0}\left(\frac{2}{3R}a_{\mathbf{u}}^{II}\mu(\overline{\theta})\overline{\rho}-C\delta-C\varepsilon\right)\|\mathbf{u}\|_{L^{2}(\partial\Omega)}^2+\lambda_{0}\left(\frac{16}{15R}a_{\theta}^{I}\frac{\kappa(\overline{\theta})}{\overline{\theta}}\overline{\rho}-C\delta-C\varepsilon\right)\|\zeta\|_{L^{2}(\partial\Omega)}^2 \\
        &+\lambda_{1}\frac{R}{4h}\frac{\overline{\theta}}{\overline{\rho}}\sqrt{\frac{2}{R}}\frac{2}{3R}a_{\mathbf{u}}^{I}a_{\mathbf{u}}^{II}\frac{\mu(\overline{\theta})}{\sqrt{\overline{\theta}}}(\|\partial_{i}\varphi\|_{L^{2}(\Omega)}^2-\|\partial_{i}\varphi'\|_{L^{2}(\Omega)}^2+\|\partial_{i}\varphi-\partial_{i}\varphi'\|_{L^{2}(\Omega)}^2) \\
        &+\lambda_{1}\left(\frac{1}{3R}\sqrt{\frac{2}{R}}a_{\mathbf{u}}^{I}a_{\mathbf{u}}^{II}\frac{\mu(\overline{\theta})^{2}}{\sqrt{\overline{\theta}}}-C\delta\right)\left\|\partial_{i}\nabla\mathbf{u}+\partial_{i}\nabla\mathbf{u}^{\mathsf{T}}-\frac{2}{3}\partial_{i}\div\mathbf{u}\mathbb{I}_{3}\right\|_{L^{2}(\Omega)}^{2} \\
        &+\lambda_{1}\left(\frac{32}{75R^2}\sqrt{\frac{2}{R}}a_{\theta}^{I}a_{\theta}^{II}\frac{\kappa(\overline{\theta})^{2}}{\overline{\theta}\sqrt{\overline{\theta}}}-C\delta\right)\|\nabla\partial_{i}\zeta\|_{L^{2}(\Omega)}^2 \\
        &+\lambda_{1}\left(\frac{2}{3R}a_{\mathbf{u}}^{II}\mu(\overline{\theta})\overline{\rho}-C\delta-C\varepsilon\right)\|\partial_{i}\mathbf{u}\|_{L^{2}(\partial\Omega)}^2+\left(\frac{16}{15R}a_{\theta}^{I}\frac{\kappa(\overline{\theta})}{\overline{\theta}}\overline{\rho}-C\delta-C\varepsilon\right)\|\partial_{i}\zeta\|_{L^{2}(\partial\Omega)}^2 \\
        \leq&\mu(\overline{\theta})\|\partial_{1}\varphi\|_{L^{2}(\Omega)}\sum_{j=2,3}\|\nabla\partial_{j}\mathbf{u}\|_{L^{2}(\Omega)} +R\overline{\rho}\|\partial_{1}\varphi\|_{L^{2}(\Omega)}\|\partial_{1}\zeta\|_{L^{2}(\Omega)} \\
        &+\lambda_{0}\left(\frac{2}{5R^2}a_{\mathbf{u}}^{II}a_{\theta}^I\frac{\kappa(\overline{\theta})}{\overline{\theta}}+C\delta\right)\left\|\nabla\mathbf{u}+\nabla\mathbf{u}^{\mathsf{T}}-\frac{2}{3}\div\mathbf{u}\mathbb{I}_{3}\right\|_{L^{2}(\Omega)}\|\nabla\zeta\|_{L^{2}(\Omega)} \nonumber\\
        &+\lambda_{0}\left(R\frac{4}{5R^2}a_{\mathbf{u}}^{II}a_\theta^I\kappa(\overline{\theta})+C\delta\right)\|\zeta\|_{L^{2}(\Omega)}\|\partial_{1}\varphi'\|_{L^{2}(\Omega)} \nonumber\\
        &+\lambda_{0}C\delta(\|\nabla\varphi'\|_{L^{2}(\Omega)}^{2}+\|\mathbf{u}\|_{L^{2}(\Omega)}^{2})+\lambda_{0}\frac{h}{CR}\overline{\rho}\overline{\theta}\sqrt{\frac{2}{R}}\frac{2}{3R}a_{\mathbf{u}}^{I}a_{\mathbf{u}}^{II}\frac{\mu(\overline{\theta})}{\sqrt{\overline{\theta}}}\|\nabla\mathbf{u}\|_{L^{2}(\Omega)}^{2} \nonumber\\
        &+\lambda_{1}\left(\frac{2}{5R^2}a_{\mathbf{u}}^{II}a_{\theta}^I\frac{\kappa(\overline{\theta})}{\overline{\theta}}+C\delta\right)\left\|\partial_{i}\nabla\mathbf{u}+\partial_{i}\nabla\mathbf{u}^{\mathsf{T}}-\frac{2}{3}\partial_{i}\div\mathbf{u}\mathbb{I}_{3}\right\|_{L^{2}(\Omega)}\|\partial_{i}\nabla\zeta\|_{L^{2}(\Omega)} \nonumber\\
        &+\lambda_{1}\left(R\frac{4}{5R^2}a_{\mathbf{u}}^{II}a_\theta^I\kappa(\overline{\theta})+C\delta\right)\|\partial_{i}^2\zeta\|_{L^{2}(\Omega)}\|\partial_{1}\varphi'\|_{L^{2}(\Omega)} \nonumber\\
        &+\max\{\lambda_{0},\lambda_{1}\}C\delta\sum_{0\leq j\leq1}(\|\nabla\partial_{i}^{j}\varphi'\|_{L^{2}(\Omega)}^{2}+\|\partial_{i}^{j}\mathbf{u}\|_{L^{2}(\Omega)}^{2})+\max\{\lambda_{0},\lambda_{1}\}Ch\sum_{0\leq j\leq1}\|\nabla\partial_{i}^{j}\mathbf{u}\|_{L^{2}(\Omega)}^{2} \\
        &+\max\{\lambda_{0},\lambda_{1}\}\eta\sum_{0\leq j\leq2}(\|\partial_{i}^{j}\mathbf{u}\|_{L^{2}(\Omega)}^{2}+\|\partial_{i}^{j}\zeta\|_{L^{2}(\Omega)}^{2})+\max\{\lambda_{0},\lambda_{1}\}\frac{C}{\eta}(1+N_{1}(\varphi,\mathbf{u},\zeta))^{2}(\sqrt{\delta}+N_{2}(\varphi,\mathbf{u},\zeta))^{4}.
    \end{align*} 
    Then choosing $\delta, h, \eta>0$ sufficiently small, and by assumption (A), Korn's inequality \eqref{KornInequality} and Poincar\'{e}'s inequality \eqref{PoincareTraceInequality}, there exists a constant $C>0$ such that
    \begin{align}\label{partBoundednessH2Estimate}
        &\|\partial_{1}\varphi\|_{L^{2}(\Omega)}^{2}+\frac{1}{h}(\|\varphi\|_{H^{1}(\Omega)}^{2}-\|\varphi'\|_{H^{1}(\Omega)}^{2}+\|\varphi-\varphi'\|_{H^{1}(\Omega)}^{2}) \nonumber\\
        &+\sum_{0\leq j\leq1}\|\partial_{i}^{j}\mathbf{u}\|_{H^{1}(\Omega)}^{2} +\sum_{0\leq j\leq1}\|\partial_{i}\zeta\|_{H^{1}(\Omega)}^{2} \nonumber\\
        \leq&C(1+N_{1}(\varphi,\mathbf{u},\zeta))^{2}(\sqrt{\delta}+N_{2}(\varphi,\mathbf{u},\zeta))^{4}.
    \end{align}
  
    Applying \eqref{partBoundednessH2Estimate} to Lemma \ref{divergenceBasicEstimateLemma}, we also have
    \begin{equation*}
        \left(\frac{4}{3}\mu(\overline{\theta})-\eta\right)\|\partial_{1}\div\mathbf{u}\|_{L^{2}(\Omega)}^{2} \nonumber\\
        \leq C(1+N_{1}(\varphi,\mathbf{u},\zeta))^{2}(\sqrt{\delta}+N_{2}(\varphi,\mathbf{u},\zeta))^{4}.
    \end{equation*}
    Together with the Stokes' estimate \eqref{StokesEstimate} in Lemma \ref{StokesEstimateTheorem} and the elliptic estimate \eqref{ellipticEstimate} in Lemma \ref{ellipticEstimateTheorem}, we have
    \begin{equation}\label{H2StokesEllipticEstimate}
        \|\nabla\varphi'\|_{L^{2}(\Omega)}^{2}+\|\mathbf{u}\|_2^2+\|\zeta\|_2^2 
        \leq C(1+N_{1}(\varphi,\mathbf{u},\zeta))^{2}(\sqrt{\delta}+N_{2}(\varphi,\mathbf{u},\zeta))^{4}.
    \end{equation}

    Combining \eqref{partBoundednessH2Estimate} and \eqref{H2StokesEllipticEstimate}, by Poincar\'{e}'s inequality \eqref{PoincareInequality} and choosing $h>0$ sufficiently small, we obtain
    \begin{equation}\label{BoundednessH2Estimate}
        \|\varphi\|_{H^{1}(\Omega)}^{2}+\frac{1}{h}\sum_{0\leq j\leq1}(\|\partial_{i}^{j}\varphi\|_{L^{2}(\Omega)}^{2}-\|\partial_{i}^{j}\varphi'\|_{L^{2}(\Omega)}^{2})+\|\mathbf{u}\|_{H^{2}(\Omega)}^{2}+\|\zeta\|_{H^{2}(\Omega)}^{2}\leq C(1+N_{1}(\varphi,\mathbf{u},\zeta))^{2}(\sqrt{\delta}+N_{2}(\varphi,\mathbf{u},\zeta))^{4}.
    \end{equation}

    Let $i=2,3$ and $\lambda_{0}$ be a constant such that the assumption (A) is valid, then by Lemma \ref{H2ConormalBasicEstimateLemma} and Lemma \ref{densityH1EstimateLemma} with $i=2,3$, it follows that
    \begin{align*}
        &\frac{4}{3}\frac{\mu(\overline{\theta})}{\overline{\rho}}\frac{1}{2h}(\|\partial_{1}\partial_{i}\varphi\|_{L^{2}(\Omega)}^{2}-\|\partial_{1}\partial_{i}\varphi'\|_{L^{2}(\Omega)}^{2})+\left(\frac{4}{3}\frac{\mu(\overline{\theta})}{\overline{\rho}}\frac{1}{2h}-\frac{R\overline{\theta}}{2}\right)\left\|\partial_{1}\partial_{i}\varphi-\partial_{1}\partial_{i}\varphi'\right\|_{L^{2}(\Omega)}^{2} \\
        &+(\frac{R\overline{\theta}}{2}-\eta)\|\partial_{1}\partial_{i}\varphi\|_{L^{2}(\Omega)}^{2}+\frac{R\overline{\theta}}{2}\|\partial_{1}\partial_{i}\varphi'\|_{L^{2}(\Omega)}^{2} \\
        &+\lambda_{0}\frac{R}{4h}\frac{\overline{\theta}}{\overline{\rho}}\sqrt{\frac{2}{R}}\frac{2}{3R}a_{\mathbf{u}}^{I}a_{\mathbf{u}}^{II}\frac{\mu(\overline{\theta})}{\sqrt{\overline{\theta}}}\left(\|\partial_{i}^2\varphi\|_{L^{2}(\Omega)}^2-\|\partial_{i}^2\varphi'\|_{L^{2}(\Omega)}^2+\|\partial_{i}^2\varphi-\partial_{i}^2\varphi'\|_{L^{2}(\Omega)}^2\right) \\
        &+\lambda_{0}\left(\frac{1}{3R}\sqrt{\frac{2}{R}}a_{\mathbf{u}}^{I}a_{\mathbf{u}}^{II}\frac{\mu(\overline{\theta})^{2}}{\sqrt{\overline{\theta}}}-C\delta\right)\left\|\partial_{i}^2\nabla\mathbf{u}+\partial_{i}^{2}\nabla\mathbf{u}^{\mathsf{T}}-\frac{2}{3}\partial_{i}^{2}\div\mathbf{u}\mathbb{I}_{3}\right\|_{L^{2}(\Omega)}^2 \\
        &+\lambda_{0}\left(\frac{32}{75R^2}\sqrt{\frac{2}{R}}a_{\theta}^{I}a_{\theta}^{II}\frac{\kappa(\overline{\theta})^{2}}{\overline{\theta}\sqrt{\overline{\theta}}}-C\delta\right)\|\nabla\partial_{i}^2\zeta\|_{L^{2}(\Omega)}^2 \\
        &+\lambda_{0}\left(\frac{2}{3R}a_{\mathbf{u}}^{II}\mu(\overline{\theta})\overline{\rho}-C\delta-C\varepsilon\right)\|\partial_{i}^2\mathbf{u}\|_{L^{2}(\partial\Omega)}^2+\left(\frac{16}{15R}a_{\theta}^{I}\frac{\kappa(\overline{\theta})}{\overline{\theta}}\overline{\rho}-C\delta-C\varepsilon\right)\|\partial_{i}^2\zeta\|_{L^{2}(\partial\Omega)}^2 \\
        \leq&\mu(\overline{\theta})\|\partial_{1}\varphi\|_{L^{2}(\Omega)}\sum_{j=2,3}\|\partial_{j}\nabla^{2}\mathbf{u}\|_{L^{2}(\Omega)} +R\overline{\rho}\|\partial_{1}\varphi\|_{L^{2}(\Omega)}\|\partial_{1}\zeta\|_{L^{2}(\Omega)} \\
        &+\lambda_{0}\left(\frac{2}{5R^2}a_{\mathbf{u}}^{II}a_{\theta}^I\frac{\kappa(\overline{\theta})}{\overline{\theta}}+C\delta\right)\left\|\partial_{i}^{2}\nabla\mathbf{u}+\partial_{i}^{2}\nabla\mathbf{u}^{\mathsf{T}}-\frac{2}{3}\partial_{i}^{2}\div\mathbf{u}\mathbb{I}_{3}\right\|_{L^{2}(\Omega)}\|\partial_{i}^{2}\nabla\zeta\|_{L^{2}(\Omega)} \nonumber\\
        &+\lambda_{0}\left(R\frac{4}{5R^2}a_{\mathbf{u}}^{II}a_\theta^I\kappa(\overline{\theta})+C\delta\right)\|\partial_{i}^3\zeta\|_{L^{2}(\Omega)}\|\partial_{1}\partial_{i}\varphi'\|_{L^{2}(\Omega)} \nonumber\\
        &+\lambda_{0}C\delta(\|\nabla\partial_{i}^{2}\varphi'\|_{L^{2}(\Omega)}^{2}+\|\partial_{i}\mathbf{u}\|_{L^{2}(\Omega)}^{2})+Ch\|\nabla\partial_{i}^{2}\mathbf{u}\|_{L^{2}(\Omega)}^{2} \nonumber\\
        &+\lambda_{0}\eta\sum_{j=2,3}(\|\partial_{i}^{j}\mathbf{u}\|_{L^{2}(\Omega)}^{2}+\|\partial_{i}^{j}\zeta\|_{L^{2}(\Omega)}^{2})+\lambda_{0}\frac{C}{\eta}(1+N_{2}(\varphi,\mathbf{u},\zeta))^{2}(\sqrt{\delta}+N_{2}(\varphi,\mathbf{u},\zeta))^{4}.
    \end{align*}
    Then choosing $\delta, h, \eta>0$ sufficiently small, and by assumption (A), Korn's inequality \eqref{KornInequality}, Poincar\'{e}'s inequality \eqref{PoincareTraceInequality} and \eqref{BoundednessH2Estimate}, there exists a constant $C>0$ such that
    \begin{align}\label{partBoundednessConormalH3Estimate}
        &\|\partial_{1}\partial_{i}\varphi\|_{L^{2}(\Omega)}^{2}+\frac{1}{h}\left(\|\partial_{i}\varphi\|_{H^{1}(\Omega)}^{2}-\|\partial_{i}\varphi'\|_{H^{1}(\Omega)}^{2}+\|\partial_{i}\varphi-\partial_{i}\varphi'\|_{H^{1}(\Omega)}^{2}\right) \nonumber\\
        &+\|\partial_{i}^{2}\mathbf{u}\|_{H^{1}(\Omega)}^{2}+\|\partial_{i}^{2}\zeta\|_{H^{1}(\Omega)}^{2} \nonumber\\
        \leq&C(1+N_{2}(\varphi,\mathbf{u},\zeta))^{2}(\sqrt{\delta}+N_{2}(\varphi,\mathbf{u},\zeta))^{4}.
    \end{align}
  
    Applying \eqref{partBoundednessConormalH3Estimate} to Lemma \ref{divergenceH1Estimate}, we have
    \begin{equation*}
        \left(\frac{4}{3}\mu(\overline{\theta})-\eta\right)\|\partial_{1}\partial_{i}\div\mathbf{u}\|_{L^{2}(\Omega)}^{2}
        \leq C(1+N_{2}(\varphi,\mathbf{u},\zeta))^{2}(\sqrt{\delta}+N_{2}(\varphi,\mathbf{u},\zeta))^{4}.
    \end{equation*}
     Together with \eqref{BoundednessH2Estimate}, the Stokes' estimate \eqref{StokesEstimate} in Lemma \ref{StokesEstimateTheorem} and the elliptic estimate \eqref{ellipticEstimate} in Lemma \ref{ellipticEstimateTheorem}, we have
    \begin{equation}\label{conormalStokesEllipticEstimate}
        \|\partial_{i}\nabla\varphi'\|_{L^{2}(\Omega)}^{2}+\|\partial_{i}\mathbf{u}\|_2^2+\|\partial_{i}\zeta\|_2^2
        \leq C(1+N_{2}(\varphi,\mathbf{u},\zeta))^{2}(\sqrt{\delta}+N_{2}(\varphi,\mathbf{u},\zeta))^{4}.
    \end{equation}

    Combining \eqref{partBoundednessConormalH3Estimate} and \eqref{conormalStokesEllipticEstimate}, by Poincar\'{e}'s inequality \eqref{PoincareInequality} and choosing $h>0$ sufficiently small, we obtain
    \begin{align}\label{BoundednessConormalH3Estimate}
        &\|\partial_{i}\varphi\|_{H^{1}(\Omega)}^{2}+\frac{1}{h}\left(\|\partial_{i}\varphi\|_{H^{1}(\Omega)}^{2}-\|\partial_{i}\varphi'\|_{H^{1}(\Omega)}^{2}\right) \nonumber\\
        &+\|\partial_{i}\mathbf{u}\|_{H^{2}(\Omega)}^{2}+\|\partial_{i}\zeta\|_{H^{2}(\Omega)}^{2} \nonumber\\
        \leq&C(1+N_{2}(\varphi,\mathbf{u},\zeta))^{2}(\sqrt{\delta}+N_{2}(\varphi,\mathbf{u},\zeta))^{4}.
    \end{align}

    Let $i=1$, applying \eqref{BoundednessConormalH3Estimate} to Lemma \ref{divergenceH1EstimateLemma}, we have
    \begin{equation*}
        \|\partial_{1}^2\div\mathbf{u}\|_{L^{2}(\Omega)}^{2}
        \leq C(1+N_{2}(\varphi,\mathbf{u},\zeta))^{2}(\sqrt{\delta}+N_{2}(\varphi,\mathbf{u},\zeta))^{4}.
    \end{equation*}
    Together with \eqref{BoundednessConormalH3Estimate}, Stokes' estimate \eqref{StokesEstimate} in Lemma \ref{StokesEstimateTheorem} and the elliptic estimate \eqref{ellipticEstimate} in Lemma \ref{ellipticEstimateTheorem}, we have
    \begin{equation}\label{H3StokesEllipticEstimate}
        \|\nabla\varphi'\|_{H^{1}(\Omega)}^{2}+\|\mathbf{u}\|_{H^{3}(\Omega)}^2+\|\zeta\|_{H^{3}(\Omega)}^{2} 
        \leq C(1+N_{2}(\varphi,\mathbf{u},\zeta))^{2}(\sqrt{\delta}+N_{2}(\varphi,\mathbf{u},\zeta))^{4}.
    \end{equation}
    Then applying \eqref{BoundednessConormalH3Estimate} to Lemma \ref{densityH1EstimateLemma}, we have
    \begin{align*}
        &\|\partial_{1}^2\varphi\|_{L^{2}(\Omega)}^{2}+\|\partial_{1}^2\varphi'\|_{L^{2}(\Omega)}^{2}+\frac{1}{h}\left(\|\partial_{1}^2\varphi\|_{L^{2}(\Omega)}^{2}-\|\partial_{1}^2\varphi'\|_{L^{2}(\Omega)}^{2}+\left\|\partial_{1}^2\varphi-\partial_{1}^2\varphi'\right\|_{L^{2}(\Omega)}^{2}\right) \\
        \leq&C(1+N_{2}(\varphi,\mathbf{u},\zeta))^{2}(\sqrt{\delta}+N_{2}(\varphi,\mathbf{u},\zeta))^{4}.
    \end{align*}
    Together with \eqref{BoundednessConormalH3Estimate} and \eqref{H3StokesEllipticEstimate}, by Poincar\'{e}'s inequality and choosing $h>0$ sufficiently small, we obtain
    \begin{align}\label{BoundednessH3Estimate}
        &\|\varphi\|_{H^{2}(\Omega)}^{2}+\frac{1}{h}\left(\|\varphi\|_{H^{2}(\Omega)}^{2}-\|\varphi'\|_{H^{2}(\Omega)}^{2}\right) \nonumber\\
        &+\|\mathbf{u}\|_{H^{3}(\Omega)}^{2}+\|\zeta\|_{H^{3}(\Omega)}^{2} \nonumber\\
        \leq&C(1+N_{2}(\varphi,\mathbf{u},\zeta))^{2}(\sqrt{\delta}+N_{2}(\varphi,\mathbf{u},\zeta))^{4}.
    \end{align}
    Then by choosing $\delta>0$ sufficiently small in \eqref{BoundednessH3Estimate}, we obtain that $N_{2}(\varphi,\mathbf{u},\zeta)\leq\varepsilon$ and therefore complete the proof of Theorem \ref{boundednessTheorem}.
\end{proof}

\section{Proof of the main theorem}
\begin{proof}[Proof of Theorem \ref{mainTheorem}]
    Consider the Hilbert space
    \begin{equation*}
      H:=\{(\varphi,\mathbf{u},\zeta)\in L_{0}^{2}(\Omega)\times V^{1}(\Omega)\times H^{1}(\Omega)\},
    \end{equation*}
    and let $K$ be the following compact convex subset of $H$:
    \begin{equation*}
      K:=\{(\varphi,\mathbf{u},\zeta)\in (L_{0}^{2}\cap H^{2}(\Omega))\times V^{3}(\Omega)\times H^{3}(\Omega): N_{2}(\varphi,\mathbf{u},\zeta)\leq\varepsilon\}.
    \end{equation*}
    
    By Lemma \ref{continuityTheorem}, for arbitrary $(\varphi_{i}',\mathbf{u}_{i}',\zeta_{i}')\in K$,
   \begin{equation*}
     \|\mathds{T}(\varphi_{1}',\mathbf{u}_{1}',\zeta_{1}')-\mathds{T}(\varphi_{2}',\mathbf{u}_{2}',\zeta_{2}')\|_{H}\leq C\|(\varphi_{1}'-\varphi_{2}',\mathbf{u}_{1}'-\mathbf{u}_{2}',\zeta_{1}'-\zeta_{2}')\|_{H},
   \end{equation*}
   which implies that $\mathds{T}$ is a continuous operator with respect to the topology of $H$. Moreover, by Lemma \ref{boundednessTheorem}, for arbitrary $(\varphi',\mathbf{u}',\zeta')\in K$, we have $\mathds{T}(\varphi',\mathbf{u}',\zeta')\in K$. Therefore, by Schauder's fixed point theorem, there exists a unique point $(\varphi,\mathbf{u},\zeta)\in K$ such that 
   \begin{equation*}
     \mathds{T}(\varphi,\mathbf{u},\zeta)=(\varphi,\mathbf{u},\zeta),
   \end{equation*}
   which implies that $(\varphi,\mathbf{u},\zeta)\in K$ is a solution to the problem \eqref{PSNS} and \eqref{PBC}. Then by our construction, $(\rho,\mathbf{u},\theta):=(\overline{\rho}+\varphi,\mathbf{u},\tilde{\theta}+\zeta)$ is a solution to the problem \eqref{SNS} and \eqref{BC} and
   \begin{equation*}
     \|(\rho-\overline{\rho},\mathbf{u},\theta-\overline{\theta})\|_{K}^{2}\leq C\|(\rho-\overline{\rho},\mathbf{u},\theta-\tilde{\theta})\|_{K}^{2},
   \end{equation*}
   which implies \eqref{mainEstimate}. The proof of Theorem \ref{mainTheorem} is complete.
\end{proof}

\section{Appendices}
Korn's inequality is commonly used in fluid dynamics, particularly in the presence of Navier's slip boundary conditions. Roughly speaking, Korn's inequality asserts that the $L^{2}(\Omega)$ norm of a gradient field can be controlled by the $L^{2}(\Omega)$ norm of its symmetric part under certain conditions. For our purpose, we give the following Korn's type inequality for $\Omega=(0,1)\times\mathbb{T}^2$.

\begin{lemma}
Let $\Omega=(0,1)\times\mathbb{T}^2$, it holds that
      \begin{equation}\label{KornInequality}
        \|\nabla\mathbf{u}\|_{L^{2}(\Omega)}^2\leq\frac{1}{2}\left\|\nabla\mathbf{u}+\nabla\mathbf{u}^{\mathsf{T}}-\frac{2}{3}\div\mathbf{u}\right\|_{L^{2}(\Omega)}^{2},
      \end{equation}
      for all $\mathbf{u}\in V^{1}(\Omega)$.
\end{lemma}

\begin{proof}
    Since
    \begin{equation*}
      \frac{1}{2}\left\|\nabla\mathbf{u}+\nabla\mathbf{u}^{\mathsf{T}}-\frac{2}{3}\div\mathbf{u}\right\|_{L^{2}(\Omega)}^{2}=\frac{1}{2}\|\nabla\mathbf{u}+\nabla\mathbf{u}^{\mathsf{T}}\|_{L^{2}(\Omega)}^2-\frac{2}{3}\|\div\mathbf{u}\|_{L^{2}(\Omega)}^2,
    \end{equation*}
    and
    \begin{equation*}
    \frac{1}{2}\int_{(0,1)\times\mathbb{T}^2}|\nabla\mathbf{u}+\nabla\mathbf{u}^{\mathsf{T}}|^2d\mathbf{x}=\int_{(0,1)\times\mathbb{T}^2}\nabla\mathbf{u}:(\nabla\mathbf{u}+\nabla\mathbf{u}^{\mathsf{T}})d\mathbf{x}=\int_{(0,1)\times\mathbb{T}^2}|\nabla\mathbf{u}|^2+|\div\mathbf{u}|^2d\mathbf{x}.
    \end{equation*}
    Combining the above two facts, the result is proved.
\end{proof}

\begin{remark}
  It is important to compute explicitly Korn's constant to verify the assumption (A).  For Korn's inequality with the slip (tangency) boundary condition in general domains, DESVILLETTES and VILLANI \cite{Desvillettes-Villani-COCV-2002} showed that Korn's inequality holds for non-axisymmetric $C^{1}$ domains under the slip (tangency) boundary condition. They also characterized the Korn's constant but without explicit value. Later, BAUER and PAULY \cite{Bauer-Pauly-MMAS-2016, Bauer-Pauly-AUF-2016} proved the Korn's inequality in non-axisymmetric Lipschitz domains under the slip (tangency) boundary condition. See also LEWICKA and MULLER's result \cite{Lewicka-Muller-IUMJ-2016} for domains not being non-axisymmetric. 
\end{remark}

Poincar\'{e} inequality with traces is another important tool, asserting that the integrability of a function can be estimated by the integrability of its gradient and trace. For our purpose, we show the validity of the following Poincar\'{e} type inequality for $\Omega=(0,1)\times\mathbb{T}^2$.

\begin{lemma}
Let $\Omega=(0,1)\times\mathbb{T}^2$, it holds that
\begin{equation}\label{PoincareTraceInequality}
  \|\zeta\|_{L^{2}(\Omega)}^{2}\leq\|\nabla\zeta\|_{L^{2}(\Omega)}^{2}+\|\zeta\|_{L^{2}(\partial\Omega)}^{2},
\end{equation}
for all $\zeta\in H^{1}(\Omega)$.
\end{lemma}
\begin{proof}
Using the fundamental theorem of Calculus,
\begin{align*}
  &\left|\int_{\mathbb{T}^2}\int_{0}^{1}|\zeta(x_1,x')|^2d\mathbf{x}_1d\mathbf{x}'-\frac{1}{2}\left(\int_{\mathbb{T}^2}|\zeta(0,x')|^2d\mathbf{x}'+\int_{\mathbb{T}^2}|\zeta(1,x')|^2d\mathbf{x}'\right)\right|\\
  \leq&\left|\int_{\mathbb{T}^2}\int_{0}^{\frac{1}{2}}\int_{0}^{x_1}2\zeta(y_1,x')\partial_1\zeta(y_1,x')dy_1d\mathbf{x}_1d\mathbf{x}'\right|+\left|\int_{\mathbb{T}^2}\int_{\frac{1}{2}}^{1}\int_{x_1}^{1}2\zeta(y_1,x')\partial_1\zeta(y_1,x')dy_1d\mathbf{x}_1d\mathbf{x}'\right|\\
  \leq&\int_{\mathbb{T}^2}\int_{0}^{1}|\zeta(y_1,x')|\cdot|\partial_1\zeta(y_1,x')|dy_1d\mathbf{x}'\\
  \leq&\frac{1}{2}\int_{\mathbb{T}^2}\int_{0}^{1}|\zeta(x_1,x')|^2d\mathbf{x}'+\frac{1}{2}\int_{\mathbb{T}^2}\int_{0}^{1}|\partial_1\zeta(x_1,x')|^2d\mathbf{x}_1d\mathbf{x}',
\end{align*}
which implies the result.
\end{proof}

\begin{remark}
   It is also important to compute explicitly Poincar\'{e}'s constant to verify the assumption (A). For Poincar\'{e} inequality with traces in general domains, BREZIS and LIEB \cite{Brezis-Lieb-JFA-1985} proved that for bounded domain $\Omega\in\mathbb{R}^n$, the $L^{\frac{2n}{n-2}}(\Omega)$ norm of a function can be estimated by the $L^2(\Omega)$ norm of its gradient and $L^{\frac{2(n-1)}{n-2}}(\Omega)$ norm of its trace. Later, MAGGI and VILLANI \cite{Maggi-Villani-JGA-2005} extended BREZIS and LIEB's result to $L^{p}(\Omega)$ case for $p\in[1,n)$ and they gave explicit bounds for the constants involved in the inequality by using the transportation techniques. More recently, BUCUR, GIACOMINI and TREBESCHI \cite{Bucur-Giacomini-Trebeschi-AIHP-2019} proved a control of $L^{p}(\Omega)$ norm a function by the $L^{p}(\Omega)$ norm of its gradient and $L^{p}(\Omega)$ norm of its traces, they also give a best constant description in their result. 
\end{remark}

In the rest of this section, we give some well-known results and theorems used in our proof for the reader's convenience.

First, we recall 
the following type of Poincar\'{e}'s inequality.
\begin{lemma}[Poincar\'{e}'s inequality]
  Let $\Omega\in\mathbb{R}^{3}$ be a smooth bounded domain. Then there exists a constant $C>0$ such that
  \begin{equation}\label{PoincareInequality}
    \|\varphi\|_{L^{2}(\Omega)}\leq C\|\nabla\varphi\|_{L^{2}(\Omega)},
  \end{equation}
  for all $\varphi\in H^{1}(\Omega)\cap L_{0}^{2}(\Omega)$.
\end{lemma}

Second, we refer to \cite{Grisvard-1985, Tapia-Amrouche-Conca-Ghosh-JDE-2021} for the results concerning the inhomogeneous boundary value problems for Poisson equation with Robin boundary condition and Stokes system with slip boundary conditions.

Consider the boundary value problem for $\zeta$:
\begin{equation}\label{Poisson}
  -\Delta\zeta=\mathfrak{G},
\end{equation}
for $x\in\Omega$, with
\begin{equation}\label{RobinBC}
\zeta+\mathfrak{a}\mathbf{n}\cdot\nabla\zeta+\mathfrak{b}=0,
\end{equation}
for $x\in\partial\Omega$. Then we have the following theorem. 
\begin{theorem}\label{ellipticEstimateTheorem}
  Let $\Omega\subset\mathbb{R}^3$ be a smooth bounded domain, $k\in\mathbb{N}$ and $\mathfrak{a}\in C^{2}(\partial\Omega)$ is a positive function. If $\mathfrak{G}\in H^{1}(\Omega)$ and $\mathfrak{b}\in H^{k+\frac{1}{2}}(\partial\Omega)$, then the boundary value problem \eqref{Poisson} with \eqref{RobinBC} has a unique strong solution $\zeta\in H^{k+2}(\Omega)$. Moreover,
  \begin{equation}\label{ellipticEstimate}
    \|\zeta\|_{H^{k+2}(\Omega)}^2\leq C(\|\mathfrak{G}\|_{H^{k}(\Omega)}^2+\|\mathfrak{b}\|_{H^{k+\frac{1}{2}}(\partial\Omega)}^2).
  \end{equation}
\end{theorem}

Consider the boundary value problem for $(\mathbf{u},\nabla\varphi)$:
\begin{equation}\label{Stokes}
\left\{
\begin{aligned}
   -\Delta\mathbf{u}+\nabla\varphi=&\boldsymbol{\mathfrak{F}}, \\
   \div\mathbf{u}=&\mathfrak{G}, 
\end{aligned}\right.
\end{equation}
for $x\in\Omega$, with \begin{equation}\label{slipBC}
    \left\{
    \begin{aligned}
    &\mathbf{u}\cdot\mathbf{n}=0, \\
    &\mathbf{u}\cdot\mathbf{t}+\mathfrak{a}(\nabla\mathbf{u}+\nabla\mathbf{u}^{\mathsf{T}}):\mathbf{n}\otimes\mathbf{t}+\mathfrak{b}=0,
    \end{aligned}
    \right.
\end{equation}
for $x\in\partial\Omega$. Then we have the following results.
\begin{theorem}\label{StokesEstimateTheorem}
    Let $\Omega\subset\mathbb{R}^3$ be a smooth bounded domain, $k\in\mathbb{N}$ and $\mathfrak{a}\in C^{2}(\partial\Omega)$ is a positive function. If$\mathfrak{F}\in H^{k}(\Omega)$, $\mathfrak{G}\in H^{k}(\Omega)\cap L_0^2(\Omega)$ and $\mathfrak{b}\in H^{k+\frac{1}{2}}(\partial\Omega)$, then the boundary value problem \eqref{Stokes} with \eqref{slipBC} has a unique strong solution $(\mathbf{u},\varphi)\in V^{k+2}(\Omega)\times (H^{k+1}(\Omega)\cap L_{0}^{2}(\Omega))$. Moreover,
    \begin{equation}\label{StokesEstimate}
        \|\mathbf{u}\|_{H^{k+2}(\Omega)}^2+\|\varphi\|_{H^{k+1}(\Omega)}^2\leq C(\left\|\boldsymbol{\mathfrak{F}}\right\|_{H^{k}(\Omega)}^2+\|\mathfrak{G}\|_{H^{k+1}(\Omega)}^2+\|\boldsymbol{\mathfrak{b}}\|_{H^{k+\frac{1}{2}}(\partial\Omega)}^2).
    \end{equation}
\end{theorem}

In addition, we recall the following two elementary theorems.

\begin{lemma}[Lax-Milgram theorem]
   Let $H$ be a Hilbert space, if $B:H\times H\rightarrow\mathbb{R}$ is a bilinear mapping, for which there exist constants $\alpha,\beta>0$ such that for all $u,v\in H$,
  \begin{equation*}
    |B(u,v)|\leq\alpha\|u\|_{H}\cdot\|v\|_{H},
  \end{equation*}
  and
  \begin{equation*}
    \beta\|u\|_{H}^{2}\leq B(u,u).
  \end{equation*}
  Then for a bounded linear functional $f:H\rightarrow\mathbb{R}$, there exists a unique element $u\in H$ such that
  \begin{equation*}
    B(u,v)=f(v),
  \end{equation*}
  for all $v\in H$.
\end{lemma}

\begin{lemma}[Schauder's fixed point theorem]
  Let $X$ be a Banach space and $M\subset X$ be a nonempty convex compact subset. If $T: M\rightarrow M$ is a continuous mapping, then $T$ has a fixed point in $M$.
\end{lemma}


\medskip
\noindent {\bf Acknowledgment:}\,
The research of Renjun Duan was partially supported by the General Research Fund (Project No.~14303523) from RGC of Hong Kong and a Direct Grant from CUHK.

%
%
%
%

\end{document}